\documentclass[11pt]{article}

\usepackage{ifpdf}
\usepackage{graphicx, lineno}
\usepackage{verbatim}
\usepackage{graphicx}
\usepackage{amsmath,amssymb,amsfonts,euscript,enumerate,latexsym}
\usepackage{amsthm}
\usepackage{color,wrapfig}
\usepackage{pxfonts}
\usepackage{mathrsfs}
\usepackage{mathtools}
\usepackage{epsfig,subfigure}
\usepackage{marvosym}
\usepackage{txfonts}
\usepackage{hyperref}
\usepackage{times}
\usepackage{enumerate}
\usepackage{ulem}

\def\titlerunning#1{\gdef\titrun{#1}}
\makeatletter
\def\author#1{\gdef\autrun{\def\and{\unskip, }#1}\gdef\@author{#1}}
\def\address#1{{\def\and{\\\hspace*{18pt}}\renewcommand{\thefootnote}{}%
		\footnote {#1}}%
	\markboth{\autrun}{\titrun}}
\makeatother
\def\email#1{e-mail: #1}

\def\keywords#1{\par\medskip
	\noindent\textbf{Keywords.} #1}

\usepackage{etoolbox}
\makeatletter
\patchcmd{\@maketitle}
{\ifx\@empty\@dedicatory}
{\ifx\@empty\@date \else {\vskip3ex \centering\footnotesize\@date\par\vskip1ex}\fi
	\ifx\@empty\@dedicatory}
{}{}
\patchcmd{\@adminfootnotes}
{\ifx\@empty\@date\else \@footnotetext{\@setdate}\fi}
{}{}{}
\makeatother

\newtheorem{thm}{Theorem}[section]
\newtheorem{cor}[thm]{Corollary}
\newtheorem{lemma}[thm]{Lemma}

\newtheorem{remark}[thm]{Remark}

\newcommand{\R}{{\mathbb{R}}}

\newcommand{\N}{{\mathbb{N}}}

\newcommand{\esssup}{{\mathrm{ess}\sup}}
\newcommand{\essinf}{{\mathrm{ess}\inf}}

\newcommand{\vp}{\varphi}
\newcommand{\osc}{\operatornamewithlimits{osc}}

\newcommand{\D}{\nabla}

\newcommand{\La}{\triangle}
\newcommand{\bs}{\backslash}

\begin{document}

\baselineskip=16pt

\titlerunning{}

\title{Properties of Generalized Degenerate Parabolic Systems}
	
\author{ Sunghoon Kim  
	\and Ki-Ahm Lee }

\date{}

\maketitle

\address{ Sunghoon Kim (\Letter) :
	Department of Mathematics, The Catholic University of Korea,\\
	43 Jibong-ro, Bucheon-si, Gyeonggi-do, 14662, Republic of Korea \\
	\email{math.s.kim@catholic.ac.kr}
	\and
	Ki-Ahm Lee :
	Department of Mathematical Sciences, Seoul National University, Gwanak-ro 1, Gwanak-Gu, Seoul, 08826, Republic of Korea \\
	\& Korea Institute for Advanced Study, Seoul 02455, Republic of Korea \\
	\email{ kiahm@snu.ac.kr }
}

\begin{abstract}
In this paper, we consider the solution $\bold{u}=\left(u^1,\cdots,u^k\right)$ of the generalized parabolic system
\begin{equation*}
\left(u^i\right)_t=\nabla\cdot\left(mU^{m-1}\mathcal{A}\left(\nabla u^i,u^i,x,t\right)+\mathcal{B}\left(u^i,x,t\right)\right), \qquad \left(1\leq i\leq k\right)
\end{equation*}
in the range of exponents $m>\frac{n-2}{n}$ where the diffusion coefficient $U$ depends on the components of the solution $\bold{u}$. Under suitable structure conditions on the vector fields $\mathcal{A}$ and $\mathcal{B}$, we first show the uniform $L^{\infty}$ bound of the function $U$ for $t\geq \tau>0$ and law of $L^1$ mass conservation of each component $u^i$, $(i=1,\cdots,k)$, with system version of Harnack type inequality. As the last result, we also deal with the local continuity of solution $\bold{u}=\left(u^1,\cdots,u^k\right)$ with the intrinsic scaling. If the ratio between $U$ and components $u^i$, $(i=1,\cdots,k)$, is uniformly bounded above and below, all components of the solution $\bold{u}$ have the same modulus of continuity.
\end{abstract}


\setcounter{equation}{0}
\setcounter{section}{0}

\section{Introduction and Main Results}\label{section-intro}
\setcounter{equation}{0}
\setcounter{thm}{0}

\indent Let $n\geq 2$ and consider a closed system $\left(\subset \R^n\right)$ in which various species exist. Let $k\in\N$ be the number of different species in that system and let $u^i$, $\left(1\leq i\leq k\right)$ represent the population density of $i$-th species. Since the system is closed, the diffusion of population of each species will be governed by some quantity depending only on the populations of all species in the system, i.e., if we denote by $U$ the diffusion coefficients of the system then $U$ will be expressed by
\begin{equation*}
U=U\left(u^1,\cdots,u^k\right).
\end{equation*}
Next two parabolic systems would be the considerable mathematical modellings which formulate the evolution of population density of each species in a closed system:
\begin{equation}\label{eq-one-simplest-system-with-absolute-of-solution-as-diffusion}
\left(u^i\right)_t=\nabla\cdot\left(U^{m-1}\nabla u^i\right)=\nabla\cdot\left(\left(\sum_{i=1}^ku^i\right)^{m-1}\nabla u^i\right),\qquad i=1,\cdots,k
\end{equation}
where $U$ depends on the total population $\sum_{i=1}^ku^i$, and
\begin{equation}\label{eq-one-simplest-system-with-absolute-of-solution-as-diffusion-2}
\left(u^i\right)_t=\nabla\cdot\left(U^{m-1}\nabla u^i\right)=\nabla\cdot\left(\left(\sum_{i=1}^{k}\left(u^i\right)^2\right)^{\frac{m-1}{2}}\nabla u^i\right)=\nabla\cdot\left(\left|\bold{u}\right|^{m-1}\nabla u^i\right),\qquad i=1,\cdots,k
\end{equation}
where $U$ depends on the size of the vector $\bold{u}=\left(u^1,\cdots,u^k\right)$.\\
\indent As a generalized version of \eqref{eq-one-simplest-system-with-absolute-of-solution-as-diffusion} and \eqref{eq-one-simplest-system-with-absolute-of-solution-as-diffusion-2}, we are going to consider the generalized parabolic system in this paper. More precisely, let $\bold{u}=\left(u^1,\cdots,u^k\right)$ be a solution of the parabolic system
\begin{equation}\label{eq-standard-form-for-local-continuity-estimates-intro-1}\tag{GPS}
\begin{aligned}
\left(u^i\right)_t=\nabla\cdot\left(mU^{m-1}\mathcal{A}\left(\nabla u^i,u^i,x,t\right)+\mathcal{B}\left(u^i,x,t\right)\right) \qquad \left(1\leq i\leq k\right), 
\end{aligned}
\end{equation}
with the conditions 
\begin{equation}\label{eq-standard-form-for-local-continuity-estimates-intro-2}
\begin{aligned}
u^i\geq 0\qquad \mbox{and} \qquad 0\leq \lambda_i \left(u^i\right)^{\,\beta_i}\leq U=U\left(u^1,\cdots,u^k\right) \qquad \forall 1\leq i\leq k
\end{aligned}
\end{equation}
in $\R^n\times\left(0,\infty\right)$ where constants $m$, $\lambda_i$ and $\beta_i$ are  such that
\begin{equation*}
m>\frac{n-2}{n}, \qquad \mbox{and}\qquad \lambda_i>0,\,\,\beta_i\geq 0.
\end{equation*}
In \eqref{eq-standard-form-for-local-continuity-estimates-intro-2}, the first condition comes from the nonnegativity of each species and the second one represents the relationship between diffusion coefficient and the population of each species.\\
\indent The aim of this paper is to provide the regularity theory of the diffusion coefficient $U^{m-1}$ and components of the solution $\bold{u}=\left(u^1,\cdots,u^k\right)$ of the system \eqref{eq-standard-form-for-local-continuity-estimates-intro-1} when the vectors $\mathcal{A}$ and $\mathcal{B}$ are assumed to be measurable in $(x,t)\in\R^n\times[0,\infty)$ and continuous with respect to $u$ and $\nabla u$ for almost all $(x,t)$. \\
\indent To deal with effects from diffusion in the energy type inequality for \eqref{eq-standard-form-for-local-continuity-estimates-intro-1}, suitable structural assumptions are needed to be imposed on the function $U$ and vector fields $\mathcal{A}$ and $\mathcal{B}$. In this point of view, we consider that the function $U=U\left(u^1,\cdots,u^k\right)$ satisfies the conditions
\begin{equation}\label{eq-condition-for-U-primary--1}
	U^{m}U_{u^i}\in H_0^1\left(\R^n\times(0,\infty)\right)\qquad \forall 1\leq i\leq k \tag{A1}
\end{equation}
where $U_{u^i}$ be the derivative of $U$ with respect to $u^i$ and
\begin{equation}\label{eq-condition-for-U-primary-0}
U(0,\cdots,0)=0,\qquad \sum_{i=1}^{k}\left|U_{\xi^i}\left(\xi^1,\cdots,\xi^k\right)\xi^i\right|\leq \textbf{C}_1U,\qquad U_{\xi^i}\left(\xi^1,\cdots,\xi^k\right)\geq 0\qquad \forall \xi=\left(\xi^1,\cdots,\xi^k\right)\in\R^k. \tag{A2}
\end{equation}
The vectors $\mathcal{A}\left(p,z,x,t\right)$ and $\mathcal{B}\left(z,x,t\right)$, $\left(\left(z,p\right)\in\R^+\times\R^n\right)$, are assumed to have the following structures 
\begin{align}
&\sum_{i=1}^{k}\left(mU^{m-1}\mathcal{A}\left(\nabla u^i,u^i,x,t\right)+\mathcal{B}\left(u^i,x,t\right)\right)\nabla U_{u^i}\geq 0,\tag{A3}\label{eq-condition-for-measurable-functions-mathcal-A-bounded-below-and-above}\\
&\nabla U\cdot\left[\sum_{i=1}^{k}U_{u^i}\mathcal{A}\left(\nabla u^i,u^i,x,t\right)\right]\geq \textbf{c}\left|\nabla U\right|^2-\textbf{C}_2U^2\tag{A4}\label{eq-condition-1-for-measurable-functions-mathcal-A-and-mathcal-B}
\end{align}
and, for any positive function $u\geq 0$
\begin{align}
\mathcal{A}\left(\nabla u,u,x,t\right)\cdot\nabla u&\geq \textbf{c}\left|\nabla u\right|^2-\textbf{C}_2u^2,\tag{A5}\label{eq-condition-1-01-for-measurable-functions-mathcal-A-and-mathcal-B}\\
\left|\mathcal{A}\left(\nabla u,u,x,t\right)\right|&\leq \textbf{C}_3\left|\nabla u\right|+\textbf{C}_4u,\tag{A6}\label{eq-condition-1-02-for-measurable-functions-mathcal-A-and-mathcal-B}\\
\left|\mathcal{B}\left(u,x,t\right)\right|&\leq \textbf{C}_5u^q\tag{A7}\label{eq-condition-3-for-measurable-functions-mathcal-A-and-mathcal-B}
\end{align}
where  
\begin{equation*}
0<\textbf{c}\leq 1\leq \textbf{C}_1,\, \textbf{C}_2,\, \textbf{C}_3,\, \textbf{C}_4,\, \textbf{C}_5,\,<\infty 
\end{equation*}
 and
\begin{equation}\label{eq-range-of-constant-q-in-forcing-term}
1<q<\left(m\left(1+\frac{1+m}{mn}\right)-1\right)\min_{1\leq i\leq k}\beta_i+1=\left(m\left(1+\frac{1+m}{mn}\right)-1\right)\beta_{\ast}+1:=m\left(\beta_{\ast}\right).
\end{equation}
\begin{remark}
	\begin{itemize}
		\item Let $k=1$, $\mathcal{A}=u^i$ and $\mathcal{B}=0$. Then the system \eqref{eq-standard-form-for-local-continuity-estimates-intro-1} is called the porous medium equation (or slow diffusion equation) for $m>1$, heat equation for $m=1$ and fast diffusion equation for $m<1$.
		\item $\frac{n-2}{n}$ is the critical number of porous medium equation and $m>\frac{n-2}{n}$ in the standard porous medium equation gives the conservation law of $L^1$ mass.
		\item Condition \eqref{eq-condition-for-U-primary-0} is a natural growth condition and the monotonicity of $U$ with respect to $u^i$, $\left(1\leq i\leq k\right)$.
		\item Condition \eqref{eq-condition-1-01-for-measurable-functions-mathcal-A-and-mathcal-B} is coercivity condition to get parabolicity and Conditions \eqref{eq-condition-1-02-for-measurable-functions-mathcal-A-and-mathcal-B} and \eqref{eq-condition-3-for-measurable-functions-mathcal-A-and-mathcal-B} are growth conditions. We refer the reader to the papers \cite{DGV1}, \cite{DGV2}, \cite{FDV} for the structure conditions of parabolic partial differential equations.
		\item If $m=1$ and $\beta_{\ast}=1$ , then the constant $m\left(\beta_{\ast}\right)$ in the condition \eqref{eq-range-of-constant-q-in-forcing-term} will be $\frac{n+2}{n}$, which becomes critical number for the standard heat equation in the energy estimates.
	\end{itemize}
\end{remark}
Let $\Omega$ be an open set in $\R^n$, and for $T>0$ let $\Omega_T$ denote the parabolic domain $\Omega\times\left(0,T\right]$. We say that $\bold{u}=\left(u^1,\cdots,u^k\right)$ is a weak (energy) solution of \eqref{eq-standard-form-for-local-continuity-estimates-intro-1} in $\Omega_T$ if the component $u^i$, $\left(1\leq i\leq k\right)$, is a locally integrable function satisfying
	\begin{enumerate}
		\item $u^i$ belongs to function space:
		\begin{equation}\label{eq-first-condition-of-weak-soluiton-u-with-U}
			U^{m-1}\left|\mathcal{A}\left(\nabla u^i,u^i,x,t\right)\right|\in L^2\left(0,T:L^2\left(\Omega\right)\right).  
		\end{equation}
		\item $u^i$ satisfies the identity: 
		\begin{equation}\label{eq-identity--of-formula-for-weak-solution}
			\int_{0}^{T}\int_{\Omega}\left\{m\,U^{m-1}\mathcal{A}\left(\nabla u^i,u^i,x,t\right)\cdot\nabla\vp+\mathcal{B}\left(u^i,x,t\right)\cdot\nabla\vp+u^i\vp_t\right\}\,dxdt=0
		\end{equation}
		holds for any test function $\vp\in H^{1}\left(0,T:L^2\left(\Omega\right)\right)\cap L^2\left(0,T:H^1_0\left(\Omega\right)\right)$.
\end{enumerate}

\indent Parabolic systems \eqref{eq-one-simplest-system-with-absolute-of-solution-as-diffusion} and \eqref{eq-one-simplest-system-with-absolute-of-solution-as-diffusion-2} are ones of the simplest examples of the parabolic systems in the form of \eqref{eq-standard-form-for-local-continuity-estimates-intro-1} which satisfies the conditions \eqref{eq-standard-form-for-local-continuity-estimates-intro-2} and \eqref{eq-condition-for-U-primary--1}-\eqref{eq-condition-3-for-measurable-functions-mathcal-A-and-mathcal-B}. Mathematical theories of the systems \eqref{eq-one-simplest-system-with-absolute-of-solution-as-diffusion} and \eqref{eq-one-simplest-system-with-absolute-of-solution-as-diffusion-2} with the range of exponents $m>1$ were investigated by
S. Kim and Ki-Ahm Lee. They considered the local H\"older continuity and asymptotic large time behaviour of the parabolic systems \eqref{eq-one-simplest-system-with-absolute-of-solution-as-diffusion} and \eqref{eq-one-simplest-system-with-absolute-of-solution-as-diffusion-2} in \cite{KL2} and \cite{KL3}. \\
\indent In \cite{KMV}, authors proved the existence of a unique weak solution and derived regularity estimates of the parabolic system \eqref{eq-one-simplest-system-with-absolute-of-solution-as-diffusion} under some assumptions. \\
\indent Compared to the divergence type parabolic system, there are also some studies on the non-divergence type parabolic system. We refer the reader to the paper \cite{ST} for the boundedness on the degenerate parabolic system
\begin{equation}\label{eq-system-non-divergence}
\bold{u}_t=\La\left(\left|\bold{u}\right|^{m-1}\bold{u}\right), \qquad m>1,\,\,\,\bold{u}=\left(u^1,\cdots,u^k\right).
\end{equation}
In \cite{ST}, they studied the sharp estimate of $\left\|\left|\bold{u}\right|\right\|_{L^{\infty}}$ for the decay in time and finite speed of propagation of the solution $\bold{u}$.\\
\indent As the time evolves, the solutions of parabolic systems lose the information given by the initial data, and diffuse only under the laws governed by the systems. Hence, the evolution of the solutions are determined by the diffusion coefficients and external forces after the large time. Therefore, it is crucial to understand the influence of the diffusion coefficients for long time behaviour of solutions.\\
\indent Suppose that $m>1$. Then, by \eqref{eq-standard-form-for-local-continuity-estimates-intro-2} the diffusion coefficient $U^{m-1}$ is bounded from below by $\lambda_i^{m-1}\left(u^i\right)^{\beta_i(m-1)}$. Thus, we need to control the diffusion coefficient from above to improve the regularity theories of \eqref{eq-standard-form-for-local-continuity-estimates-intro-1}. The first part of the paper is about a priori estimates of the function $U$ which determines the diffusion coefficient of the parabolic system \eqref{eq-standard-form-for-local-continuity-estimates-intro-1}. The statement is as follow.
\begin{thm}[\textbf{Uniform $L^{\infty}$ boundedness of $U$}]\label{thm-Main-boundedness-of-diffusion-coefficients-of-generaized-equation}
Let $n\geq 2$, $m>\frac{n-2}{n}$. Let $\bold{u}=\left(u^1,\cdots,u^k\right)$ be a solution of \eqref{eq-standard-form-for-local-continuity-estimates-intro-1} satisfying conditions \eqref{eq-condition-for-measurable-functions-mathcal-A-bounded-below-and-above}-\eqref{eq-condition-3-for-measurable-functions-mathcal-A-and-mathcal-B} where the constant $q$ is given by \eqref{eq-range-of-constant-q-in-forcing-term}. If $U=U\left(u^1,\cdots,u^k\right)$ is a function such that \eqref{eq-condition-for-U-primary--1}, \eqref{eq-condition-for-U-primary-0} hold,
then for a small $t_0>0$ there exists a constant $\textbf{K}(t_0)>0$ such that
\begin{equation}\label{conclusion-L-infty-boundedness-for-time-T-geq-tau-strictly-positive}
\sup_{x\in\R^n,\,t\geq t_0}\left|U\left(x,t\right)\right|\leq \textbf{K}\left(t_0\right).
\end{equation}
Moreover, if the constant $\textbf{C}_5$ in the structure \eqref{eq-condition-3-for-measurable-functions-mathcal-A-and-mathcal-B} is zero, then the range of exponents $m$ is independent of $q$, i.e., \eqref{conclusion-L-infty-boundedness-for-time-T-geq-tau-strictly-positive} holds for all $m>\frac{n-2}{n}$.
\end{thm}
\begin{remark}\label{remark-for-upperbound-textbf-K-t-0-1}
	\begin{itemize}
 The bound $\textbf{K}\left(t_0\right)$ blows up as $t_0\to 0$. In the scalar case, it is well known that $\textbf{K}\left(t_0\right)$ is universal and optimal for the Barenblatt solution.
	\end{itemize}
\end{remark}
The concept of $L^1$ mass conservation appears widely in many fields, such as mechanics and fluid dynamics. In the theory of partial differential equations, it also plays an important role in the study of the asymptotic behaviour of solutions. Although the law of mass conservation is considered as a part of assumptions in various studies based on the classical mechanics, but in general, mass is not always preserved in systems. Thus, it is very important to check that a solution or a component of a solution maintains its mass for all time by the rules of the parabolic system. \\
\indent In the second part of this paper, we are going to show that the $L^1$ mass conservation of each component $u^i$, $(i=1,\cdots,k)$, is preserved under structure assumptions \eqref{eq-condition-for-U-primary--1}-\eqref{eq-condition-3-for-measurable-functions-mathcal-A-and-mathcal-B}. \\
\indent In the study of parabolic equations, this result is well known in the case where the constants $\textbf{C}_2$, $\textbf{C}_4$, $\textbf{C}_5$ in the structures \eqref{eq-condition-1-for-measurable-functions-mathcal-A-and-mathcal-B}-\eqref{eq-condition-3-for-measurable-functions-mathcal-A-and-mathcal-B} are all zeros or $m=1$, but it is non-trivial for general constants $\textbf{C}_2$, $\textbf{C}_4$, $\textbf{C}_5$ and $m$. We refer the reader to the paper \cite{FDV} for the mass conservation of solutions to a class of singular parabolic equations. \\
\indent As an improvement of previous papers on the mass conservation, we only assume that the constant $\textbf{C}_5$ in the structure \eqref{eq-condition-3-for-measurable-functions-mathcal-A-and-mathcal-B} is zero (i.e., the vector $\mathcal{B}=0$) in the second part of this paper. The first result of the second part is the $L^1$ mass conservation of component $u^i$, $(i=1,\cdots,k)$, in nondegenerate case.
\begin{thm}[Mass Conservation of Components in Nondegenerate Case]\label{thm-L-1-mass-conservation-nondegenerate-range}
	Let $n\geq 2$ and $m=1$. For each $1\leq i\leq k$, let $u^i_0$ be a positive function in $\R^n$. Let $\bold{u}=\left(u^1,\cdots,u^k\right)$ be a weak solution to the parabolic system \eqref{eq-standard-form-for-local-continuity-estimates-intro-1} with initial data $\bold{u}_0=\left(u^1_0,\cdots,u^k_0\right)$ and conditions \eqref{eq-standard-form-for-local-continuity-estimates-intro-2}, \eqref{eq-condition-for-U-primary--1}-\eqref{eq-condition-3-for-measurable-functions-mathcal-A-and-mathcal-B} in $\R^n\times[0,\infty)$. Suppose that
	\begin{equation*}
	u_0^i\in L^1\left(\R^n\right)\cap L^p\left(\R^n\right)
	\end{equation*}
	for some constant $p>1$ and $\textbf{C}_5=0$ in the structure \eqref{eq-condition-3-for-measurable-functions-mathcal-A-and-mathcal-B}. Then for any $t>0$
	\begin{equation}\label{eq-in-thm-for-L-1-mass-conservation-in-nondegenerate-case}
	\int_{\R^n}u^i(x,t)\,dx=\int_{\R^n}u^i_0(x)\,dx.
	\end{equation}
\end{thm}

\indent Let $m>1$ and consider the special case $\mathcal{A}(\nabla u^i,u^i,x,t)=\nabla u^i$. By \eqref{eq-standard-form-for-local-continuity-estimates-intro-2}, 
\begin{equation*}\label{eq-compare-between-U-and-u-i-when-m->-1-degenerate}
\lambda_i ^{m-1}\left(u^i\right)^{\,\beta_i(m-1)}\leq U^{m-1}\qquad \mbox{if $m>1$}.
\end{equation*}
Thus, the diffusion of the parabolic system \eqref{eq-standard-form-for-local-continuity-estimates-intro-1} is  deteremined by the function $U$ rather than the component $u^i$ when $m>1$. If 
\begin{equation*}
u^i=\left(\frac{U}{\lambda_i}\right) ^{\frac{1}{\beta_i}},
\end{equation*} 
then the parabolic system \eqref{eq-standard-form-for-local-continuity-estimates-intro-1} can be expressed in the form of 
\begin{equation*}
\left(u^i\right)_t=\nabla \cdot\left(U^{m-1}\mathcal{A}\left(\nabla u^i,u^i,x,t\right)\right)=\nabla\cdot\left(U^{m-1}\nabla u^i\right)=\frac{1}{\beta_i\left(m-1+\frac{1}{\beta_i}\right)\lambda_i^{\frac{1}{\beta_i}}}\La U^{m-1+\frac{1}{\beta_i}}.
\end{equation*}
This becomes degenerate when $U=0$ in the range of exponents $m-1+\frac{1}{\beta_i}>1$. Hence we need to keep the both ranges 
\begin{equation*}
m-1+\frac{1}{\beta_i}>1\qquad \mbox{and} \qquad m>1
\end{equation*}
in mind for the study of degenerate parabolic system \eqref{eq-standard-form-for-local-continuity-estimates-intro-1}. Under this assumption, we are going to state the second result of the second part: $L^1$ mass conservation of components to the degenerate parabolic system.
\begin{thm}[Mass Conservation of Components in Degenerate Range]\label{thm-mass-conservation-of-component-general-degeneraate}
	Let $n\geq 2$ and $m>1$. For each $1\leq i\leq k$, let $u^i_0$ be a positive function. Let $\bold{u}=\left(u^1,\cdots,u^k\right)$ be a weak solution to the Degenerate Parabolic System \eqref{eq-standard-form-for-local-continuity-estimates-intro-1} initial data $\bold{u}_0=\left(u^1_0,\cdots,u^k_0\right)$ and with conditions \eqref{eq-standard-form-for-local-continuity-estimates-intro-2}, \eqref{eq-condition-for-U-primary--1}-\eqref{eq-condition-3-for-measurable-functions-mathcal-A-and-mathcal-B} in $\R^n\times[0,\infty)$. Suppose that $\textbf{C}_5=0$ in the structure \eqref{eq-condition-3-for-measurable-functions-mathcal-A-and-mathcal-B} and
		\begin{equation}\label{eq-L-p-conditions-of-U-for-mass-conservation}
		m>\max\left(1,2-\frac{1}{\beta_i}\right)	\qquad \mbox{and}\qquad 	U_0\in L^{\frac{1}{\beta_i}}\left(\R^n\right)\cap L^{m-1+\frac{1}{\beta_i}}\left(\R^n\right)
		\end{equation}
		and
		\begin{equation}\label{eq-condition-of-upper-bound-M-t}
		\int_{0}^{t}\textbf{K}^{\,m-1}(\tau)\,d\tau \to 0 \qquad \mbox{as $t\to 0$}
		\end{equation}
		where $\textbf{K}(t)$ is given by Theorem \ref{thm-Main-boundedness-of-diffusion-coefficients-of-generaized-equation}.
	 Then for any $t>0$
	\begin{equation*}
	\int_{\R^n}u^i\left(x,t\right)\,dx=\int_{\R^n}u_0^i\left(x\right)\,dx.
	\end{equation*}
\end{thm}
\begin{remark}
As mentioned in Remark \ref{remark-for-upperbound-textbf-K-t-0-1}, $\textbf{K}\left(t_0\right)\to\infty$ as $t_0\to 0$ and the blow-up rate is optimal for the Barenblatt solution when $k=1$. Thus the assumption \eqref{eq-condition-of-upper-bound-M-t} is natural for the scalar porous medium equation.
	\end{remark}
The proof of Theorem \ref{thm-mass-conservation-of-component-general-degeneraate} is based on suitable energy type estimates. Global boundedness estimates (Theorem \ref{thm-Main-boundedness-of-diffusion-coefficients-of-generaized-equation}) and condition \eqref{eq-condition-of-upper-bound-M-t} play important roles on controlling the remainder term of the weak energy inequality, such as $\left\|U\right\|_{L^p}$ for some $p>1$. \\
\indent In the singular case $(m<1)$, the diffusion coefficient of the parabolic system \eqref{eq-standard-form-for-local-continuity-estimates-intro-1} is bounded from above by $\lambda_i ^{m-1}\left(u^i\right)^{\,\beta_i(m-1)} $ i.e., 
\begin{equation*}\label{eq-compare-between-U-and-u-i-when-m-<-1-singular}
U^{m-1}\leq\lambda_i ^{m-1}\left(u^i\right)^{\,\beta_i(m-1)}
\end{equation*}
since $m-1<0$. Hence, the main controller of the diffusion coefficient of parabolic system \eqref{eq-standard-form-for-local-continuity-estimates-intro-1} will be the component $u^i$ itself in the singular case. If $\mathcal{A}(\nabla u^i,u^i,x,t)=\nabla u^i$ and
\begin{equation*}
U=\lambda_i\left(u^i\right)^{\beta_i},
\end{equation*} 
then we can have
\begin{equation*}
\left(u^i\right)_t=\nabla \cdot\left(U^{m-1}\mathcal{A}\left(\nabla u^i,u^i,x,t\right)\right)=\nabla\cdot\left(U^{m-1}\nabla u^i\right)=\frac{\lambda_i^{m-1}}{\beta_i(m-1)+1} \La \left(u^i\right)^{\beta_i(m-1)+1}=\frac{\lambda_i^{m-1}}{m_i} \La \left(u^i\right)^{m_i}
\end{equation*}
where $m_i=\beta_i(m-1)+1$. This becomes singular when $u^i=0$ in the range of exponents $\beta_i(m-1)+1=m_i<1$. Hence we have to consider the range $m_i<1$ as well as the range $m<1$ for the study of singular parabolic system \eqref{eq-standard-form-for-local-continuity-estimates-intro-1}. The last result of the second part is as follows.

\begin{thm}[$L^1$ Mass Conservation of Components in Supercritical range]\label{thm-L-1-mass-conservation-super-critical-range}
	Let $n\geq 2$, $\textbf{C}_5=0$ in the structure \eqref{eq-condition-3-for-measurable-functions-mathcal-A-and-mathcal-B} and let $0<m<1$ satisfy
	\begin{equation}\label{reanage-of-m-for-L-1-mass-conservation}
	 \frac{n^2+n+4+\sqrt{2n(7n+11)}}{n^2+5n+8}<m_i=\beta_i\left(m-1\right)+1<1.
	\end{equation}
	For each $1\leq i\leq k$, let $u^i_0$ be a positive, integrable function with compact support in $\R^n$. Let $\bold{u}=\left(u^1,\cdots,u^k\right)$ be a weak solution to the Singular Parabolic System \eqref{eq-standard-form-for-local-continuity-estimates-intro-1} with initial data $\bold{u}_0=\left(u^1_0,\cdots,u^k_0\right)$ and with conditions \eqref{eq-standard-form-for-local-continuity-estimates-intro-2}, \eqref{eq-condition-for-U-primary--1}-\eqref{eq-condition-3-for-measurable-functions-mathcal-A-and-mathcal-B} in $\R^n\times[0,\infty)$. Suppose that there exist constants $R^{\ast}$ and $\Lambda$ such that
		\begin{equation*}
		U(x,t)\leq \Lambda \qquad \mbox{a.e. on $\left\{|x|\geq R^{\ast},\,t\geq 0\right\}$}.
		\end{equation*}
	 Then for any $t>0$
	\begin{equation}\label{eq-in-thm-for-L-1-mass-conservation-in-super-critical-case}
	\int_{\R^n}u^i(x,t)\,dx=\int_{\R^n}u^i_0(x)\,dx.
	\end{equation}
	Moreover, if the constants $\textbf{C}_2$ and $\textbf{C}_4$ in the structures \eqref{eq-condition-1-for-measurable-functions-mathcal-A-and-mathcal-B}-\eqref{eq-condition-1-02-for-measurable-functions-mathcal-A-and-mathcal-B} are all zeros, then $L^1$-mass conservation \eqref{eq-in-thm-for-L-1-mass-conservation-in-super-critical-case} holds for $\frac{n^2-n+3+\sqrt{7n^2+2n-7}}{n^2+2n+4}<m_i=\beta_i(m-1)+1<1$.
\end{thm}
\begin{remark}
	 By the diffusion coefficient $U\left(u^1,\cdots,u^k\right)$,  each component of the system \eqref{eq-standard-form-for-local-continuity-estimates-intro-1} could not evolve in a way determined by the component itself, i.e., they interact each other. Thus, in order for a component of the system to have the law of $L^1$-mass conservation, the influence of other components must be controllable, i.e., the lower bound of the exponents $m$ in the diffusion coefficients must be more restrictive. The interval \eqref{reanage-of-m-for-L-1-mass-conservation} gives the minimal range that the exponents $m$ must have for the existence of $L^1$ mass conservation of each component in the nonhomogeneous parabolic system.
\end{remark}
By structure condition  \eqref{eq-condition-for-U-primary-0} and the assumption of $u^i_0$, $(1\leq i\leq k)$, in Theorem \ref{thm-L-1-mass-conservation-super-critical-range}, 
\begin{equation*}
U(x,0)=0, \qquad \forall |x|>>1,
\end{equation*}
i.e., $U$ is bounded above at $t=0$ on the region far away from $0$, and This condition, "boundedness of $U$ far away from $0$", persists for a short time $t_0>0$. By Theorem \ref{thm-Main-boundedness-of-diffusion-coefficients-of-generaized-equation}, boundedness of $U(t)$ is also obtained when $t\geq t_0$ if the exponents $m$ satisfy
\begin{equation*}\label{eq-range-of-expoenntial-range-m-upper-of-super-critical}
m>\frac{n-2}{n}.
\end{equation*}
From these two observations, we can get rid of the condition "Boundedness of $U$" from Theorem \ref{thm-L-1-mass-conservation-super-critical-range} if the constant $m$ is sufficiently close to $1$. Therefore, we can guarantee the law of $L^1$ mass conservation for some special $m$ in the super critical range without the condition "Boundedness".
\begin{cor}
	Let $n\geq 2$, $\textbf{C}_5=0$ in the structure \eqref{eq-condition-3-for-measurable-functions-mathcal-A-and-mathcal-B} and 
	\begin{equation*}
	m\in\left(1-\frac{2}{n},1\right)\cap \left(1-\frac{4(n+1)}{\beta_i(n^2+5n+8)}+\frac{\sqrt{2n(7n+11)}}{\beta_i(n^2+5n+8)},1\right).
	\end{equation*} 
	For each $1\leq i\leq k$, let $u^i_0$ be a positive, integrable function with compact support in $\R^n$. Let $\bold{u}=\left(u^1,\cdots,u^k\right)$ be a weak solution to the Singular Parabolic System \eqref{eq-standard-form-for-local-continuity-estimates-intro-1} with initial data $\bold{u}_0=\left(u^1_0,\cdots,u^k_0\right)$ and with conditions \eqref{eq-standard-form-for-local-continuity-estimates-intro-2}, \eqref{eq-condition-for-U-primary--1}-\eqref{eq-condition-3-for-measurable-functions-mathcal-A-and-mathcal-B} in $\R^n\times[0,\infty)$. Then for any $t>0$
	\begin{equation}\label{eq-in-thm-for-L-1-mass-conservation-in-super-critical-case-02}
		\int_{\R^n}u^i(x,t)\,dx=\int_{\R^n}u^i_0(x)\,dx.
	\end{equation}
	If the constants $\textbf{C}_2$ and $\textbf{C}_4$ in the structures \eqref{eq-condition-1-for-measurable-functions-mathcal-A-and-mathcal-B}-\eqref{eq-condition-1-02-for-measurable-functions-mathcal-A-and-mathcal-B} are all zeros, then $L^1$-mass conservation \eqref{eq-in-thm-for-L-1-mass-conservation-in-super-critical-case-02} holds for
	\begin{equation*}
	m\in\left(1-\frac{2}{n},\,1\right)\cap\left(1-\frac{3n+1}{\beta_i\left(n^2+2n+4\right)}+\frac{\sqrt{7n^2+2n-7}}{\beta_i\left(n^2+2n+4\right)},\,1\right).
	\end{equation*}
\end{cor}
	Contrary to the degenerate case, it is very hard to get energy inequality and boundedness estimates in the singular case. The proofs for mass conservation in the singular case are based on a system version of integral Harnack estimate which controls the speed of the propagation of the solution $\bold{u}=\left(u^1,\cdots,u^k\right)$. We remark that Harnack type estimates are known for the general parabolic equations but, as far as we know, not for general parabolic systems.\\
\indent At the end of this paper, we are going to give an explanation about the local continuity of the weak solution $\bold{u}=\left(u^1,\cdots,u^k\right)$ to the parabolic system \eqref{eq-standard-form-for-local-continuity-estimates-intro-1} by showing that the differences between supremum and infimum of components $u^i$, $\left(1\leq i\leq k\right)$, on a chosen set decrease as the radius of the set shrinks.\\
\indent If the diffusion coefficient $U^{m-1}$ is uniformly parabolic, then De Giorgi and Moser's technique \cite{De}, \cite{Mo} on regularity theory for uniformly elliptic and parabolic PDE's are enough to show the local continuity of the solution $\bold{u}$. Otherwise, we need to take care of the difficulties coming from the diffusion coefficient and ratio between the diffusion coefficient and components $u^i$, $\left(1\leq i\leq k\right)$. To overcome it, we use the well known technique called intrinsic scaling whose parameters are determined by the size of oscillation of $u^i$, $\left(1\leq i\leq k\right)$. The statement of the last result is as follow.

\begin{thm}\label{eq-local-continuity-of-solution}
Let $n\geq 2$ and $m>1$. Suppose that
\begin{equation}\label{eq-condition-fo-q-for-local-continuity-in-Thm}
\textbf{C}_5=0 \qquad \mbox{or} \qquad \frac{1}{2}\left(m-1\right)\min_{1\leq i\leq k}\beta_i+1<q<\left(m\left(1+\frac{1+m}{mn}\right)-1\right)\min_{1\leq i\leq k}\beta_i+1
\end{equation} 
 in the structure condition \eqref{eq-range-of-constant-q-in-forcing-term}. Any weak solution of degenerate parabolic system \eqref{eq-standard-form-for-local-continuity-estimates-intro-1} is locally continuous in $\R^n\times\left(0,\infty\right)$.
\end{thm}
\begin{remark}
The local continuity of Theorem \ref{eq-local-continuity-of-solution} can be extended to the fast diffusion type system, i.e., Theorem \ref{eq-local-continuity-of-solution} also holds for the range of exponents $0<m<1$ (See Remarks \ref{remark-first-alternative-for-fast-diffusion-type-system} and \ref{remark-explain-extension-to-fde-system-on-second-alternative}).
\end{remark}
To take care of the difficulties from $\left(u^i\right)_t$, we introduce the Lebesgue-Steklov average $\left(u^i\right)_h$ of the function $\left(u^i\right)$, for $h>0$:
\begin{equation*}
\left(u^i\right)_h(\cdot,t)=\frac{1}{h}\int_{t}^{t+h}u^i(\cdot,\tau)\,d\tau.
\end{equation*}
$\left(u^i\right)_h$ is well-defined and it converges to $u^i$ as $h\to 0$ in $L^{p}$ for all $p\geq 1$.  In addition, it is differentiable in time for all $h>0$ and its derivative is 
\begin{equation*}
\frac{u^i(t+h)-u^i(t)}{h}.
\end{equation*}
Fix $t\in(0,T)$ and let $h$ be a small positive number such that $0<t<t+h<T$. Then for every compact subset $\mathcal{K}\subset\R^n$ the following formulation is equivalent to \eqref{eq-identity--of-formula-for-weak-solution}:
\begin{equation}\label{eq-formulation-for-weak-solution-of-u-h}
\int_{\mathcal{K}\times\{t\}}\left[\left(\left(u^i\right)_h\right)_t\vp+m\left(U^{m-1}\mathcal{A}\left(\nabla u^i,u^i,x,t\right)\right)_h\cdot\nabla\vp+\left(\mathcal{B}\left(u^i,x,t\right)\right)_h\cdot\nabla\vp\right]\,dx=0, \qquad \forall 0<t<T-h
\end{equation}
for any $\vp\in H^1_0\left(\mathcal{K}\right)$. Thus, from now on, we consider the weak formulation \eqref{eq-identity--of-formula-for-weak-solution} of \eqref{eq-standard-form-for-local-continuity-estimates-intro-1} as the limit of formulation \eqref{eq-formulation-for-weak-solution-of-u-h} with respect to $h$.\\
\indent A brief outline of the paper is as follows. In Section 2, we will prove uniform $L^{\infty}$ boundedness of the function $U$ which is a main controller of the diffusion coefficient. Section 3 is devoted to the proofs of $L^1$ mass conservation in nondegenerate case (Theorem \ref{thm-L-1-mass-conservation-nondegenerate-range}), in degenerate case (Theorem \ref{thm-mass-conservation-of-component-general-degeneraate}) and in singular case (Theorem \ref{thm-L-1-mass-conservation-super-critical-range}). Lastly, we investigate the regularity properties of weak solution of \eqref{eq-standard-form-for-local-continuity-estimates-intro-1} in Section 4. In particular, we deal with the local continuity of components $u^i$, $(1\leq i\leq k)$ with intrinsic scaling technique developed by \cite{Di} and \cite{HU}.

\noindent \textbf{Notations} Before we deal with the main idea of the paper, let us summarize the notations and definitions that will be used.
\begin{itemize}
	\item We denote by $B_{R}(x)$ the ball centered at $x\in\R^n$ of radius $R>0$. We let $B_{R}=B_R(0)$.
	\item $Q\left(R,r\right)=B_R\times\left(-r,0\right)$.
	\item Let $E$ be an open set in $\R^n$. We denote by $E_T$ the parabolic domain $E\times\left(0,T\right]$ for $T>0$.
	\item Numbers: $\lambda_i$, $\beta_i$ are given by \eqref{eq-standard-form-for-local-continuity-estimates-intro-2}, $m_i=\beta_i\left(m-1\right)+1$, and $\textbf{C}_1$, $\cdots$, $\textbf{C}_5$ are given by structure conditions \eqref{eq-condition-for-U-primary--1}-\eqref{eq-condition-3-for-measurable-functions-mathcal-A-and-mathcal-B}.

\end{itemize}

\section{Uniform Boundedness of the function $U$}\label{Uniform Boundedness of the function U}
\setcounter{equation}{0}
\setcounter{thm}{0}

\indent This section is aimed at providing the proof of Theorem \ref{thm-Main-boundedness-of-diffusion-coefficients-of-generaized-equation}. The proof is based on a recurrence relation between a series of truncations of $U$. An energy inequality which deals with $\left|\nabla U^m\right|$ by $U^{1+m}$, and an embedding which controls $U$ by $\left|\nabla U\right|$ play important roles on the proof.  
\begin{proof}[\textbf{Proof of Theorem \ref{thm-Main-boundedness-of-diffusion-coefficients-of-generaized-equation}}]
We will use a modification of the proof of Theorem 1 of \cite{CV}. Let 
\begin{equation}\label{eq-level-equence-L-j-with-mathcal-K}
L_j=\textbf{K}\left(1-\frac{1}{2^{j}}\right) \qquad \mbox{and} \qquad U_j=\left(U-L_j\right)_+
\end{equation}
for a constant $\textbf{K}>2$ which will be determined later. Then
\begin{equation*}\label{eq-lower-bound-of-U-on-U-j-positibve}
U\geq \frac{\textbf{K}}{2}>1 \mbox{on $\left\{U_j> 0\right\}$} \qquad \forall j\in\N.
\end{equation*}
By structures \eqref{eq-condition-for-U-primary--1} and \eqref{eq-condition-for-measurable-functions-mathcal-A-bounded-below-and-above}, we have the following energy type inequality for the truncation $U_j$:
\begin{equation}\label{eq-energy-type-inequality-for-boundedness-of-U--0}
\begin{aligned}
&\frac{1}{1+m}\frac{\partial}{\partial t}\left[\int_{\R^n}U_j^{\,1+m}\,dx\right]+m\int_{\R^n}\sum_{i=1}^{k}\left[\left(mU^{m-1}\,U_{u^i}\,\mathcal{A}\left(\nabla u^i,u^i,x,t\right)+U_{u^i}\,\mathcal{B}\left(u^i,x,t\right)\right)\cdot\nabla U\right]U_j^{m-1}\,dx\\
&\qquad \qquad = \frac{1}{1+m}\frac{\partial}{\partial t}\left[\int_{\R^n}U_j^{\,1+m}\,dx\right]+m^2\int_{\R^n}U^{m-1}U_j^{m-1}\,\left[\left(\sum_{i=1}^{k}U_{u^i}\mathcal{A}\left(\nabla u^i,u^i,x,t\right)\right)\cdot\nabla U\right]\,dx\\
&\qquad \qquad \qquad \qquad +m\int_{\R^n}U_j^{m-1}\left(\sum_{i=1}^{k}U_{u^i}\,\mathcal{B}\left(u^i,x,t\right)\right)\cdot\nabla U\,dx\\
&\qquad \qquad \leq 0.
\end{aligned}
\end{equation}
Let $\underline{\beta}=\min_{1\leq i\leq  k}\beta_i$ and $\underline{\lambda^{\frac{1}{\beta}}}=\min_{1\leq i\leq  k}\lambda_i^{\,\frac{1}{\beta_i}}$. By \eqref{eq-standard-form-for-local-continuity-estimates-intro-2}, \eqref{eq-condition-for-U-primary-0}, \eqref{eq-condition-1-for-measurable-functions-mathcal-A-and-mathcal-B}, \eqref{eq-condition-3-for-measurable-functions-mathcal-A-and-mathcal-B}, \eqref{eq-energy-type-inequality-for-boundedness-of-U--0} and H\"older inequality, we can get
\begin{equation}\label{eq-energy-type-inequality-for-boundedness-of-U--1}
\begin{aligned}
&\frac{1}{1+m}\frac{\partial}{\partial t}\left[\int_{\R^n}U_j^{\,1+m}\,dx\right]+\,\textbf{c}\int_{\R^n}\left|\nabla U_j^{\,m}\right|^2\,dx\\
&\qquad \qquad \leq \int_{\R^n}\left(m^2\textbf{C}_2U^{2m}+\textbf{C}_5\left(\sum_{i=1}^{k}\left(u^i\right)^{q-1}\left|U_{u^i}u^i\right|\right)\left|\nabla U_j^{m}\right|\right)\chi_{_{\left\{U_j> 0\right\}}}\,dx\\
&\qquad \qquad \leq \int_{\R^n}\left(m^2\textbf{C}_2U^{2m}+\frac{\textbf{C}_1\textbf{C}_5}{\left(\underline{\lambda^{\frac{1}{\beta}}}\right)^{q-1}}U^{\frac{q-1}{\underline{\beta}}+1}\left|\nabla U_j^{m}\right|\right)\chi_{_{\left\{U_j> 0\right\}}}\,dx\\
&\qquad \qquad \leq \int_{\R^n}\left(m^2\textbf{C}_2U^{2m}+\frac{\textbf{C}_1\textbf{C}_5}{4\textbf{c}\left(\underline{\lambda^{\frac{1}{\beta}}}\right)^{q-1}}U^{\frac{2(q-1)}{\underline{\beta}}+2}+\frac{\textbf{c}}{2}\left|\nabla U_j^m\right|^2\right)\chi_{_{\left\{U_j> 0\right\}}}\,dx.
\end{aligned}
\end{equation}
By \eqref{eq-level-equence-L-j-with-mathcal-K} and \eqref{eq-energy-type-inequality-for-boundedness-of-U--1}, there exist constants $C_1$, $C_2$ depending on $\textbf{C}_1$, $\textbf{C}_2$, $\textbf{C}_5$, $\textbf{c}$, $m$, $q$ and  $\underline{\lambda^{\beta}}$ such that
\begin{equation}\label{eq-energy-type-inequality-for-boundedness-of-U}
\begin{aligned}
&\Rightarrow \qquad \frac{\partial}{\partial t}\left[\int_{\R^n}U_j^{1+m}\,dx\right]+\int_{\R^n}\left|\nabla U_j^{m}\right|^2\,dx\\
&\qquad \qquad \qquad \qquad \leq C_1\left(\int_{\R^n}\left(U_j^{2m}+L_j^{2m}\chi_{_{\left\{U_j> 0\right\}}}\right)\,dx+\int_{\R^n}\left(U_j^{\frac{2(q-1)}{\underline{\beta}}+2}+L_j^{\frac{2(q-1)}{\underline{\beta}}+2}\chi_{_{\left\{U_j> 0\right\}}}\right)\,dx\right)\\
&\qquad \qquad \qquad \qquad \leq C_2\left(\int_{\R^n}2^{2mj}U_{j-1}^{2m}\,dx+\int_{\R^n}2^{\left(\frac{2(q-1)}{\underline{\beta}}+2\right)j}U_{j-1}^{\frac{2(q-1)}{\underline{\beta}}+2}\chi_{_{\left\{U_j> 0\right\}}}\,dx\right)
\end{aligned}
\end{equation}
since
\begin{equation*}\label{eq-lower-bound-of-U-j-1-on-U-j-positive-1}
U=U_j+L_j \qquad \mbox{and} \qquad U_{j-1}\geq \frac{L_j}{2^j} \qquad \mbox{on $\left\{U_j>0\right\}$}.
\end{equation*}
For fixed $t_0>0$, let $T_j=t_0\left(1-\frac{1}{2^{(1+m)j}}\right)$ and 
\begin{equation*}
A_j=\sup_{t\geq T_j}\left(\int_{\R^n}U_j^{1+m}\,dx\right)+\int_{T_j}^{\infty}\int_{\R^n}\left|\nabla U^{m}_j\right|^2\,dxdt.
\end{equation*}
Integrating \eqref{eq-energy-type-inequality-for-boundedness-of-U} over $\left(s,t\right)$ and $\left(s,\infty\right)$, $\left(T_{j-1}<s<T_j,\,\,t>T_j\right)$, we have
\begin{equation*}
\begin{aligned}
A_j\leq \int_{\R^n}U_j^{1+m}\left(x,s\right)\,dx +C_2\left(\int_{T_{j-1}}^{\infty}\int_{\R^n}2^{2mj}U_{j-1}^{2m}\,dxdt+\int_{T_{j-1}}^{\infty}\int_{\R^n}2^{\left(\frac{2(q-1)}{\underline{\beta}}+2\right)j}U_{j-1}^{\frac{2(q-1)}{\underline{\beta}}+2}\chi_{_{\left\{U_j\geq 0\right\}}}\,dxdt\right).
\end{aligned}
\end{equation*}
Taking the mean value in $s$ on $\left[T_{j-1},T_j\right]$, we have
\begin{equation}\label{eq-inequality-with-A-j-2}
\begin{aligned}
A_j\leq \left[\frac{2^{(1+m)j}}{t_0}\int_{T_{j-1}}^{\infty}\int_{\R^n}U_j^{1+m}\,dxdt+C_2\left(\int_{T_{j-1}}^{\infty}\int_{\R^n}2^{2mj}U_{j-1}^{2m}\,dxdt +\int_{T_{j-1}}^{\infty}\int_{\R^n}2^{\left(\frac{2(q-1)}{\underline{\beta}}+2\right)j}U_{j-1}^{\frac{2(q-1)}{\underline{\beta}}+2}\chi_{_{\left\{U_j\geq 0\right\}}}\,dxdt\right)\right].
\end{aligned}
\end{equation}
By Proposition 3.1 of  Chap. I of \cite{Di}, there exist constant $C_3>0$ such that
\begin{equation*}
\left(\int_{T_{j-1}}^{\infty}\int_{\R^n}\left(U_{j-1}^m\right)^{2\left(1+\frac{m+1}{nm}\right)}\,dxdt\right)^{\frac{1}{2\left(1+\frac{m+1}{nm}\right)}}\leq \frac{1}{C_3}\left(\sup_{t\geq T_{j+1}}\left(\int_{\R^n}\left(U_j^m\right)^{\frac{1+m}{m}}\,dx\right)+\int_{T_{j+1}}^{\infty}\int_{\R^n}\left|\nabla U^{m}_j\right|^2\,dxdt\right),
\end{equation*}
i.e., we also have
\begin{equation}\label{eq-Sobolev-and-Interpolation-inequalities-for-replacign-U-by-U-j}
\begin{aligned}
A_{j-1}\geq C_3\left(\int_{T_{j-1}}^{\infty}\int_{\R^n}U_{j-1}^{2m\left(1+\frac{m+1}{nm}\right)}\,dxdt\right)^{\frac{1}{2\left(1+\frac{m+1}{nm}\right)}}.
\end{aligned}
\end{equation}
By \eqref{eq-inequality-with-A-j-2} and \eqref{eq-Sobolev-and-Interpolation-inequalities-for-replacign-U-by-U-j}, there exists a constant $C_4\left(t_0,\textbf{K}\right)>0$ such that
\begin{equation}\label{eq-inequality-with-A-j-4}
\begin{aligned}
A_j&\leq 2^{2m\left(1+\frac{m+1}{nm}\right)j}\left(\frac{1}{t_0\textbf{K}^{m-1+\frac{2\left(m+1\right)}{n}}}+\frac{C_2}{\textbf{K}^{\frac{2\left(m+1\right)}{n}}}+\frac{C_2}{\textbf{K}^{2\left(m+\frac{m+1}{n}-\frac{q-1}{\underline{\beta}}-1\right)}}\right)\int_{T_{j-1}}^{\infty}\int_{\R^n}U_{j-1}^{2m\left(1+\frac{m+1}{nm}\right)}\,dxdt\\
&\leq \frac{2^{2m\left(1+\frac{m+1}{nm}\right)j}}{C_4\left(t_0,\textbf{K}\right)}A_{j-1}^{1+\left(1+\frac{2(m+1)}{nm}\right)}
\end{aligned}
\end{equation}
since 
\begin{equation*}
U_{j-1}\geq \frac{\textbf{K}}{2^j}\qquad  \mbox{on $\left\{U_j> 0\right\}$}.
\end{equation*}
By \eqref{eq-range-of-constant-q-in-forcing-term}, the constant $C_4\left(t_0,\textbf{K}\right)$ satisfies
\begin{equation*}
C_4\left(t_0,\textbf{K}\right)\to \infty \qquad \mbox{as $\textbf{K}\to\infty$}.
\end{equation*}
Choose the constant $\textbf{K}>0$ so large that
\begin{equation*}
A_1\leq \left(C_4\left(t_0,\textbf{K}\right)\right)^{\frac{1}{1+\frac{2(m+1)}{mn}}}4^{-m\left(1+\frac{m+1}{nm}\right)\left(\frac{1}{1+\frac{2(m+1)}{mn}}\right)^2}.
\end{equation*}
Then, by Lemma 4.1 of Chap. I of \cite{Di} we have 
\begin{equation*}
A_j\to 0\qquad \mbox{ as $j\to\infty$}.
\end{equation*}
Therefore,
\begin{equation*}
\sup_{x\in\R^n,\,t\geq t_0}\left|U\left(x,t\right)\right|\leq \textbf{K}=\textbf{K}\left(t_0\right) 
\end{equation*} 
and the theorem follows. 
\end{proof}

\section{Law of Mass Conservation in $L^1$}
\setcounter{equation}{0}
\setcounter{thm}{0}

\indent  This section will be devoted to prove the $L^1$ mass conservation of the parabolic system \eqref{eq-standard-form-for-local-continuity-estimates-intro-1} under the assumption that 
\begin{equation*}
\mathcal{B}\left(u^i,x,t\right)=0 \qquad \forall(x,t)\in\R^n\times[0,\infty),
\end{equation*} 
i.e., we consider the law of mass conservation when the solution $\bold{u}=\left(u^1,\cdots,u^k\right)$ satisfies the parabolic system
\begin{equation}\label{eq-main-for-L-1-mass-conservation-without-vector-B}
	\left(u^i\right)_t=\nabla\cdot\left(mU^{m-1}\mathcal{A}\left(\nabla u^i,u^i,x,t\right)\right) \qquad \mbox{in $\R^n\times(0,\infty)$}\qquad \left(1\leq i\leq k\right)
\end{equation}
with the structural conditions \eqref{eq-standard-form-for-local-continuity-estimates-intro-2} and \eqref{eq-condition-for-U-primary--1}-\eqref{eq-condition-1-02-for-measurable-functions-mathcal-A-and-mathcal-B}. We divide this section into three subsections with respect to the range of the exponents $m$ in the diffusion coefficient. The first one is about the $L^1$ mass conservation on the nondegenerate parabolic systems.
\subsection{Nondegenerate Case: $m=1$}

Let $0<\theta<1$ and let $\eta_n\in C^{\infty}\left(\R^n\right)$ be a cut-off function such that
\begin{equation*}
\eta_n(x)=1\qquad \mbox{for $|x|\leq n$},\qquad \eta_n(x)=0\qquad \mbox{for $|x|\geq n+1$},\qquad 0<\eta_n(x)<1\qquad \mbox{for $n<|x|<n+1$}.
\end{equation*}
and
\begin{equation*}
\left\|\nabla \eta_n\right\|_{L^{\infty}}\leq 2 \qquad \forall n\in \N.
\end{equation*}
Then, by  \eqref{eq-condition-1-01-for-measurable-functions-mathcal-A-and-mathcal-B}, \eqref{eq-main-for-L-1-mass-conservation-without-vector-B} and Young's inequality, we have the following energy type inequality:
\begin{equation}\label{eq-energy-type-equation-after-multi-U-m-1-plus-1-over-beta-0--1}
\begin{aligned}
&\frac{1}{1+\theta}\frac{\partial}{\partial t}\left(\int_{\R^n}\left(u^i\right)^{1+\theta}\eta_n^2\,dx\right)+\frac{\theta\,\textbf{c}}{2}\int_{\R^n}\left(u^i\right)^{\theta-1}\left|\nabla u^i\right|^2\eta_n^2\,dx\\
&\qquad \qquad \leq \left(\textbf{C}_2m\,\theta+\frac{2\textbf{C}_3}{\theta\textbf{c}}+4\textbf{C}_4\right)\int_{\R^n}\left(u^i\right)^{1+\theta}\,dx= C_{\ast}\int_{\R^n}\left(u^i\right)^{1+\theta}\,dx.
\end{aligned}
\end{equation}
Integrating over $(0,t)$ and letting $n\to\infty$ in \eqref{eq-energy-type-equation-after-multi-U-m-1-plus-1-over-beta-0--1}, we have
\begin{equation}\label{eq-energy-type-equation-after-simplified-0--1}
\begin{aligned}
&\left(1-C_{\ast}(1+\theta)t\right)\sup_{0<\tau<t}\int_{\R^n}\left(u^i\right)^{1+\theta}(x,\tau)\,dx+\frac{\theta(1+\theta)\,\textbf{c}}{2}\int_{0}^{t}\int_{\R^n}\left(u^i\right)^{\theta-1}\left|\nabla u^i\right|^2\,dxd\tau\\
&\qquad \qquad \qquad \qquad \leq \int_{\R^n}\left(u^i\right)^{1+\theta}(x,0)\,dx.
\end{aligned}
\end{equation}
Let $t_0=\frac{1}{C_{\ast}(1+\theta)}$. Then by \eqref{eq-energy-type-equation-after-simplified-0--1},
\begin{equation*}
\begin{aligned}
\sup_{0<\tau<t_0}\int_{\R^n}\left(u^i\right)^{1+\theta}(x,\tau)\,dx+\theta(1+\theta)\,\textbf{c}\int_{0}^{t_0}\int_{\R^n}\left(u^i\right)^{\theta-1}\left|\nabla u^i\right|^2\,dxd\tau\leq 2\int_{\R^n}\left(u^i\right)^{1+\theta}(x,0)\,dx.
\end{aligned}
\end{equation*}
Applying above arguments on $(t_0,2t_0)$, we can get
\begin{equation*}
\begin{aligned}
\sup_{0<\tau<2t_0}\int_{\R^n}\left(u^i\right)^{1+\theta}(x,\tau)\,dx+\theta(1+\theta)\,\textbf{c}\int_{0}^{2t_0}\int_{\R^n}\left(u^i\right)^{\theta-1}\left|\nabla u^i\right|^2\,dxd\tau\leq 2(1+2)\int_{\R^n}\left(u^i\right)^{1+\theta}(x,0)\,dx.
\end{aligned}
\end{equation*}
Continuing in this manner, we finally get
\begin{align}
&\sup_{0<\tau<t}\int_{\R^n}\left(u^i\right)^{1+\theta}(x,\tau)\,dx+\theta(1+\theta)\,\textbf{c}\int_{0}^{t}\int_{\R^n}\left(u^i\right)^{\theta-1}\left|\nabla u^i\right|^2\,dxd\tau\leq 2(1+2)^{n'}\int_{\R^n}\left(u^i\right)^{1+\theta}(x,0)\,dx \label{eq-energy-type-equation-after-simplified-with-time-t-1-2}\\
&\Rightarrow \qquad \sup_{0<\tau<t}\int_{\R^n}\left(u^i\right)^{1+\theta}(x,\tau)\,dx\leq 2(1+2)^{n'}\int_{\R^n}\left(u^i\right)^{1+\theta}(x,0)\,dx \qquad\left(0<\theta<1\right) \label{eq-energy-type-equation-after-simplified-with-time-t-1-3}
\end{align}
for any $t>0$ where $n'$ is the natural number satisfying
\begin{equation*}
\left(n'-1\right)t_0<t\leq n't_0.
\end{equation*} 
Letting $\theta\to 0$ in \eqref{eq-energy-type-equation-after-simplified-with-time-t-1-3},
\begin{equation}\label{eq-L-1-of-compoinent-u0i-in-nondegenerated-case}
\left\|u^i\left(\cdot,t\right)\right\|_{L^1(\R^n)}\leq C(t)\left\|u_0^i\right\|_{L^1(\R^n)}
\end{equation}
for some constant $0<C(t)<\infty$. \\
\indent We are now ready for the proof of Theorem \ref{thm-L-1-mass-conservation-nondegenerate-range}.

\begin{proof}[\textbf{Proof of Theorem \ref{thm-L-1-mass-conservation-nondegenerate-range}}]
	Let $\zeta_0\in C^{\infty}\left(\R^n\right)$ be a cut-off function such that
	\begin{equation*}
	\zeta_0(x)=1\quad \mbox{for $|x|\leq 1$},\qquad \zeta_0(x)=0\quad \mbox{for $|x|\geq 2$},\qquad 0<\zeta_0(x)<1\quad \mbox{for $1<|x|<2$}
	\end{equation*}
	and let $\zeta_R(x)=\zeta_0\left(\frac{x}{R}\right)$ for any $R>1$. By weak formulation \eqref{eq-formulation-for-weak-solution-of-u-h} for  \eqref{eq-main-for-L-1-mass-conservation-without-vector-B},  we can get
	\begin{equation}\label{eq-first-just-after-multi-test-cut-off-in-general-nondegenerate-case}
	\begin{aligned}
	\int_{\R^n}u^i(x,t)\,\zeta_R(x)\,dx-\int_{\R^n}u_0^i(x,t)\,\zeta_R(x)\,dx=-\int_{0}^{t}\int_{\R^n}\mathcal{A}\left(\nabla u^i,u^i,x,t\right)\cdot\nabla\zeta_R\,dxdt.
	\end{aligned}
	\end{equation}
	Let $0<\theta<1$ be a constant satisfying 
	\begin{equation}\label{eq-range-of-q-for-L-1-mass-conservatino-in-nondegenerate-case}
	0<\theta<\min\left(1,\,\frac{2}{n},p-1\right)
	\end{equation}
	 where the constant $p$ is given in Theorem \ref{thm-L-1-mass-conservation-nondegenerate-range}. Then, by \eqref{eq-condition-1-02-for-measurable-functions-mathcal-A-and-mathcal-B}, \eqref{eq-first-just-after-multi-test-cut-off-in-general-nondegenerate-case} and H\"older inequality we have
	\begin{equation}\label{eq-L-1-difference-after-changing-cut-off-fuinction-by-ratio-of-R-in-nondegenerate}
	\begin{aligned}
	&\left|\int_{\R^n}u^i(x,t)\,\zeta_R(x)\,dx-\int_{\R^n}u_0^i(x,t)\,\zeta_R(x)\,dx\right|\\
	&\qquad \qquad \leq \frac{\textbf{C}_4\left\|\nabla\zeta_0\right\|_{L^{\infty}(\R^n)}}{R}\int_{0}^{t}\int_{B_{2R}\bs B_{R}}u^i\,dxdt\\
	&\qquad \qquad \qquad\qquad  +\frac{\textbf{C}_3\left\|\nabla\zeta_0\right\|_{L^{\infty}(\R^n)}}{R}\left(\int_{0}^{t}\int_{B_{2R}\bs B_{R}}\left(u^i\right)^{1-\theta}\,dxdt\right)^{\frac{1}{2}}\left(\int_{0}^{t}\int_{B_{2R}\bs B_{R}}\left(u^i\right)^{\theta-1}\left|\nabla u^i\right|^2\,dxdt\right)^{\frac{1}{2}}\\
	&\qquad \qquad \leq \frac{\textbf{C}_4\left\|\nabla\zeta_0\right\|_{L^{\infty}(\R^n)}}{R}\int_{0}^{t}\int_{B_{2R}\bs B_{R}}u^i\,dxdt\\
	&\qquad \qquad \qquad\qquad  +\frac{\textbf{C}_3\left\|\nabla\zeta_0\right\|_{L^{\infty}(\R^n)}}{R}\left(\int_{0}^{t}R^{nq}\left(\int_{B_{2R}\bs B_{R}}u^i\,dx\right)^{1-\theta}dt\right)^{\frac{1}{2}}\left(\int_{0}^{t}\int_{B_{2R}\bs B_{R}}\left(u^i\right)^{\theta-1}\left|\nabla u^i\right|^2\,dxdt\right)^{\frac{1}{2}}\\
	&\qquad \qquad \leq \frac{\textbf{C}_4\left\|\nabla\zeta_0\right\|_{L^{\infty}(\R^n)}}{R}\int_{0}^{t}\int_{B_{2R}\bs B_{R}}u^i\,dxdt\\
	&\qquad \qquad \qquad\qquad  +\frac{\textbf{C}_3\left\|\nabla\zeta_0\right\|_{L^{\infty}(\R^n)}}{R^{1-\frac{nq}{2}}}\left(\int_{0}^{t}\left(\int_{B_{2R}\bs B_{R}}u^i\,dx\right)^{1-\theta}dt\right)^{\frac{1}{2}}\left(\int_{0}^{t}\int_{B_{2R}\bs B_{R}}\left(u^i\right)^{\theta-1}\left|\nabla u^i\right|^2\,dxdt\right)^{\frac{1}{2}}.
	\end{aligned}
	\end{equation}
	By \eqref{eq-energy-type-equation-after-simplified-with-time-t-1-2} and \eqref{eq-L-1-of-compoinent-u0i-in-nondegenerated-case}, we get
	\begin{equation}\label{eq-L-1-control-of-u-to-some-power-and-nabla-u-double}
	\left(u^i\right)^{\theta-1}\left|\nabla u^i\right|^2\in L^1\left(\R^n\times(0,t)\right), \qquad u^i\in L^1\left(0,t:L^1\left(\R^n\right)\right)
	\end{equation} 
	since $u^i_0\in L^1\left(\R^n\right)\cap L^p\left(\R^n\right)$. By  \eqref{eq-L-1-control-of-u-to-some-power-and-nabla-u-double}, the right hand side of  \eqref{eq-L-1-difference-after-changing-cut-off-fuinction-by-ratio-of-R-in-nondegenerate} converges to zero as $R\to\infty$. Therefore,
	\begin{equation*}
	\left|\int_{\R^n}u^i(x,t)\,\zeta_R(x)\,dx-\int_{\R^n}u_0^i(x,t)\,\zeta_R(x)\,dx\right|\to 0\qquad \mbox{as $R\to\infty$}
	\end{equation*}
	and the theorem follows. 
\end{proof}

\subsection{Degenerate Case : $m>\max\left(1,2-\frac{1}{\beta_i}\right)$}

\indent By  \eqref{eq-main-for-L-1-mass-conservation-without-vector-B}, and structure assumptions \eqref{eq-condition-for-measurable-functions-mathcal-A-bounded-below-and-above}, \eqref{eq-condition-1-for-measurable-functions-mathcal-A-and-mathcal-B}, and similar argument for \eqref{eq-energy-type-equation-after-simplified-0--1}, we have the following energy type inequality:
\begin{equation}\label{eq-energy-type-equation-after-simplified-0}
\begin{aligned}
&\sup_{0<\tau<t}\int_{\R^n}U^{m-1+\frac{1}{\beta_i}}(x,\tau)\,dx+\int_{0}^{t}\int_{\R^n}\left|\nabla U^{m-1+\frac{1}{2\beta_i}}\right|^2\,dxd\tau\\
&\qquad \qquad \leq C_1\left(\int_{\R^n}U^{m-1+\frac{1}{\beta_i}}(x,0)\,dx+\int_{0}^t\int_{\R^n}U^{2m-2+\frac{1}{\beta_i}}\,dxd\tau\right)
\end{aligned}
\end{equation}
for some constant $C_1>0$ depending on $m$, $\beta_i$, $\textbf{c}$ and $\textbf{C}_2$.\\
\indent  Let $0<\theta<1$. By \eqref{eq-main-for-L-1-mass-conservation-without-vector-B}, and structure assumptions \eqref{eq-condition-1-01-for-measurable-functions-mathcal-A-and-mathcal-B}, \eqref{eq-condition-1-02-for-measurable-functions-mathcal-A-and-mathcal-B} and a similar argument for \eqref{eq-energy-type-equation-after-simplified-0--1}  we also have
\begin{equation}\label{eq-energy-type-inequality-for-u-i-to-theta-01}
\begin{aligned}
&\sup_{0<\tau<t}\int_{\R^n}\left(u^i\right)^{1+\theta}(x,\tau)\,dx+\int_{0}^{t}\int_{\R^n}U^{m-1}\left(u^i\right)^{\theta-1}\left|\nabla u^i\right|^2\,dxdt\\
&\qquad \qquad \leq C_2\left(\int_{\R^n}\left(u^i\right)^{1+\theta}(x,0)\,dx+ \int_{0}^{t}\int_{\R^n}U^{m-1}\left(u^i\right)^{1+\theta}\,dxdt\right)
\end{aligned}
\end{equation}
for some constant $C_2>0$ depending on $m$, $\beta_i$, $\textbf{c}$ and $\textbf{C}_2$.\\
\indent To control last terms in \eqref{eq-energy-type-equation-after-simplified-0} and \eqref{eq-energy-type-inequality-for-u-i-to-theta-01}, we suppose that the bound $\textbf{K}(t)$ in the Theorem \ref{thm-Main-boundedness-of-diffusion-coefficients-of-generaized-equation} satisfies the following condition
\begin{equation*}
\int_{0}^{t}\textbf{K}^{m-1}(\tau)\,d\tau \to 0 \qquad \mbox{as $t\to 0$}.
\end{equation*}
Then there exists a constant $t_0>0$ such that
\begin{equation*}
\int_{0}^{t_0}\textbf{K}^{m-1}(\tau)\,d\tau<\frac{1}{2\left(C_1+C_2\right)}.
\end{equation*}
Then,
	\begin{equation}\label{eq-sufficinetely-small-of-integral-of-uniform-bound-of-component-in-a-short-time-11}
	\begin{aligned}
	\int_{0}^{t_0}\int_{\R^n}U^{2m-2+\frac{1}{\beta_i}}\,dxd\tau&\leq \left(\int_{0}^{t_0}\textbf{K}^{m-1}(\tau)\,d\tau\right)\left(\sup_{0<\tau<t_0}\int_{\R^n}U^{m-1+\frac{1}{\beta_i}}(x,\tau)\,dx\right)\\
	&\leq \frac{1}{2\left(C_1+C_2\right)}\left(\sup_{0<\tau<t_0}\int_{\R^n}U^{m-1+\frac{1}{\beta_i}}(x,\tau)\,dx\right)
	\end{aligned}
	\end{equation}
	and
	\begin{equation}\label{eq-sufficinetely-small-of-integral-of-uniform-bound-of-component-in-a-short-time-22}
	\begin{aligned}
	\int_{0}^{t_0}\int_{\R^n}U^{m-1}\left(u^i\right)^{1+\theta}\,dxd\tau&\leq \left(\int_{0}^{t_0}\textbf{K}^{m-1}(\tau)\,d\tau\right)\left(\sup_{0<\tau<t_0}\int_{\R^n}\left(u^i\right)^{1+\theta}(x,\tau)\,dx\right)\\
	&\leq \frac{1}{2\left(C_1+C_2\right)}\left(\sup_{0<\tau<t_0}\int_{\R^n}\left(u^i\right)^{1+\theta}(x,\tau)\,dx\right).
	\end{aligned}
	\end{equation}
By \eqref{eq-energy-type-equation-after-simplified-0}, \eqref{eq-energy-type-inequality-for-u-i-to-theta-01}, \eqref{eq-sufficinetely-small-of-integral-of-uniform-bound-of-component-in-a-short-time-11} and \eqref{eq-sufficinetely-small-of-integral-of-uniform-bound-of-component-in-a-short-time-22}, we have
\begin{equation*}
\begin{aligned}
\sup_{0<\tau<t_0}\int_{\R^n}U^{m-1+\frac{1}{\beta_i}}(x,\tau)\,dx+\int_{0}^{t_0}\int_{\R^n}\left|\nabla U^{m-1+\frac{1}{2\beta_i}}\right|^2\,dxd\tau \leq 2C_1\int_{\R^n}U^{m-1+\frac{1}{\beta_i}}(x,0)\,dx
\end{aligned}
\end{equation*}
and
\begin{equation*}
\begin{aligned}
\sup_{0<\tau<t_0}\int_{\R^n}\left(u^i\right)^{1+\theta}(x,\tau)\,dx+\int_{0}^{t_0}\int_{\R^n}U^{m-1}\left(u^i\right)^{\theta-1}\left|\nabla u^i\right|^2\,dxdt\leq 2C_2\int_{\R^n}\left(u^i\right)^{1+\theta}(x,0)\,dx.
\end{aligned}
\end{equation*}
Since the constants $C_1$ and $C_2$ are all independent of $t$ and 
\begin{equation*}
\textbf{K}(\tau_1)\geq \textbf{K}(\tau_2) \qquad \forall \tau_1<\tau_2
\end{equation*}
, applying above arguments on $(t_0,2t_0)$, we also have 
\begin{equation*}
\begin{aligned}
\sup_{0<\tau<2t_0}\int_{\R^n}U^{m-1+\frac{1}{\beta_i}}(x,\tau)\,dx+\int_{0}^{2t_0}\int_{\R^n}\left|\nabla U^{m-1+\frac{1}{2\beta_i}}\right|^2\,dxd\tau \leq 2C_1\left(1+2C_1\right)\int_{\R^n}U^{m-1+\frac{1}{\beta_i}}(x,0)\,dx
\end{aligned}
\end{equation*}
and
\begin{equation*}
\begin{aligned}
\sup_{0<\tau<2t_0}\int_{\R^n}\left(u^i\right)^{1+\theta}(x,\tau)\,dx+\int_{0}^{2t_0}\int_{\R^n}U^{m-1}\left(u^i\right)^{\theta-1}\left|\nabla u^i\right|^2\,dxdt\leq 2C_2\left(1+2C_2\right)\int_{\R^n}\left(u^i\right)^{1+\theta}(x,0)\,dx.
\end{aligned}
\end{equation*}
Continuing in this manner,  we finally get
\begin{equation}\label{eq-belong-L-p-space-by-initial-condition-energy-type-equation-after-simplified-0}
\begin{aligned}
\sup_{0<\tau<t}\int_{\R^n}U^{m-1+\frac{1}{\beta_i}}(x,\tau)\,dx+\int_{0}^{t}\int_{\R^n}\left|\nabla U^{m-1+\frac{1}{2\beta_i}}\right|^2\,dxd\tau \leq 2C_1\left(1+2C_1\right)^{n'}\int_{\R^n}U_0^{m-1+\frac{1}{\beta_i}}(x)\,dx\qquad \left(q>0\right)
\end{aligned}
\end{equation}
and
\begin{equation}\label{eq-belong-L-p-space-by-initial-condition-energy-type-inequality-for-u-i-to-theta-01}
\begin{aligned}
\sup_{0<\tau<t}\int_{\R^n}\left(u^i\right)^{1+\theta}(x,\tau)\,dx+\int_{0}^{t}\int_{\R^n}U^{m-1}\left(u^i\right)^{\theta-1}\left|\nabla u^i\right|^2\,dxdt &\leq 2C_2\left(1+2C_2\right)^{n'}\int_{\R^n}\left(u_0^i\right)^{1+\theta}(x)\,dx\\
& \leq \frac{2C_2\left(1+2C_2\right)^{n'}}{\left(\lambda_i\right)^{\frac{1+\theta}{\beta_i}}}\int_{\R^n}U_0^{\frac{1+\theta}{\beta_i}}(x)\,dx
\end{aligned}
\end{equation}
for any $t>0$ where $n'$ is the natural number satisfying
\begin{equation*}
(n'-1)t_0<t\leq n't_0.
\end{equation*}
We are now ready for the proof of Theorem \ref{thm-mass-conservation-of-component-general-degeneraate}.

\begin{proof}[\textbf{Proof of Theorem \ref{thm-mass-conservation-of-component-general-degeneraate}}]
	For $R>1$, let $\zeta_R\in C^{\infty}\left(\R^n\right)$ be a cut-off function given in the proof of Theorem \ref{thm-L-1-mass-conservation-nondegenerate-range}. Multiply the equation \eqref{eq-main-for-L-1-mass-conservation-without-vector-B} by $\zeta_R$ and integrate it over $\R^n\times(0,t)$. Then, by weak formulation \eqref{eq-formulation-for-weak-solution-of-u-h} for  \eqref{eq-main-for-L-1-mass-conservation-without-vector-B} we have
	\begin{equation}\label{eq-first-just-after-multi-test-cut-off-in-general-case}
	\begin{aligned}
	\int_{\R^n}u^i(x,t)\,\zeta_R(x)\,dx-\int_{\R^n}u_0^i(x,t)\,\zeta_R(x)\,dx=-m\int_{0}^{t}\int_{\R^n}U^{m-1}\mathcal{A}\left(\nabla u^i,u^i,x,t\right)\cdot\nabla\zeta_R\,dxdt.
	\end{aligned}
	\end{equation}
	Let $0<\theta<1$ be a constant satisfying 
	\begin{equation}\label{eq-condition-of-m-for=-mass-conservations-01}
	m>1+\frac{\theta}{\beta_i} \qquad \mbox{and} \qquad m-1+\frac{1}{\beta_i}>\frac{n\theta}{\beta_i}.
	\end{equation}
	
	Then, by \eqref{eq-condition-1-02-for-measurable-functions-mathcal-A-and-mathcal-B}, \eqref{eq-first-just-after-multi-test-cut-off-in-general-case} and H\"older inequality, we have
	\begin{equation}\label{eq-L-1-difference-after-changing-cut-off-fuinction-by-ratio-of-R}
	\begin{aligned}
	&\left|\int_{\R^n}u^i(x,t)\,\zeta_j(x)\,dx-\int_{\R^n}u_0^i(x,t)\,\zeta_j(x)\,dx\right|\\
	&\qquad \leq \frac{m\textbf{C}_4\left\|\nabla\zeta_0\right\|_{L^{\infty}(\R^n)}}{\left(\lambda_i\right)^{\frac{m-1}{\beta_i}}R}\int_{0}^{t}\int_{B_{2R}\bs B_{R}}U^{m-1+\frac{1}{\beta_i}}\,dxdt\\
	&\qquad\qquad  +\frac{m\textbf{C}_3\left\|\nabla\zeta_0\right\|_{L^{\infty}(\R^n)}}{\left(\lambda_i\right)^{\frac{1-\theta}{2\beta_i}}R}\left(\int_{0}^{t}\int_{B_{2R}\bs B_{R}}U^{m-1+\frac{1-\theta}{\beta_i}}\,dxdt\right)^{\frac{1}{2}}\left(\int_{0}^{t}\int_{B_{2R}\bs B_{R}}U^{m-1}\left(u^i\right)^{\theta-1}\left|\nabla u^i\right|^2\,dxdt\right)^{\frac{1}{2}}\\
	&\qquad \leq \frac{m\textbf{C}_4\left\|\nabla\zeta_0\right\|_{L^{\infty}(\R^n)}}{\left(\lambda_i\right)^{\frac{m-1}{\beta_i}}R}\int_{0}^{t}\int_{B_{2R}\bs B_{R}}U^{m-1+\frac{1}{\beta_i}}\,dxdt\\
	&\qquad\qquad  +\frac{m\textbf{C}_3\left\|\nabla\zeta_0\right\|_{L^{\infty}(\R^n)}}{\left(\lambda_i\right)^{\frac{1-\theta}{2\beta_i}}R^{1-\frac{n\cdot\frac{\theta}{\beta_i}}{m-1+\frac{1}{\beta_i}}}}\left(\int_{0}^{t}\left(\int_{B_{2R}\bs B_{R}}U^{m-1+\frac{1}{\beta_i}}\,dx\right)^{1-\frac{\frac{\theta}{\beta_i}}{m-1-\frac{1}{\beta_i}}}dt\right)^{\frac{1}{2}}\left(\int_{0}^{t}\int_{B_{2R}\bs B_{R}}U^{m-1}\left(u^i\right)^{\theta-1}\left|\nabla u^i\right|^2\,dxdt\right)^{\frac{1}{2}}
	\end{aligned}
	\end{equation}
	By \eqref{eq-belong-L-p-space-by-initial-condition-energy-type-equation-after-simplified-0} and \eqref{eq-belong-L-p-space-by-initial-condition-energy-type-inequality-for-u-i-to-theta-01}, there exists a constant $C_3>0$ such that
	\begin{equation}\label{eq-belong-L-p-space-by-initial-condition-energy-type-equation-after-simplified-with-q=m-2-1-beta}
	\left\|U^{m-1+\frac{1}{\beta_i}}(\cdot,t)\right\|_{L^{1}\left(\R^n\right)}\leq C_3\left\|U^{m-1+\frac{1}{\beta_i}}_0\right\|_{L^{1}\left(\R^n\right)} \qquad \mbox{and} \qquad\left\|U^{m-1}\left(u^i\right)^{\theta-1}\left|\nabla u^i\right|^2\right\|_{L^1\left(\R^n\times(0,t)\right)} \leq C_3\left\|U_0^{\frac{1+\theta}{\beta_i}}\right\|_{L^1\left(\R^n\right)}\qquad \forall t>0.
	\end{equation}
	Thus, by \eqref{eq-L-p-conditions-of-U-for-mass-conservation} and \eqref{eq-belong-L-p-space-by-initial-condition-energy-type-equation-after-simplified-with-q=m-2-1-beta}, the right hand side of \eqref{eq-L-1-difference-after-changing-cut-off-fuinction-by-ratio-of-R} converges to zero as $R\to\infty$. Therefore
	\begin{equation*}
	\left|\int_{\R^n}u^i(x,t)\,\eta_j(x)\,dx-\int_{\R^n}u_0^i(x,t)\,\eta_j(x)\,dx\right|\to 0\qquad \mbox{as $j\to\infty$}
	\end{equation*}
	and the theorem follows. 
\end{proof}

\subsection{Singular Case: $\max\left(0,1-\frac{1}{\beta_i}\right)<m<1$}

Integrate \eqref{eq-main-for-L-1-mass-conservation-without-vector-B} in a ball of radius $R$ at time $t$. Then we have
	\begin{equation*}
	\begin{aligned}
	\left|\frac{d}{dt}\left(\int_{|x|\leq R}u^i(x,t)\,dx\right)\right|&=m\left|\int_{|x|=R}U^{m-1}\mathcal{A}\left(x,t\right)\cdot\nu\,dS\right|\\
	&\leq m\int_{|x|=R}U^{m-1}\left(\textbf{C}_3\left|\nabla u^i\right|+\textbf{C}_4u^i\right)\,d\sigma=\Psi(R,t)
	\end{aligned}
	\end{equation*}
	where $\nu$ is the unit outward normal vector and $d\sigma$ is the area element on $S^{n-1}$. Thus  the claim of mass conservation in the singular case is completed if we can show that
	\begin{equation}\label{eq-converges-to-zero-of-Psi-as=R-to-zoer01}
	\Psi(R,t)\to 0 \qquad \mbox{as $R\to\infty$}.
	\end{equation}
Next four lemmas will be used to get an upper bound for the component $u^i$, which plays a key role on the proof of  \eqref{eq-converges-to-zero-of-Psi-as=R-to-zoer01}. We first are going to get the following energy estimate for the system \eqref{eq-main-for-L-1-mass-conservation-without-vector-B}.

\begin{lemma}\label{lem-conclusion-of-fisrt-inequality-of-gradient-of-sudden-power-by-supremeum}
	Let $\alpha\in\left(-1,0\right)$ be a constant such that
	\begin{equation}\label{eq-condition-of-alpha-negative-small-engothsh-012}
	0<\beta_i(m-1)+1+\alpha,\,\,\,1+\alpha,\,\,\,\beta_i\left(m-1\right)+1-\alpha<1.
	\end{equation}  
	 For $T>0$, let $\bold{u}=\left(u^1,\cdots,u^k\right)$ be a weak solution to the singular parabolic system \eqref{eq-standard-form-for-local-continuity-estimates-intro-1}  in $E_T$ with structures \eqref{eq-standard-form-for-local-continuity-estimates-intro-2}, \eqref{eq-condition-for-U-primary--1}-\eqref{eq-condition-3-for-measurable-functions-mathcal-A-and-mathcal-B}. Suppose that
	\begin{equation*}
	0<m<1\qquad \mbox{and} \qquad 1-\frac{1}{\beta_i}<m<1.
	\end{equation*}
	Then there exists a positive constant $\gamma$ depending on the data $\alpha$, $m$, $N$, $\lambda_i$, $\textbf{c}$ and $\textbf{C}_3$ such that for all cylinder $B_{\rho}(y)\times(s,t]\subset E_T$, all $\sigma\in(0,1)$,
	\begin{equation}\label{eq-conclusion-of-fisrt-inequality-of-gradient-of-sudden-power-by-supremeum}
	\begin{aligned}
	&\int_{s}^{t}\int_{B_{\rho}(y)}U^{m-1}\left(u^i\right)^{\alpha-1}\left|\nabla u^i\right|^2\,dxd\tau\\
	&\qquad \qquad\leq  \gamma\left[\left(\textbf{C}_2+\textbf{C}_4+\frac{1+\textbf{C}_4}{\sigma^2\rho^2}\right)\left(S_{\sigma}^{i}\right)^{\beta_i(m-1)+\alpha+1}(t-s)\rho^{n(-\alpha-\beta_i(m-1))}+\frac{\left(S_{\sigma}^{i}\right)^{1+\alpha}}{\rho^{n\alpha }}\right]
	\end{aligned}
	\end{equation}
	where 
	\begin{equation*}
	S_{\sigma}^i=\sup_{s<\tau<t}\int_{B_{(1+\sigma)\rho}(y)}u^i(x,\tau)\,dx.
	\end{equation*}
\end{lemma}
\begin{proof}
	We will use a modification of the proof of  Lemma B.1.1 of \cite{DGV1} to prove the lemma. Without loss of generality we let $(y,s)=(0,0)$. Consider a nonnegative, piecewise smooth cut off function $\xi(x)$ such that
	\begin{equation*}
	\begin{cases}
	\begin{array}{cccl}
	0\leq \xi\leq 1 &&&\mbox{in $B_{(1+\sigma)\rho}$}\\
	\xi=1&&&\mbox{in $B_{\rho}$}\\
	\xi=0&&&\mbox{on $\partial B_{(1+\sigma)\rho}$}\\
	\left|\nabla\xi\right|\leq \frac{1}{\sigma\rho}&&&\mbox{in $B_{(1+\sigma)\rho}$}.
	\end{array}
	\end{cases}
	\end{equation*}
	We multiply the equation in \eqref{eq-main-for-L-1-mass-conservation-without-vector-B} by $\left(u^i+\epsilon\right)^{\alpha}\xi^2$ and integrate it over $B_{(1+\sigma)\rho}\times(0,t]$. Then, letting $\epsilon\to0$, we have
	\begin{equation}\label{eq-energy-type-ineqwuality-for0U-for-control-of-gradient}
	\begin{aligned}
	\frac{1}{1+\alpha}\int_{B_{(1+\sigma)\rho}}\left(u^i\right)^{1+\alpha}\xi^2\,dx(t)&= \frac{1}{1+\alpha}\int_{B_{(1+\sigma)\rho}}\left(u^i\right)^{1+\alpha}\xi^2\,dx(0)\\
	&\qquad \qquad -\int_{0}^{t}\int_{B_{(1+\sigma)\rho}}mU^{m-1}\mathcal{A}\left(\nabla u^i,u^i,x,t\right)\cdot\nabla\left(\left(u^i\right)^{\alpha}\xi^2\right)\,dxd\tau.
	\end{aligned}
	\end{equation}
	By H\"older inequality,
	\begin{equation}\label{eq-first-terms-in-energy-type-ineqwuality-for0U-for-control-of-gradient}
	\left|	\frac{1}{1+\alpha}\int_{B_{(1+\sigma)\rho}}\left(u^i\right)^{1+\alpha}\xi^2\,dx(t)-\frac{1}{1+\alpha}\int_{B_{(1+\sigma)\rho}}\left(u^i\right)^{1+\alpha}\xi^2\,dx(0)\right|\leq\frac{2^{2n+1}}{1+\alpha}\left(S_{\sigma}^{i}\right)^{1+\alpha}\rho^{-n\alpha}.
	\end{equation}
	By structure conditions \eqref{eq-condition-1-01-for-measurable-functions-mathcal-A-and-mathcal-B}, \eqref{eq-condition-1-02-for-measurable-functions-mathcal-A-and-mathcal-B} and Young's inequality, we have
	\begin{equation}\label{eq-second-terms-in-energy-type-ineqwuality-for0U-for-control-of-gradient}
	\begin{aligned}
	&\int_{0}^{t}\int_{B_{(1+\sigma)\rho}}mU^{m-1}\mathcal{A}\left(\nabla u^i,u^i,x,t\right)\cdot\nabla\left(\left(u^i\right)^{\alpha}\xi^2\right)\,dxd\tau\\
	&\qquad \qquad \leq -\textbf{c}m\left|\alpha\right|\int_{0}^{t}\int_{B_{(1+\sigma)\rho}}U^{m-1}\left(u^i\right)^{\alpha-1}\left|\nabla u^i\right|^2\xi^2\,dxd\tau+m\left|\alpha\right|\textbf{C}_2\int_{0}^{t}\int_{B_{(1+\sigma)\rho}}U^{m-1}\left(u^i\right)^{\alpha+1}\xi^2\,dxd\tau\\
	&\qquad \qquad \qquad \qquad +2m\textbf{C}_3\int_{0}^{t}\int_{B_{(1+\sigma)\rho}}U^{m-1}\left(u^i\right)^{\alpha}\xi\left|\nabla u^i\right|\left|\nabla\xi\right|\,dxd\tau+2m\textbf{C}_4\int_{0}^{t}\int_{B_{(1+\sigma)\rho}}U^{m-1}\left(u^i\right)^{\alpha+1}\xi\left|\nabla\xi\right|\,dxd\tau\\
	&\qquad \qquad \leq -\frac{\textbf{c}m\left|\alpha\right|}{2}\int_0^t\int_{B_{(1+\sigma)\rho}}U^{m-1}\left(u^i\right)^{\alpha-1}\left|\nabla u^i\right|^2\xi^2\,dxd\tau\\
	&\qquad \qquad \qquad \qquad+\left(m\left|\alpha\right|\textbf{C}_2+\frac{2m\textbf{C}_4}{\sigma\rho}+\frac{2m\textbf{C}_3^2}{\textbf{c}\left|\alpha\right|\sigma^2\rho^2}\right)\int_0^t\int_{B_{(1+\sigma)\rho}}U^{m-1}\left(u^i\right)^{\alpha+1}\,dxd\tau\\
	&\qquad \qquad \leq -\frac{\textbf{c}m\left|\alpha\right|}{2}\int_0^t\int_{B_{(1+\sigma)\rho}}U^{m-1}\left(u^i\right)^{\alpha-1}\left|\nabla u^i\right|^2\xi^2\,dxd\tau\\
	&\qquad \qquad \qquad \qquad +\frac{1}{\left(\lambda_i\right)^{1-m}}\left(m\left|\alpha\right|\textbf{C}_2+\frac{2m\textbf{C}_4}{\sigma\rho}+\frac{2m\textbf{C}_3^2}{\textbf{c}\left|\alpha\right|\sigma^2\rho^2}\right)\int_0^t\int_{B_{(1+\sigma)\rho}}\left(u^i\right)^{\beta_i(m-1)+\alpha+1}\,dxd\tau\\
	&\qquad \qquad \leq -\frac{\textbf{c}m\left|\alpha\right|}{2}\int_0^t\int_{B_{(1+\sigma)\rho}}U^{m-1}\left(u^i\right)^{\alpha-1}\left|\nabla u^i\right|^2\xi^2\,dxd\tau\\
	&\qquad \qquad \qquad \qquad +\gamma_0\left(\textbf{C}_2+\frac{\textbf{C}_4}{\sigma\rho}+\frac{1}{\sigma^2\rho^2}\right)\left(S_{\sigma}^i\right)^{\beta_i(m-1)+\alpha+1}t\rho^{n\left(-\alpha-\beta_i(m-1)\right)}.
	\end{aligned}
	\end{equation}
	where
	\begin{equation*}
	\gamma_0=\frac{1}{\left(\lambda_i\right)^{1-m}}\max\left\{2m,\frac{2m\textbf{C}_3^2}{\textbf{c}\left|\alpha\right|}\right\}.
	\end{equation*}
	By \eqref{eq-energy-type-ineqwuality-for0U-for-control-of-gradient}, \eqref{eq-first-terms-in-energy-type-ineqwuality-for0U-for-control-of-gradient} and \eqref{eq-second-terms-in-energy-type-ineqwuality-for0U-for-control-of-gradient}, \eqref{eq-conclusion-of-fisrt-inequality-of-gradient-of-sudden-power-by-supremeum} holds  and the lemma follows. 
\end{proof}

Now, we will get the following integral Harnack estimate for the system \eqref{eq-main-for-L-1-mass-conservation-without-vector-B}.

\begin{lemma}\label{lem-the-first-step-for-controlling-sup-u-i-byu-int-of-yu-i-and-ration-t-and-raduis}
	Let $\alpha\in\left(-1,0\right)$ be a constant given by \eqref{eq-condition-of-alpha-negative-small-engothsh-012} and let $0<m<$ satisfy
\begin{equation*}
1-\frac{1}{\beta_i}<m<1.
\end{equation*}
	For$T>0$, let $\bold{u}=\left(u^1,\cdots,u^k\right)$ be a weak solution to the singular parabolic system \eqref{eq-standard-form-for-local-continuity-estimates-intro-1} in $E_T$ with structure conditions \eqref{eq-standard-form-for-local-continuity-estimates-intro-2}, \eqref{eq-condition-for-U-primary--1}-\eqref{eq-condition-3-for-measurable-functions-mathcal-A-and-mathcal-B}.
	Then, there exists a positive constant $\gamma$ depending on the data  $\alpha$, $m$, $N$, $\lambda_i$, $\textbf{c}$ and $\textbf{C}_3$ such that for all cylinder $B_{2\rho}(y)\times(s,t]\in E_T$
	\begin{equation}\label{eq-conclusion-sup-bounded-by-inf-and-term-of-radius=time}
	\sup_{s<\tau<t}\int_{B_{\rho}(y)}u^i(x,\tau)\,dx\leq \gamma\left[\inf_{s<\tau<t}\int_{B_{2\rho}(y)}u^i(x,\tau)\,dx+\left(\sqrt{1+\textbf{C}_4}+\left(\textbf{C}_4+\sqrt{\textbf{C}_2+\textbf{C}_4}\right)\rho^{\frac{1}{\beta_i(1-m)}}\right)\left(\frac{t-s}{\rho^{\theta_i}}\right)^{\frac{1}{\beta_i(1-m)}}\right]
	\end{equation}
	where $\theta_i=\beta_in\left(m-1\right)+2$.
\end{lemma}
\begin{proof}
	We will use a modification of the proof of  Proposition B.1.1 of \cite{DGV1} to prove the lemma. Without loss of generality we let $(y,s)=(0,0)$. For each $j\in\N$, set
	\begin{equation*}
	\rho_j=\sum_{l=1}^j\frac{1}{2^l}\rho,\qquad B_j=B_{\rho_j},\qquad \widetilde{\rho}_j=\frac{\rho_j+\rho_{j+1}}{2}, \qquad \widetilde{B}_j=B_{\widetilde{\rho}_j}.
	\end{equation*}
	We also consider a nonnegative, piecewise smooth cut off function $\xi_j(x)$ such that
	\begin{equation*}
	\begin{cases}
	\begin{array}{cccl}
	0\leq \xi_j\leq 1 &&&\mbox{in $\widetilde{B}_j$}\\
	\xi_j=1&&&\mbox{in $B_j$}\\
	\xi_j=0&&&\mbox{on $\partial\widetilde{B}_j$}\\
	\left|\nabla\xi_j\right|\leq \frac{2^{j+2}}{\rho}&&&\mbox{in $\widetilde{B}_j$}.
	\end{array}
	\end{cases}
	\end{equation*}
	Multiplying the equation in \eqref{eq-main-for-L-1-mass-conservation-without-vector-B} by $\xi_j$ and integrating it over $\widetilde{B}_j\times(\tau_1,t)$, we have
	\begin{equation}\label{eq-integration-and-split-for-control=of-supremum-by-infimum-and-S-0n--1}
	\begin{aligned}
	\int_{\widetilde{B}_j}u^i(x,t)\,dx&\leq \int_{\widetilde{B}_j}u^i(x,\tau_1)\,dx+\frac{2^{j+2}}{\rho}\int_{\tau_1}^{t}\int_{\widetilde{B}_j}\left|mU^{m-1}\mathcal{A}\left(\nabla u^i,u^i,x,t\right)\right|\,dxd\tau\\
	&\leq\int_{\widetilde{B}_j}u^i(x,\tau_1)\,dx+\frac{m2^{j+2}}{\rho}\int_{\tau_1}^{t}\int_{\widetilde{B}_j}\left(\textbf{C}_3U^{m-1}\left|\nabla u^i\right|+\textbf{C}_4U^{m-1}u^i\right)\,dxd\tau.
	\end{aligned}
	\end{equation}
	By \eqref{eq-standard-form-for-local-continuity-estimates-intro-2}, \eqref{eq-integration-and-split-for-control=of-supremum-by-infimum-and-S-0n--1} and H\"older inequality,
	\begin{equation}\label{eq-integration-and-split-for-control=of-supremum-by-infimum-and-S-0n}
	\begin{aligned}
	\int_{\widetilde{B}_j}u^i(x,t)\,dx&\leq\int_{\widetilde{B}_j}u^i(x,\tau_1)\,dx+\frac{\textbf{C}_4m2^{j+2}}{\left(\lambda_i\right)^{1-m}\rho}\int_{\tau_1}^{t}\int_{\widetilde{B}_j}\left(u^i\right)^{\beta_i(m-1)+1}\,dxd\tau\\
	&\qquad \qquad +\frac{\textbf{C}_3m2^{j+2}}{\left(\lambda_i\right)^{\frac{1-m}{2}}\rho}\left(\int_{\tau_1}^{t}\int_{\widetilde{B}_j}U^{m-1}\left(u^i\right)^{\alpha-1}\left|\nabla u^i\right|^2\,dxd\tau\right)^{\frac{1}{2}}\left(\int_{\tau_1}^{t}\int_{\widetilde{B}_j}\left(u^i\right)^{\beta_i(m-1)+1-\alpha}\,dxd\tau\right)^{\frac{1}{2}}
	\end{aligned}
	\end{equation}
   where the constant $\alpha$ is given by \eqref{eq-condition-of-alpha-negative-small-engothsh-012}. By H\"older inequality, Lemma \ref{lem-conclusion-of-fisrt-inequality-of-gradient-of-sudden-power-by-supremeum} and \eqref{eq-integration-and-split-for-control=of-supremum-by-infimum-and-S-0n}, there exists a constant $\gamma_0>0$ depending on $\alpha$, $m$, $N$, $\lambda_i$, $\textbf{c}$ and $\textbf{C}_3$ such that
   \begin{equation*}
   \begin{aligned}
   \int_{\widetilde{B}_j}u^i(x,t)\,dx&\leq \int_{\widetilde{B}_j}u^i(x,\tau_1)\,dx+\frac{\gamma_0\textbf{C}_42^{j}t}{\rho^{\,\beta_i\,n\left(m-1\right)+1}}\left(S_{j+1}^i\right)^{\beta_i\left(m-1\right)+1}\\
   &\qquad \qquad +\gamma_0 2^j\left(\left(\frac{\sqrt{\textbf{C}_2+\textbf{C}_4}}{\rho}+\frac{2^j\sqrt{1+\textbf{C}_4}}{\rho^2}\right)\sqrt{t}\,\left(S_{j+1}^i\right)^{\frac{\beta_i(m-1)+\alpha+1}{2}}\rho^{\frac{n(-\alpha-\beta_i(m-1))}{2}}+\frac{\left(S_{j+1}^i\right)^{\frac{1+\alpha}{2}}\rho^{-\frac{n\alpha}{2}}}{\rho}\right)\\
   &\qquad \qquad\qquad \qquad  \times\left(\sqrt{t}\,\rho^{\frac{n\left(-\beta_i(m-1)+\alpha\right)}{2}}\left(S_{j+1}^i\right)^{\frac{\beta_i(m-1)+1-\alpha}{2}}\right)\\
   &=\int_{\widetilde{B}_j}u^i(x,\tau_1)\,dx+\gamma_0\textbf{C}_42^{j}\rho\left(S_{j+1}^i\right)^{\beta_i\left(m-1\right)+1}\left(\frac{t}{\rho^{\theta_i}}\right)\\
   &\qquad\qquad +\gamma_0 2^j\left(\frac{\sqrt{\textbf{C}_2+\textbf{C}_4}}{\rho}+\frac{2^j\sqrt{1+\textbf{C}_4}}{\rho^2}\right)\rho^2\left(S_{j+1}^i\right)^{\beta_i\left(m-1\right)+1}\left(\frac{t}{\rho^{\theta_i}}\right)\\
   &\qquad \qquad \qquad \qquad +\gamma_0 2^j\left(S_{j+1}^i\right)^{\frac{\beta_i\left(m-1\right)+2}{2}}\left(\frac{t}{\rho^{\theta_i}}\right)^{\frac{1}{2}}\\
   &\leq \int_{\widetilde{B}_j}u^i(x,\tau_1)\,dx+\gamma_02^j\left(\textbf{C}_4+\sqrt{\textbf{C}_2+\textbf{C}_4}\right)\rho\left(S_{j+1}^i\right)^{\beta_i\left(m-1\right)+1}\left(\frac{t}{\rho^{\theta_i}}\right)\\
   &\qquad \qquad +\gamma_02^j\left[2^j\sqrt{1+\textbf{C}_4}\left(S_{j+1}^i\right)^{\beta_i\left(m-1\right)+1}\left(\frac{t}{\rho^{\theta_i}}\right)+\left(S_{j+1}^i\right)^{\frac{\beta_i\left(m-1\right)+2}{2}}\left(\frac{t}{\rho^{\theta_i}}\right)^{\frac{1}{2}}\right]
   \end{aligned}
   \end{equation*}
	where
	\begin{equation*}
	S^i_j=\sup_{0<\tau\leq t}\int_{B_j}u^i(x,\tau)\,dx.
	\end{equation*}
	We now choose $\tau_1$ such that
	\begin{equation*}
	\int_{B_{2\rho}}u^i(x,\tau_1)\,dx=\inf_{0<\tau\leq t}\int_{B_{2\rho}}u^i(x,\tau)\,dx:=I^i
	\end{equation*}
	and let $m_i=\beta_i(m-1)+1\in(0,1)$. Then we have
	\begin{equation}\label{eq-inequality-for-iteration-S-i-j-0549t}
	\begin{aligned}
	S_{j}^i\leq I^i+\gamma_0 2^{j}\left(\textbf{C}_4+\sqrt{\textbf{C}_2+\textbf{C}_4}\right)\rho\, S_{j+1}^{m_i}\left(\frac{t}{\rho^{\theta_i}}\right)+\gamma_0\left( 4^j\sqrt{1+\textbf{C}_4}S_{j+1}^{m_i}\left(\frac{t}{\rho^{\theta_i}}\right)+2^jS_{j+1}^{\frac{m_i+1}{2}}\left(\frac{t}{\rho^{\theta_i}}\right)^{\frac{1}{2}}\right).
	\end{aligned}
	\end{equation}
	By Young's inequality,
	\begin{equation}\label{eq-split-by-Youngs-inequality-S-m-i-j-plus-1-and-t-over-rho-theta-power-01}
	2^j\rho\, S_{j+1}^{m_i}\left(\frac{t}{\rho^{\theta_i}}\right)\leq  m_i \epsilon^{\frac{1}{m_i}}S_{j+1}+(1-m_i)\left(2^{\frac{1}{1-m_i}}\right)^j\epsilon^{-\frac{1}{1-m_i}}\rho^{\frac{1}{1-m_i}}\left(\frac{t}{\rho^{\theta_i}}\right)^{\frac{1}{1-m_i}},
	\end{equation}
	\begin{equation}\label{eq-split-by-Youngs-inequality-S-m-i-j-plus-1-and-t-over-rho-theta-power-02}
	4^jS_{j+1}^{m_i}\left(\frac{t}{\rho^{\theta_i}}\right)\leq  m_i \epsilon^{\frac{1}{m_i}}S_{j+1}+(1-m_i)\left(4^{\frac{1}{1-m_i}}\right)^j\epsilon^{-\frac{1}{1-m_i}}\left(\frac{t}{\rho^{\theta_i}}\right)^{\frac{1}{1-m_i}}
	\end{equation}
	and
	\begin{equation}\label{eq-split-by-Youngs-inequality-S-m-i-j-plus-1-and-t-over-rho-theta-power-03}
	2^jS_{j+1}^{\frac{m_i+1}{2}}\left(\frac{t}{\rho^{\theta_i}}\right)^{\frac{1}{2}}\leq  \frac{m_i+1}{2}\epsilon^{\frac{2}{m_i+1}} S_{j+1}+\frac{1-m_i}{2}\left(4^{\frac{1}{1-m_i}}\right)^j\epsilon^{\frac{2}{1-m_i}}\left(\frac{t}{\rho^{\theta_i}}\right)^{\frac{1}{1-m_i}},
	\end{equation}
	for any $\epsilon\in\left(0,1\right)$. By \eqref{eq-inequality-for-iteration-S-i-j-0549t}, \eqref{eq-split-by-Youngs-inequality-S-m-i-j-plus-1-and-t-over-rho-theta-power-01}, \eqref{eq-split-by-Youngs-inequality-S-m-i-j-plus-1-and-t-over-rho-theta-power-02} and \eqref{eq-split-by-Youngs-inequality-S-m-i-j-plus-1-and-t-over-rho-theta-power-03}, we have
	\begin{equation*}
	\begin{aligned}
   S_j^i&\leq \gamma_0\left(\left(\textbf{C}_4+\sqrt{\textbf{C}_2+\textbf{C}_4}+\sqrt{1+\textbf{C}_4}\right)\epsilon^{\frac{1}{m_i}}+\epsilon^{\frac{2}{1+m_i}}\right)S_{j+1}\\
   &\qquad \qquad +\gamma_0\left(\frac{1}{\epsilon^{\frac{1}{1-m_i}}}+\frac{1}{\epsilon^{\frac{2}{1-m_i}}}\right)\left(4^{\frac{1}{1-m_i}}\right)^{j}\left(I^i+\left(\sqrt{1+\textbf{C}_4}+\left(\textbf{C}_4+\sqrt{\textbf{C}_2+\textbf{C}_4}\right)\rho^{\frac{1}{1-m_i}}\right)\left(\frac{t}{\rho^{\theta_i}}\right)^{\frac{1}{1-m_i}}\right)\\
   &:= \epsilon_0S_{j+1}^i+\gamma_0\left(\frac{1}{\epsilon^{\frac{1}{1-m_i}}}+\frac{1}{\epsilon^{\frac{2}{1-m_i}}}\right)\left(I^i+\left(\sqrt{1+\textbf{C}_4}+\left(\textbf{C}_4+\sqrt{\textbf{C}_2+\textbf{C}_4}\right)\rho^{\frac{1}{1-m_i}}\right)\left(\frac{t}{\rho^{\theta_i}}\right)^{\frac{1}{1-m_i}}\right)\left(4^{\frac{1}{1-m_i}}\right)^{j}.
   \end{aligned}	
	\end{equation*}
	By iteration,
	\begin{equation}\label{eq-iteration-for-conclusion-sub-bounded-by-inf-and-some-term}
	\begin{aligned}
	S_0^i\leq \epsilon_0^{j}S^i_j++\gamma\left(\alpha,\frac{1}{\epsilon_0}\right)\left(I^i+\left(\sqrt{1+\textbf{C}_4}+\left(\textbf{C}_4+\sqrt{\textbf{C}_2+\textbf{C}_4}\right)\rho^{\frac{1}{1-m_i}}\right)\left(\frac{t}{\rho^{\theta_i}}\right)^{\frac{1}{1-m_i}}\right)\sum_{l=0}^{j-1}\left(\epsilon_04^{\frac{1}{1-m_i}}\right)^l.
	\end{aligned}
	\end{equation}
	Now we choose the constant $\epsilon_0$ so small that the last term is bounded by a convergence series. Then,
	letting $j\to\infty$ in \eqref{eq-iteration-for-conclusion-sub-bounded-by-inf-and-some-term} we have the inequality \eqref{eq-conclusion-sup-bounded-by-inf-and-term-of-radius=time} and the lemma follows. 
\end{proof}

Through the De Giorgi iteration, we will get the following $L^{\infty}$ estimate of the system \eqref{eq-main-for-L-1-mass-conservation-without-vector-B}.

\begin{lemma}\label{lem-the-second-step-for-controlling-sup-u-i-byu-int-of-yu-i-and-ration-t-and-raduis}
	Let $\alpha\in\left(-1,0\right)$ be a constant given by \eqref{eq-condition-of-alpha-negative-small-engothsh-012} and let $0<m<1$ satisfy
	\begin{equation*}
	1-\frac{1}{\beta_i}<m<1, \qquad \qquad \theta_i=n\left(m_i-1\right)+2=n\beta_i(m-1)+2>0.
	\end{equation*}
	For $T>0$, let $\bold{u}=\left(u^1,\cdots,u^k\right)$ be a weak solution to the singular parabolic system \eqref{eq-standard-form-for-local-continuity-estimates-intro-1} in $E_T$ with structures \eqref{eq-standard-form-for-local-continuity-estimates-intro-2}, \eqref{eq-condition-for-U-primary--1}-\eqref{eq-condition-3-for-measurable-functions-mathcal-A-and-mathcal-B}. Suppose that there exists a constant $\Lambda>0$ such that
	\begin{equation*}
	U(x,t)\leq \Lambda \qquad \mbox{a.e. on $E_T$}.
	\end{equation*}
	Then, there exists a positive constant $\gamma$ depending on the data $\alpha$, $m$, $N$, $\textbf{c}$ and $\textbf{C}_3$ such that for all cylinder $B_{2\rho}(y)\times(2s-t,t]\subset E_T$
	\begin{equation}\label{eq-conclusion-sup-bounded-by-inf-and-term-of-radius=time-1}
	\sup_{B_{\frac{\rho}{2}}(y)\times\left(s,t\right]}u^i\leq \gamma\left(2+\left(\textbf{C}_2+\textbf{C}_4\right)\rho^2\right)^{\frac{n+2}{\theta_i}}\left(\frac{\rho^2}{t-s}\right)^{\frac{n}{\theta_i}}\left(1+\frac{\rho^2}{t-s}\right)^{\frac{n+2}{\theta_i}}\left(\frac{1}{\rho^n(t-s)}\int_{2s-t}^{t}\int_{B_{\rho}(y)}u^i(x,\tau)\,dxd\tau\right)^{\frac{2}{\theta_i}}+\left(\frac{t-s}{\rho^{2}}\right)^{\frac{1}{\beta_i(1-m)}}.
	\end{equation}
\end{lemma}

\begin{proof}
	 We wil use a modification of the proof of Proposition B.4.1 of \cite{DGV1} to prove the lemma.  Without loss of generality we let $(y,s)=(0,0)$. For fixed $\sigma\in(0,1)$ and each $j\in\N$, set
	\begin{equation*}
	\rho_j=\sigma\rho+\frac{1-\sigma}{2^j}\rho,\qquad t_n=-\sigma t-\frac{1-\sigma}{2^j}t,\qquad B_j=B_{\rho_j}, \qquad Q_j=B_j\times(t_j,t).
	\end{equation*}
	Then we first observe that
	\begin{equation*}
	Q_0=B_{\rho}\times(-t,t) \qquad \mbox{and} \qquad Q_{\infty}=B_{\sigma\rho}\times(-\sigma t,t).
	\end{equation*}
	Set
	\begin{equation*}
	M^i=\esssup_{Q_0}\max\left\{u^i,0\right\},\qquad M_{\sigma}^i=\esssup_{Q_{\infty}}\max\left\{u^i,0\right\}.
	\end{equation*}
     Let  $\xi_j(x,t)=\xi_j^1(x)\xi_j^2(t)$ be a nonnegative, piecewise smooth cut-off function in $Q_j$ such that
		\begin{equation*}
	\begin{cases}
	\begin{array}{cccl}
	\xi_j^1=1&&&\mbox{in $ B_{j+1}$}\\
	\xi_j^1=0&&&\mbox{on $\R^n\bs B_j$}\\
	\left|\nabla\xi_j^1\right|\leq \frac{2^{j+1}}{(1-\sigma)\rho}&&&\mbox{in $\R^n$}.
	\end{array}
	\end{cases}
	\qquad \mbox{and}\qquad 
	\begin{cases}
	\begin{array}{cccl}
	\xi_j^2=0&&&\mbox{for $t\leq t_j$}\\
	\xi_j^2=1&&&\mbox{for $t\geq t_{j+1}$}\\
	\left|\left(\xi_j^2\right)_t\right|\leq \frac{2^{j+1}}{(1-\sigma)\,t}&&&\mbox{for $t\in\R$}.
	\end{array}
	\end{cases}
	\end{equation*}
	Then $\xi_j$ equals one on $Q_{j+1}$. Consider the increasing sequence
	\begin{equation*}
	l_j=\left(1-\frac{1}{2^j}\right)l
	\end{equation*}
	where $k>2$ will be determined later. Let 
	\begin{equation*}
	m_i=\beta_{i}(m-1)+1.
	\end{equation*}
	Multiply the equation in \eqref{eq-main-for-L-1-mass-conservation-without-vector-B} by $\left(\left(u^i\right)^{m_i}-l_{j+1}^{m_i}\right)_+\xi_j^2$ and integrating it over $Q_j$, we have
	\begin{equation}\label{eq-after-with-cut-off-for-the-second-step-sup-u-bounded-byu-int-u-and-ratio-t-and-rho}
	\begin{aligned}
	0&=\int_{Q_j}\left(u^i\right)_{\tau}\left(\left(u^i\right)^{m_i}-l_{j+1}^{m_i}\right)_+\xi_j^2\,dxd\tau+\int_{Q_j}mU^{m-1}\left(\mathcal{A}\left(\nabla u^i,u ^i,x,t\right)\cdot\nabla\left(\left(u^i\right)^{m_i}-l_{j+1}^{m_i}\right)_+\right)\xi_j^2\,dxd\tau\\
	&\qquad \qquad +2\int_{Q_{j}}mU^{m-1}\left(\left(u^i\right)^{m_i}-l_{j+1}^{m_i}\right)_+\left(\mathcal{A}\left(\nabla u^i,u ^i,x,t\right)\cdot\nabla\xi_j\right)\xi_j\,dxd\tau\\
	&=I_1+I_2+I_3.
	\end{aligned}
	\end{equation}
	By direct computation,
	\begin{equation}\label{eq-first-one-for-the-second-step-sup-u-bounded-byu-int-u-and-ratio-t-and-rho}
	\begin{aligned}
	I_1&=\frac{1}{{m_i}}\frac{\partial}{\partial \tau}\left[\int_{Q_j}\left(\int_{0}^{\left(\left(u^i\right)^{m_i}-l_{j+1}^{m_i}\right)_+}\left(s+l_{j+1}^{m_i}\right)_+^{\frac{1}{{m_i}}-1}s\,ds\right)\xi_j^2\,dxd\tau\right]\\
	&\qquad \qquad -\frac{2}{m_i}\int_{Q_j}\left(\int_{0}^{\left(\left(u^i\right)^{m_i}-l_{j+1}^{m'}\right)_+}\left(s+l_{j+1}^{m_i}\right)_+^{\frac{1}{{m_i}}-1}s\,ds\right)\xi_j\left(\xi_j\right)_{\tau}\,dxd\tau\\
	&\geq \frac{1}{m_i}\int_{B_j}\left(\int_{0}^{\left(\left(u^i\right)^{m_i}-l_{j+1}^{m_i}\right)_+}s^{\frac{1}{m_i}}\,ds\right)\xi_j^2\,dx(t)\\
	&\qquad \qquad -\frac{2}{m_i}\int_{Q_j}\left(\int_{0}^{\left(\left(u^i\right)^{m_i}-l_{j+1}^{m_i}\right)_+}\left(s+l_{j+1}^{m_i}\right)_+^{\frac{1}{m_i}}\,ds\right)\xi_j\left|\left(\xi_j\right)_{\tau}\right|\,dxd\tau\\
	&\geq  \frac{1}{m_i+1}\int_{B_j}\left(\left(u^i\right)^{m_i}-l_{j+1}^{m_i}\right)_+^{\frac{{m_i}+1}{{m_i}}}\xi_j^2\,dx(t)\\
	&\qquad \qquad -\frac{2}{m_i}\int_{Q_j}\left(u^i\right)^{m_i+1}\,\chi_{_{\left\{u^i>l_{j+1}\right\}}}\xi_j\left|\left(\xi_j\right)_{\tau}\right|\,dxd\tau.
	\end{aligned}
	\end{equation}
	Observe that
	\begin{equation*}
	u^i\geq \frac{l}{2} \qquad \mbox{on $\left\{\left(u^i-l_{j+1}\right)_+>0\right\}$}.
	\end{equation*}
	Thus, we can get
	\begin{equation}\label{eq-second-one-for-the-second-step-sup-u-bounded-byu-int-u-and-ratio-t-and-rho}
	\begin{aligned}
	I_2&\geq \frac{m\textbf{c}\Lambda^{1-m}l^{1-m_i}}{4m_i}\int_{Q_j}\left|\nabla\left(\left(u^i\right)^{m_i}-l_{j+1}^{m_i}\right)_+\xi_j\right|^2\,dxd\tau -\frac{m\textbf{c}\Lambda^{1-m}}{m_i}\int_{Q_j}\left(\left(u^i\right)^{m_i}-l_{j+1}^{m_i}\right)_+^2\left|\nabla\xi_j\right|^2\,dxd\tau\\
	&\qquad \qquad -\frac{mm_i\textbf{C}_2}{\left(\lambda_i\right)^{1-m}}\int_{Q_j}\left(u^i\right)^{2m_i}\xi_j^2\,\chi_{_{\left\{u^i>k_{j+1}\right\}}}\,dxd\tau
	\end{aligned}
	\end{equation}
	and
	\begin{equation}\label{eq-third-one-for-the-second-step-sup-u-bounded-byu-int-u-and-ratio-t-and-rho}
	\begin{aligned}
\left|I_3\right|&\leq \frac{m\textbf{c}\Lambda^{1-m}}{8m_i}\int_{Q_j}\left|\nabla\left(\left(u^i\right)^{m_i}-k_{j+1}^{m_i}\right)_+\xi_j\right|^2\,dxd\tau+\frac{m\textbf{C}_4}{\left(\lambda_i\right)^{1-m}}\int_{Q_j}\left(u^i\right)^{2m_i}\xi_j^2\,\chi_{_{\left\{u^i>k_{j+1}\right\}}}\,dxd\tau\\
	&\qquad \qquad +\frac{m}{\left(\lambda_i\right)^{1-m}}\left(\frac{2\textbf{C}_3}{m_i}+\textbf{C}_4+\frac{8\textbf{C}_3^2}{m_i\textbf{c}\Lambda^{1-m}\lambda^{1-m}l^{1-m_i}}\right)\int_{Q_j}\left(\left(u^i\right)^{m_i}-l_{j+1}^{m_i}\right)_+^2\left|\nabla\xi_j\right|^2\,dxd\tau.
	\end{aligned}
	\end{equation}
	By \eqref{eq-after-with-cut-off-for-the-second-step-sup-u-bounded-byu-int-u-and-ratio-t-and-rho},  \eqref{eq-first-one-for-the-second-step-sup-u-bounded-byu-int-u-and-ratio-t-and-rho}, \eqref{eq-second-one-for-the-second-step-sup-u-bounded-byu-int-u-and-ratio-t-and-rho} and \eqref{eq-third-one-for-the-second-step-sup-u-bounded-byu-int-u-and-ratio-t-and-rho}, there exists a constant $C_1>0$ depending $m$, $\lambda_i$, $\beta_i$, $\textbf{c}$, $\textbf{C}_3$, $\textbf{C}_4$ and $\Lambda$ such that
	\begin{equation}\label{eq-simplify-with-main-first-t0-third-one-for-the-second-step-sup-u-bounded-byu-int-u-and-ratio-t-and-rho}
	\begin{aligned}
	&\sup_{t_j\leq \tau\leq t}\int_{B_{j}}\left[\left(\left(u^i\right)^{m_i}-l_{j+1}^{m_i}\right)_+\xi_j\right]^{\frac{{m_i}+1}{{m_i}}}\,dxd\tau+\int_{Q_j}\left|\nabla\left(\left(u^i\right)^{m_i}-l_{j+1}^{m_i}\right)_+\xi_j\right|^2\,dxd\tau\\
	&\qquad \leq C_1 \left(1+\frac{1}{l^{1-m_i}}\right)\Bigg[\int_{Q_j}\left(u^i\right)^{m_i+1}\,\chi_{_{\left\{u^i>l_{j+1}\right\}}}\xi_j\left|\left(\xi_j\right)_{\tau}\right|\,dxd\tau+\int_{Q_j}\left(\left(u^i\right)^{m_i}-l_{j+1}^{m_i}\right)_+^2\left|\nabla\xi_j\right|^2\,dxd\tau\\
	&\qquad \qquad \qquad \qquad \qquad \qquad +\int_{Q_j}\left(u^i\right)^{2m_i}\xi_j^2\,\chi_{_{\left\{u^i>l_{j+1}\right\}}}\,dxd\tau\Bigg].
	\end{aligned}
	\end{equation}
	Since
	\begin{equation*}
	\left(\left(u^i\right)^{m_i}-l_{j}^{m_i}\right)_+ \geq \frac{1}{C_2}\left(\frac{l}{2^{j+1}}\right)^{m_i}\geq \frac{1}{C_2}\left(\frac{l_j}{2^{j+1}}\right)^{m_i}\qquad \mbox{on $\left\{u^i>l_{j+1}\right\}$}
	\end{equation*}
	for some constant $C_2>0$, we can get
	\begin{equation}\label{eq-upperbound-of-right-hand-side-term-of-engery-by-positive-to-1-1-over-m-i}
	\begin{aligned}
	&\int_{Q_j}\left(u^i\right)^{m_i+1}\,\chi_{_{\left\{u^i>k_{j+1}\right\}}}\xi_j\left|\left(\xi_j\right)_{\tau}\right|\,dxd\tau\leq C_3 \frac{2^{\left(\frac{m_i+1}{m_i}\right)\left(j+2\right)}}{\left(1-\sigma\right)t}\int_{Q_j}\left(\left(u^i\right)^{m_i}-l_{j}^{m_i}\right)_+^{\frac{m_i+1}{m_i}}\,dxd\tau,\\
&\int_{Q_j}\left(\left(u^i\right)^{m_i}-l_{j+1}^{m_i}\right)_+^2\left|\nabla\xi_j\right|^2\,dxd\tau \leq C_3\frac{2^{\left(\frac{m_i+1}{m_i}\right)\left(j+1\right)}}{\left(1-\sigma\right)^2\rho^2l^{1-m_i}}\int_{Q_j}\left(\left(u^i\right)^{m_i}-l_{j}^{m_i}\right)_+^{\frac{m_i+1}{m_i}}\,dxd\tau,\\
&\int_{Q_j}\left(u^i\right)^{2m_i}\xi_j^2\,\chi_{_{\left\{u^i>l_{j+1}\right\}}}\,dxd\tau\leq C_3\frac{2^{\left(\frac{m_i+1}{m_i}\right)\left(j+1\right)+2}}{l^{1-m_i}}\int_{Q_j}\left(\left(u^i\right)^{m_i}-l_{j}^{m_i}\right)_+^{\frac{m_i+1}{m_i}}\,dxd\tau
	\end{aligned}
	\end{equation}
	for some constant $C_3>0$.	 By \eqref{eq-simplify-with-main-first-t0-third-one-for-the-second-step-sup-u-bounded-byu-int-u-and-ratio-t-and-rho} and \eqref{eq-upperbound-of-right-hand-side-term-of-engery-by-positive-to-1-1-over-m-i}, there exists a constant $C_4>0$ such that
	\begin{equation*}
	\begin{aligned}
	&\sup_{t_j\leq \tau\leq t}\int_{B_{j}}\left[\left(\left(u^i\right)^{m_i}-l_{j+1}^{m_i}\right)_+\xi_j\right]^{\frac{{m_i}+1}{{m_i}}}\,dxd\tau+\int_{Q_j}\left|\nabla\left(\left(u^i\right)^{m_i}-l_{j+1}^{m_i}\right)_+\xi_j\right|^2\,dxd\tau\\
	&\qquad \qquad \leq C_4\left(1+\frac{1}{l^{1-m_i}}\right)\frac{2^{\left(\frac{m_i+1}{m_i}\right)j}}{(1-\sigma)^2t}\left(1+\left(\frac{t}{\rho^2}\right)l^{m_i-1}+\left(\textbf{C}_2+\textbf{C}_4\right)tl^{m_i-1}\right)\int_{Q_j}\left(\left(u^i\right)^{m_i}-l_{j}^{m_i}\right)_+^{\frac{m_i+1}{m_i}}\,dxd\tau
	\end{aligned}
	\end{equation*}
	since $\frac{n-2}{n}<m_i<1$. If we choose the constant $l$ to be
	\begin{equation*}
	l\geq\left(\frac{t}{\rho^2}\right)^{\frac{1}{1-m}},
	\end{equation*}
	then
	\begin{equation}\label{eq-energy-thpeun-inequalyti-for-u-m-i--k-j-m-i-to-1-1-over-m-i}
	\begin{aligned}
	&\sup_{t_j\leq \tau\leq t}\int_{B_{j}}\left[\left(\left(u^i\right)^{m_i}-l_{j+1}^{m_i}\right)_+\xi_j\right]^{\frac{{m_i}+1}{{m_i}}}\,dxd\tau+\int_{Q_j}\left|\nabla\left(\left(u^i\right)^{m_i}-l_{j+1}^{m_i}\right)_+\xi_j\right|^2\,dxd\tau\\
	&\qquad \qquad \leq C_4\frac{2^{\left(\frac{m_i+1}{m_i}\right)j}}{(1-\sigma)^2t}\left(1+\frac{\rho^2}{t}\right)\left(2+\left(\textbf{C}_2+\textbf{C}_4\right)\rho^2\right)\int_{Q_j}\left(\left(u^i\right)^{m_i}-l_{j}^{m_i}\right)_+^{\frac{m_i+1}{m_i}}\,dxd\tau.
	\end{aligned}
	\end{equation}
	On the other hand, by H\"older inequality and Proposition 3.1 of  Chap. I of \cite{Di}, there exists a constant $C_5>0$ such that
	\begin{equation}\label{eq-inequality-from-Holder-and-weighted-Sobolev-inequality-andf-energy-inequality-99}
	\begin{aligned}
	&\frac{1}{\left|Q_{j+1}\right|}\int_{Q_{j+1}}\left(\left(u^i\right)^{m_i}-l_{j}^{m_i}\right)_+^{\frac{m_i+1}{m_i}}\,dxd\tau\\
	&\qquad \leq C_5\left(\sup_{t_j\leq \tau\leq t}\int_{B_{j}}\left[\left(\left(u^i\right)^{m_i}-l_{j+1}^{m_i}\right)_+\xi_j\right]^{\frac{{m_i}+1}{{m_i}}}\,dxd\tau\right)^{\frac{2}{n}\left(\frac{m_i+1}{pm_i}\right)}\times\left(\int_{Q_j}\left|\nabla\left(\left(u^i\right)^{m_i}-l_{j+1}^{m_i}\right)_+\xi_j\right|^2\,dxd\tau\right)^{\frac{m_i+1}{pm_i}}\\
	&\qquad \qquad \qquad \times\frac{1}{\left|Q_j\right|^{\frac{m_i+1}{pm_i}}}\left(\frac{2^{\left(\frac{m_i+1}{m_i}\right)j}}{l^{m_i+1}}\frac{1}{\left|Q_j\right|}\int_{Q_j}\left(\left(u^i\right)^{m_i}-l_{j}^{m_i}\right)_+^{\frac{m_i+1}{m_i}}\,dxd\tau\right)^{1-\frac{m_i+1}{pm_i}}
	\end{aligned}
	\end{equation}
	where
	\begin{equation*}
	p=\frac{2\left(nm_i+m_i+1\right)}{nm_i}.
	\end{equation*}
	By \eqref{eq-energy-thpeun-inequalyti-for-u-m-i--k-j-m-i-to-1-1-over-m-i} and \eqref{eq-inequality-from-Holder-and-weighted-Sobolev-inequality-andf-energy-inequality-99}, there exists a constant $C_6>0$ such that
	\begin{equation*}
	\begin{aligned}
	A_{j+1}&\leq \frac{C_6}{\left|Q_j\right|^{\frac{m_i+1}{pm_i}}}\left(\frac{2^{\left(\frac{m_i+1}{m_i}\right)j}}{(1-\sigma)^2t}\left(1+\frac{\rho^2}{t}\right)\left(2+\left(\textbf{C}_2+\textbf{C}_4\right)\rho^2\right)\left|Q_j\right|A_j\right)^{\frac{2}{n}\left(\frac{m_i+1}{pm_i}\right)}\\
	&\qquad \qquad \times\left(\frac{2^{\left(\frac{m_i+1}{m_i}\right)j}}{(1-\sigma)^2t}\left(1+\frac{\rho^2}{t}\right)\left(2+\left(\textbf{C}_2+\textbf{C}_4\right)\rho^2\right)\left|Q_j\right|A_j\right)^{\frac{m_i+1}{pm_i}}\times\left(\frac{2^{\left(\frac{m_i+1}{m_i}\right)j}}{l^{m_i+1}}A_j\right)^{1-\frac{m_i+1}{pm_i}}\\
	&\leq \frac{C_6b_1^j}{\left(1-\sigma\right)^{\left(2+\frac{2}{n}\right)\left(\frac{m_i+1}{pm_i}\right)}l^{\frac{\left(m_i+1\right)\left(pm_i-m_i-1\right)}{pm_i}}}\left(2+\left(\textbf{C}_2+\textbf{C}_4\right)\rho^2\right)^{\left(1+\frac{2}{n}\right)\left(\frac{m_i+1}{pm_i}\right)}\left(\frac{\rho^2}{t}\right)^{\frac{m_i+1}{pm_i}}\left(1+\frac{\rho^2}{t}\right)^{\left(1+\frac{2}{n}\right)\left(\frac{m_i+1}{pm_i}\right)}A_j^{1+{\frac{2}{n}\left(\frac{m_i+1}{pm_i}\right)}}
	\end{aligned}
	\end{equation*}
	where
	\begin{equation*}
	b_1=2^{\left(\frac{m_i+1}{m_i}\right)\left(1+\frac{2}{n}\left(\frac{m_i+1}{pm_i}\right)\right)} \qquad \mbox{and} \qquad A_j=\frac{1}{\left|Q_j\right|}\int_{Q_j}\left(\left(u^i\right)^{m_i}-l_{j}^{m_i}\right)_+^{\frac{m_i+1}{m_i}}\,dxd\tau.
	\end{equation*}
	Let 
		\begin{equation*}
	C_7=C_6^{-\frac{1}{\frac{2}{n}\left(\frac{m_i+1}{pm_i}\right)}}, \qquad C_8=C_7^{-\frac{2}{n\left(m_i-1\right)+2m_i+2}},  \qquad b_2=b_1^{-\left(\frac{1}{{\frac{2}{n}\left(\frac{m_i+1}{pm_i}\right)}}\right)^2} \qquad \mbox{and} \qquad b_3=b_2^{-\frac{2}{n\left(m_i-1\right)+2m_i+2}}.
	\end{equation*}
	If we take the constant $l$ such that
	\begin{equation*}
	\begin{aligned}
	&A_0=\frac{1}{\left|Q_0\right|}\int_{Q_0}\left(u^i\right)^{m_i+1}\,dxd\tau\leq C_7b_2\left(1-\sigma\right)^{n+1}\left(2+\left(\textbf{C}_2+\textbf{C}_4\right)\rho^2\right)^{-\frac{n+2}{2}}\left(\frac{\rho^2}{t}\right)^{-\frac{n}{2}}\left(1+\frac{\rho^2}{t}\right)^{-\frac{n+2}{2}}l^{\frac{n\left(m_i-1\right)+2m_i+2}{2}}\\
	&\qquad \Rightarrow \qquad l\geq \frac{C_8b_3}{\left(1-\sigma\right)^{\frac{2(n+1)}{n\left(m_i-1\right)+2m_i+2}}}\left(2+\left(\textbf{C}_2+\textbf{C}_4\right)\rho^2\right)^{\frac{n+2}{n\left(m_i-1\right)+2m_i+2}}\left(\frac{\rho^2}{t}\right)^{\frac{n}{n\left(m_i-1\right)+2m_i+2}}\left(1+\frac{\rho^2}{t}\right)^{\frac{n+2}{n\left(m_i-1\right)+2m_i+2}}\\
	&\qquad \qquad \qquad \qquad\qquad \qquad  \times\left(\frac{1}{\left|Q_0\right|}\int_{Q_0}\left(u^i\right)^{m_i+1}\,dxd\tau\right)^{\frac{2}{n\left(m_i-1\right)+2m_i+2}},
	\end{aligned}
	\end{equation*}
then, by Lemma 5.1 in Chap. 2 of \cite{DGV1}, we have 
\begin{equation}\label{complete-of-ietation-A-j-to-zero=as-j-to-infty-for-mass-conservation}
A_{j}\to 0 \qquad \mbox{as $j\to\infty$}.
\end{equation}
If
\begin{equation*}
	\begin{aligned}
&\frac{C_8b_3}{\left(1-\sigma\right)^{\frac{2(n+1)}{n\left(m_i-1\right)+2m_i+2}}}\left(2+\left(\textbf{C}_2+\textbf{C}_4\right)\rho^2\right)^{\frac{n+2}{n\left(m_i-1\right)+2m_i+2}}\left(\frac{\rho^2}{t}\right)^{\frac{n}{n\left(m_i-1\right)+2m_i+2}}\\
&\qquad \qquad \qquad \qquad \times\left(1+\frac{\rho^2}{t}\right)^{\frac{n+2}{n\left(m_i-1\right)+2m_i+2}}\left(\frac{1}{\left|Q_0\right|}\int_{Q_0}\left(u^i\right)^{m_i+1}\,dxd\tau\right)^{\frac{2}{n\left(m_i-1\right)+2m_i+2}}\leq \left(\frac{t}{\rho^2}\right)^{\frac{1}{1-m}},
\end{aligned}
\end{equation*}
then we take 
\begin{equation*}
l=\left(\frac{t}{\rho^2}\right)^{\frac{1}{1-m}}.
\end{equation*}
Then by \eqref{complete-of-ietation-A-j-to-zero=as-j-to-infty-for-mass-conservation} we have
\begin{equation*}
u^i\leq l=\left(\frac{t}{\rho^2}\right)^{\frac{1}{1-m}} \qquad \mbox{on $Q_{\infty}$}
\end{equation*}
and \eqref{eq-conclusion-sup-bounded-by-inf-and-term-of-radius=time-1} holds with $\sigma=\frac{1}{2}$.\\
\indent Otherwise, we take
\begin{equation*}
	\begin{aligned}
l=&\frac{C_8b_3}{\left(1-\sigma\right)^{\frac{2(n+1)}{n\left(m_i-1\right)+2m_i+2}}}\left(2+\left(\textbf{C}_2+\textbf{C}_4\right)\rho^2\right)^{\frac{n+2}{n\left(m_i-1\right)+2m_i+2}}\left(\frac{\rho^2}{t}\right)^{\frac{n}{n\left(m_i-1\right)+2m_i+2}}\\
&\qquad \qquad \qquad \qquad \times\left(1+\frac{\rho^2}{t}\right)^{\frac{n+2}{n\left(m_i-1\right)+2m_i+2}}\left(\frac{1}{\left|Q_0\right|}\int_{Q_0}\left(u^i\right)^{m_i+1}\,dxd\tau\right)^{\frac{2}{n\left(m_i-1\right)+2m_i+2}}.
\end{aligned}
\end{equation*}
Then
\begin{equation}\label{bound-of-M-sigma-in-second-alternatlive}
	\begin{aligned}
M^i_{\sigma}\leq&\frac{C_8b_3}{\left(1-\sigma\right)^{\frac{2(n+1)}{n\left(m_i-1\right)+2m_i+2}}}\left(2+\left(\textbf{C}_2+\textbf{C}_4\right)\rho^2\right)^{\frac{n+2}{n\left(m_i-1\right)+2m_i+2}}\left(\frac{\rho^2}{t}\right)^{\frac{n}{n\left(m_i-1\right)+2m_i+2}}\\
&\qquad \qquad \qquad \qquad \times\left(1+\frac{\rho^2}{t}\right)^{\frac{n+2}{n\left(m_i-1\right)+2m_i+2}}\left(\frac{1}{\left|Q_0\right|}\int_{Q_0}\left(u^i\right)^{m_i+1}\,dxd\tau\right)^{\frac{2}{n\left(m_i-1\right)+2m_i+2}}.
\end{aligned}
\end{equation}
Set 
\begin{equation*}
\overline{\rho}_{j}=\sigma\rho+\left(1-\sigma\right)\rho\sum_{j'=1}^{j}\frac{1}{2^{j'}}, \qquad \overline{t}_j=-\sigma t-\left(1-\sigma\right)t\sum_{j'=1}^j\frac{1}{2^{j'}},
\end{equation*}
and
\begin{equation*}
\overline{Q}_j=B_{\overline{\rho}_j}\times\left(\overline{t}_j,t\right], \qquad \overline{Q}_{\infty}=B_{\rho}\times\left(-t,t\right], \qquad \overline{Q}_0=B_{\sigma\rho}\times\left(-\sigma t,t\right]
\end{equation*}
and
\begin{equation*}
\overline{M}_j^i=\esssup_{\overline{Q}_j}\max\left\{u^i,0\right\}.
\end{equation*}
By \eqref{bound-of-M-sigma-in-second-alternatlive} with $\sigma$ being replaced by $\sigma+\left(1-\sigma\right)\sum_{j'=1}^j2^{-j'}$,
\begin{equation*}
	\begin{aligned}
\overline{M}^i_{j}\leq& C_9 \frac{2^{\frac{2(n+1)j}{n\left(m_i-1\right)+2m_i+2}}}{\left(1-\sigma\right)^{\frac{2(n+1)}{n\left(m_i-1\right)+2m_i+2}}}\left(2+\left(\textbf{C}_2+\textbf{C}_4\right)\rho^2\right)^{\frac{n+2}{n\left(m_i-1\right)+2m_i+2}}\left(\frac{\rho^2}{t}\right)^{\frac{n}{n\left(m_i-1\right)+2m_i+2}}\\
&\qquad \qquad \qquad \qquad \times\left(1+\frac{\rho^2}{t}\right)^{\frac{n+2}{n\left(m_i-1\right)+2m_i+2}}\left(\frac{1}{\left|\overline{Q}_{\infty}\right|}\int_{\overline{Q}_{\infty}}u^i\,dxd\tau\right)^{\frac{2}{n\left(m_i-1\right)+2m_i+2}}\overline{M}_{j+1}^{1-\frac{\theta_i}{n\left(m_i-1\right)+2m_i+2}}
\end{aligned}
\end{equation*}
for some constant $C_9>0$. Then by an Interpolation Lemma (Lemma 5.2 in Chap. 2 of \cite{DGV1}),
\begin{equation*}
\overline{M}^i_0=\sup_{B_{\sigma\rho}\times\left(-\sigma t,t\right]}u^i\leq \frac{C_{10}}{\left(1-\sigma\right)^{\frac{2(n+1)}{\theta_i}}}\left(2+\left(\textbf{C}_2+\textbf{C}_4\right)\rho^2\right)^{\frac{n+2}{\theta_i}}\left(\frac{\rho^2}{t}\right)^{\frac{n}{\theta_i}}\left(1+\frac{\rho^2}{t}\right)^{\frac{n+2)}{\theta_i}}\left(\frac{1}{\left|\overline{Q}_{\infty}\right|}\int_{\overline{Q}_{\infty}}u^i\,dxd\tau\right)^{\frac{2}{\theta_i}}
\end{equation*}
for some constant $C_{10}>0$. This immediately implies the inequality \eqref{eq-conclusion-sup-bounded-by-inf-and-term-of-radius=time-1} and the lemma follows. 
\end{proof}

As a conseqeuence of Lemma \ref{lem-the-first-step-for-controlling-sup-u-i-byu-int-of-yu-i-and-ration-t-and-raduis} and Lemma \ref{lem-the-second-step-for-controlling-sup-u-i-byu-int-of-yu-i-and-ration-t-and-raduis}, we can get the following full version of Harnack type estimate for the system \eqref{eq-main-for-L-1-mass-conservation-without-vector-B}.

\begin{lemma}\label{thm-control-of-supremum-of-u-i-by=inf-of-integral0and-ratio-of-time-and-radius}
	Let $\alpha\in\left(-1,0\right)$ be a constant given by \eqref{eq-condition-of-alpha-negative-small-engothsh-012}and let $0<m<1$ satisfy
	\begin{equation*}
	1-\frac{1}{\beta_i}<m<1, \qquad \qquad \theta_i=n\left(m_i-1\right)+2=n\beta_i(m-1)+2>0.
	\end{equation*}
	For $T>0$, let $\bold{u}=\left(u^1,\cdots,u^k\right)$, ($u^i\geq 0$ is bounded for all $1\leq i\leq k$), be a weak solution to the singular parabolic system \eqref{eq-standard-form-for-local-continuity-estimates-intro-1}  in $E_T$ with structures \eqref{eq-standard-form-for-local-continuity-estimates-intro-2}, \eqref{eq-condition-for-U-primary--1}-\eqref{eq-condition-3-for-measurable-functions-mathcal-A-and-mathcal-B}. Then, there exists a positive constant $\gamma$ depending on the data $\alpha$, $m$, $N$ and $\textbf{c}$ such that for all cylinder $B_{4\rho}(y)\times(2s-t,t]\in E_T$
	\begin{equation}\label{eq-conclusion-sup-bounded-by-inf-and-term-of-radius=time-100}
	\begin{aligned}
	&\sup_{B_{\rho}(y)\times\left(s,t\right]}u^i\\
	&\qquad \leq \gamma\Bigg[\frac{\left(2+\left(\textbf{C}_2+\textbf{C}_4\right)\rho^2\right)^{\frac{n+2}{\theta_i}}\left(t-s+\rho^2\right)^{\frac{n+2}{\theta_i}}}{(t-s)^{\frac{2(n+1)}{\theta_i}}}\left(\inf_{2s-t<\tau<t}\int_{B_{4\rho}(y)}u^i(x,\tau)\,dxd\tau\right)^{\frac{2}{\theta_i}}\\
	&\qquad \qquad +\left(1+\frac{\left(2+\left(\textbf{C}_2+\textbf{C}_4\right)\rho^2\right)^{\frac{n+2}{\theta_i}}\left(t-s+\rho^2\right)^{\frac{n+2}{\theta_i}}}{(t-s)^{\frac{n+2}{\theta_i}}}\left(\sqrt{1+\textbf{C}_4}+\left(\textbf{C}_4+\sqrt{\textbf{C}_2+\textbf{C}_4}\right)\rho^{\frac{1}{\beta_i(1-m)}}\right)^{\frac{2}{\theta_i}}\right)\left(\frac{t-s}{\rho^{2}}\right)^{\frac{1}{\beta_i(1-m)}}\Bigg].
	\end{aligned}
	\end{equation}
\end{lemma}

We are now ready for the proof of Theorem \ref{thm-L-1-mass-conservation-super-critical-range}.

\begin{proof}[\textbf{Proof of Theorem \ref{thm-L-1-mass-conservation-super-critical-range}}]
	We will use a modification of the proof of Theorem 1.1 of \cite{FDV} to prove the theorem. By \eqref{eq-main-for-L-1-mass-conservation-without-vector-B} and divergence theorem,
	\begin{equation}\label{eq-firts-equation-in-proof-for-L-1-mass-conservation-with-divergence-thereom}
	\frac{d}{dt}\int_{B_{R}}u^i(x,t)\,dx=m\int_{\partial B_{R}}U^{m-1}\mathcal{A}\left(\nabla u^i,u^i,x,t\right)\cdot\nu\,dx	
	\end{equation}
	for any $R>0$ where $\nu$ is the unit outward normal vector on $\partial B_{R}$. Let $m_i=\beta_i(m-1)+1$. Then, by \eqref{eq-condition-1-02-for-measurable-functions-mathcal-A-and-mathcal-B} and \eqref{eq-firts-equation-in-proof-for-L-1-mass-conservation-with-divergence-thereom} we have
	\begin{equation}\label{eq-absoute-of-difference-L-1-mass-with-initial-L-1-mass-first}
	\begin{aligned}
	&\sup_{0\leq t\leq T}\left|\int_{B_{R}}u^i(x,t)\,dx-\int_{B_{R}}u_0^i(x)\,dx\right|\\
	&\qquad \qquad \leq m\left[\textbf{C}_3\int_{0}^{T}\int_{\partial B_{R}}U^{m-1}\left|\nabla u^i\right|\,d\sigma dt+\textbf{C}_4\int_{0}^{T}\int_{\partial B_{R}}U^{m-1}u^i\,d\sigma dt\right]\\
	&\qquad \qquad\leq m\left[\textbf{C}_3\int_{0}^{T}\int_{\partial B_{R}}U^{m-1}\left|\nabla u^i\right|\,d\sigma dt+\frac{\textbf{C}_4}{\left(\lambda_i\right)^{m-1}}\int_{0}^{T}\int_{\partial B_{R}}\left(u^i\right)^{m_i}\,d\sigma dt\right]\\
	&\qquad \qquad\leq R^{\frac{n-1}{2}}T^{\frac{1}{2}}\left[\textbf{C}_3\left(\int_{0}^{T}\int_{\partial B_{R}}\left(U^{m-1}\left|\nabla u^i\right|\right)^{2}\,d\sigma dt\right)^{\frac{1}{2}}+\textbf{C}_4\left(\int_{0}^{T}\int_{\partial B_{R}}\left(u^i\right)^{2m_i}\,d\sigma dt\right)^{\frac{1}{2}}\right].
	\end{aligned}
	\end{equation}
	Suppose that
	\begin{equation*}
	\textbf{supp}\,u^i_0\subset B_{R^{\ast}}
	\end{equation*}
	and let $R_0>\max\left\{8,\,8R^{\ast}\right\}$. Then
	\begin{equation}\label{eq-choice-of-R-0-from-support-of-u-i-0}
	\textbf{supp}\,u^i_0\subset B_{\frac{R_0}{8}}
	\end{equation}
	Let $\xi\in C^{\infty}\left(\R^n\right)$ be a cut-off function such that
	\begin{equation*}
	\begin{cases}
	\begin{array}{cccl}
	0\leq \xi\leq 1 &&&\mbox{in $\R^n$}\\
	\xi=1&&&\mbox{in $B_{2R_0}\bs B_{\frac{R_0}{2}}$}\\
	\xi=0&&&\mbox{in $B_{\frac{R_0}{4}}\cup\left\{\R^n\bs B_{4R_0}\right\}$}\\
	\left|\nabla\xi\right|\leq \frac{c_1}{R_0}&&&\mbox{in $\R^n$}
	\end{array}
	\end{cases}
	\end{equation*}
	for some constant $c_1>0$. Let $T>0$. Multiplying \eqref{eq-main-for-L-1-mass-conservation-without-vector-B} by $\left(u^i\right)^{m_i}\xi^2$ and integrating it over $\R^n\times[0,T]$, we have
	\begin{equation}\label{eq-energy-type-inequality-for-controllling-nablau-u-i-by-raduisu-01}
	\begin{aligned}
	&\frac{1}{1+m_i}\int_{\R^n}\left(u^i\right)^{1+m_i}(x,T)\xi^2\,\,dx-\frac{1}{1+m_i}\int_{\R^n}\left(u^i_0\right)^{1+m_i}\xi^2\,\,dx\\
	&\qquad \qquad = -mm_i\int_{0}^{T}\int_{\R^n}U^{m-1}\left(u^i\right)^{m_i-1}\left(\mathcal{A}\cdot\nabla u^i\right)\xi^2\,dxdt-2m\int_{0}^{T}\int_{\R^n}U^{m-1}\left(u^i\right)^{m_i}\left(\mathcal{A}\cdot\nabla \xi\right)\xi\,dxdt.
	\end{aligned}
	\end{equation}
	By \eqref{eq-condition-1-01-for-measurable-functions-mathcal-A-and-mathcal-B},  \eqref{eq-condition-1-02-for-measurable-functions-mathcal-A-and-mathcal-B}, \eqref{eq-choice-of-R-0-from-support-of-u-i-0}, \eqref{eq-energy-type-inequality-for-controllling-nablau-u-i-by-raduisu-01} and Young's inequality,
	\begin{equation*}
	\begin{aligned}
	&\frac{\left(\lambda_i\right)^{1-m}mm_i\textbf{c}}{2}\int_{0}^{T}\int_{\R^n}\left(U^{m-1}\left|\nabla u^i\right|\right)^2\xi^2\,dxdt\\
	&\qquad \qquad \leq \frac{m}{\left(\lambda_i\right)^{1-m}}\left(m_i\textbf{C}_2+\textbf{C}_4\right)\int_{0}^{T}\int_{B_{4R_0}\bs B_{\frac{R_0}{4}}}\left( u^i\right)^{2m_i}\xi^2\,dxdt\\
	&\qquad \qquad \qquad \qquad +\frac{m}{\left(\lambda_i\right)^{1-m}}\left(\textbf{C}_4+\frac{\textbf{C}_3^2}{\textbf{c}m_i}\right)\int_{0}^{T}\int_{B_{4R_0}\bs B_{2R_0}}\left( u^i\right)^{2m_i}\left|\nabla\xi\right|^2\,dxdt\\
	&\qquad \qquad \qquad \qquad \qquad \qquad  +\frac{m}{\left(\lambda_i\right)^{1-m}}\left(\textbf{C}_4+\frac{\textbf{C}_3^2}{\textbf{c}m_i}\right)\int_{0}^{T}\int_{B_{\frac{R_0}{2}}\bs B_{\frac{R_0}{4}}}\left( u^i\right)^{2m_i}\left|\nabla\xi\right|^2\,dxdt.
	\end{aligned}
	\end{equation*}
	Since $\left|\nabla\xi\right|\leq\frac{c_1}{R_0}$,
	\begin{equation}\label{eq-simplify-of-energy-type-inequality-using-condition-of-cut-off-function-and-support}
	\begin{aligned}
	&\int_{0}^{T}\int_{B_{2R_0}\bs B_{\frac{R_0}{2}}}\left(U^{m-1}\left|\nabla u^i\right|\right)^2\,dxdt\\
	&\qquad  \leq \frac{2}{\textbf{c}\left(\lambda_i\right)^{2(1-m)}}\left(\textbf{C}_2+\frac{\textbf{C}_4}{m_i}\right)\int_{0}^{T}\int_{B_{4R_0}\bs B_{\frac{R_0}{4}}}\left( u^i\right)^{2m_i}\,dxdt\\
	&\qquad \qquad  +\frac{2c_1}{m_i\textbf{c}\left(\lambda_i\right)^{2(1-m)}R_0^2}\left(\textbf{C}_4+\frac{\textbf{C}_3^2}{\textbf{c}m_i}\right)\left(\int_{0}^{T}\int_{B_{4R_0}\bs B_{2R_0}}\left( u^i\right)^{2m_i}\,dxdt+\int_{0}^{T}\int_{B_{\frac{R_0}{2}}\bs B_{\frac{R_0}{4}}}\left( u^i\right)^{2m_i}\,dxdt\right).
	\end{aligned}
	\end{equation}
    We now put $t=T$, $s=\frac{T}{2}$ and $\rho=\frac{|x|}{8}$ in \eqref{eq-conclusion-sup-bounded-by-inf-and-term-of-radius=time-100}. Then, there exists a constant $c_2>0$ such that
    \begin{equation}\label{eq-inequality-first-after-applying-inequality-341-in-Lemma-3-4}
    \begin{aligned}
   & u^i(x,T)\\
   &\qquad \leq \sup_{y\in B_{\frac{|x|}{8}}\left(x\right)}u^i(y,T)\\
    &\qquad \leq  c_2\Bigg[\frac{\left(2+\left(\textbf{C}_2+\textbf{C}_4\right)\left(\frac{|x|}{8}\right)^2\right)^{\frac{n+2}{\theta_i}}\left(\frac{T}{2}+\left(\frac{|x|}{8}\right)^2\right)^{\frac{n+2}{\theta_i}}2^{\frac{2(n+1)}{\theta_i}}}{T^{\frac{2(n+1)}{\theta_i}}}\left(\inf_{0<\tau<
    	T}\int_{B_{\frac{|x|}{2}}(x)}u^i(x,\tau)\,dxd\tau\right)^{\frac{2}{\theta_i}}\\
    &\qquad\qquad  +\left(1+\frac{\left(2+\left(\textbf{C}_2+\textbf{C}_4\right)\left(\frac{|x|}{8}\right)^2\right)^{\frac{n+2}{\theta_i}}\left(\frac{T}{2}+\left(\frac{|x|}{8}\right)^2\right)^{\frac{n+2}{\theta_i}}2^{\frac{n+2}{\theta_i}}}{T^{\frac{n+2}{\theta_i}}}\left(\sqrt{1+\textbf{C}_4}+\left(\textbf{C}_4+\sqrt{\textbf{C}_2+\textbf{C}_4}\right)\left(\frac{|x|}{8}\right)^{\frac{1}{\beta_i(1-m)}}\right)^{\frac{2}{\theta_i}}\right)\\
    &\qquad \qquad \qquad \qquad \qquad \qquad \times\left(\frac{T}{2\left(\frac{|x|}{8}\right)^{2}}\right)^{\frac{1}{\beta_i(1-m)}}\Bigg].
    \end{aligned}
    \end{equation}
  By \eqref{eq-choice-of-R-0-from-support-of-u-i-0} and \eqref{eq-inequality-first-after-applying-inequality-341-in-Lemma-3-4}, we have
	\begin{equation}\label{eq-upper-bound-of-u-at-x-T-by-ration-of-T-and-absolute-of-x-to-sudden-power}
		\begin{aligned}
	&u^i(x,T)\\
	&\qquad \leq c_3\left(1+\left(1+|x|^2\right)^{\frac{n+2}{\theta_i}}\left(1+\left(\textbf{C}_2+\textbf{C}_4\right)|x|^2\right)^{\frac{n+2}{\theta_i}}\left(\sqrt{1+\textbf{C}_4}+\left(\textbf{C}_4+\sqrt{\textbf{C}_2+\textbf{C}_4}\right)|x|^{\frac{1}{\beta_i(1-m)}}\right)^{\frac{2}{\theta_i}}\right)\left(\frac{1}{|x|^2}\right)^{\frac{1}{\beta_i(1-m)}}, \\
	&\qquad \qquad \qquad \qquad \qquad \qquad  \forall |x|\geq \frac{R_0}{4}.
	\end{aligned}
	\end{equation}
   for some constant $c_3=c_3(T)>0$. By \eqref{eq-simplify-of-energy-type-inequality-using-condition-of-cut-off-function-and-support} and \eqref{eq-upper-bound-of-u-at-x-T-by-ration-of-T-and-absolute-of-x-to-sudden-power}, there exists a constant $c_4>0$ depending on $c_1$, $c_2$, $c_3$ $\textbf{c}$, $m_i$, $m$, $\lambda_i$ and $T$ such that
   \begin{equation}\label{eqw-control-of-L-23-of-U-m-1-and-nabla-u-i-almaost-last}
   \begin{aligned}
   &\int_{0}^{T}\int_{B_{2R_0}\bs B_{\frac{R_0}{2}}}\left(U^{m-1}\left|\nabla u^i\right|\right)^2\,dxdt\\
   &\qquad \qquad \leq c_4\left[\left(\textbf{C}_2+\textbf{C}_4\right)R_0^{n-\frac{2\left(1+m_i\right)}{1-m_i}}R_0^2
    +\left(\textbf{C}_4+\textbf{C}_3^2\right)R_0^{n-\frac{2\left(1+m_i\right)}{1-m_i}}\right]\\
    &\qquad \qquad \qquad \qquad  \times\left[1+R_0^{\frac{4m_i(n+2)}{\theta_i}}\left(1+\left(\textbf{C}_2+\textbf{C}_4\right)R_0^2\right)^{\frac{2m_i(n+2)}{\theta_i}}\left(\sqrt{1+\textbf{C}_4}+\left(\textbf{C}_4+\sqrt{\textbf{C}_2+\textbf{C}_4}\right)R_0^{\frac{1}{1-m_i}}\right)^{\frac{4m_i}{\theta_i}}\right].
   \end{aligned}
   \end{equation}
   By \eqref{eqw-control-of-L-23-of-U-m-1-and-nabla-u-i-almaost-last} and Mean Value Theorem for Integrals, there exists $\frac{R_0}{2}\leq R_1\leq 2R_0$ such that
   \begin{equation*}
   \begin{aligned}
     &\int_{0}^{T}\int_{\partial B_{R_1}}\left(U^{m-1}\left|\nabla u^i\right|\right)^2\,d\sigma dt\\
     &\qquad \qquad \leq c_4\left[\left(\textbf{C}_2+\textbf{C}_4\right)R_0^{n-\frac{2\left(1+m_i\right)}{1-m_i}}R_0
     +\left(\textbf{C}_4+\textbf{C}_3^2\right)R_0^{n-1-\frac{2\left(1+m_i\right)}{1-m_i}}\right]\\
     &\qquad \qquad \qquad \qquad  \times\left[1+R_0^{\frac{4m_i(n+2)}{\theta_i}}\left(1+\left(\textbf{C}_2+\textbf{C}_4\right)R_0^2\right)^{\frac{2m_i(n+2)}{\theta_i}}\left(\sqrt{1+\textbf{C}_4}+\left(\textbf{C}_4+\sqrt{\textbf{C}_2+\textbf{C}_4}\right)R_0^{\frac{1}{1-m_i}}\right)^{\frac{4m_i}{\theta_i}}\right].
     \end{aligned}
   \end{equation*} 
   Combining this with \eqref{eq-absoute-of-difference-L-1-mass-with-initial-L-1-mass-first}, we have
   \begin{equation}\label{eq-just-before-L-1-mass-conservation-putting-all-in-01}
   \begin{aligned}
   &\sup_{0\leq t\leq T}\left|\int_{B_{R_1}}u^i(x,t)\,dx-\int_{B_{R_1}}u_0^i(x)\,dx\right|\\
   &\qquad\qquad \leq c_5R_0^{n-1-\frac{1+m_i}{1-m_i}}\left(1+\left(\textbf{C}_2+\textbf{C}_4\right)R_0\right)\left(1+R_0^{\frac{2m_i(n+2)}{\theta_i}}\left(1+\left(\textbf{C}_2+\textbf{C}_4\right)R_0^2\right)^{\frac{m_i(n+2)}{\theta_i}}\left(1+\left(\textbf{C}_2+\textbf{C}_4\right)R_0^{\frac{1}{1-m_i}}\right)^{\frac{2m_i}{\theta_i}}\right)\\
   &\qquad \qquad \leq c_6R_0^{n-\frac{1+m_i}{1-m_i}+\frac{4m_i(n+2)}{2-n\left(1-m_i\right)}+\frac{2m_i}{2\left(1-m_i\right)-n\left(1-m_i\right)^2}}=c_6R_0^{-\frac{\left(n^2+5n+8\right)m_i^2-2\left(n^2+n+4\right)m_i+\left(n-1\right)\left(n-2\right)}{2\left(1-m_i\right)-n\left(1-m_i\right)^2}}.
   \end{aligned}
   \end{equation}
   for some constant $c_5>0$ and $c_6>0$. Thus, if
   \begin{equation*}
   \frac{n^2+n+4+\sqrt{2n(7n+11)}}{n^2+5n+8}<m_i<1
   \end{equation*}
   then the right hand side of \eqref{eq-just-before-L-1-mass-conservation-putting-all-in-01} converges to zero as $R_1\to\infty$. Therefore, letting $
  R_1\to\infty$ in \eqref{eq-just-before-L-1-mass-conservation-putting-all-in-01} we have 
   \begin{equation*}
   \int_{\R^n}u^i(x,T)\,dx=\int_{\R^n}u_0^i(x)\,dx
   \end{equation*}
   for any $T>0$ and the theorem follows. 
\end{proof}

\begin{remark}
	\begin{enumerate}
	\item Suppose that the constants $\textbf{C}_2$ and $\textbf{C}_4$ in the structures \eqref{eq-condition-1-for-measurable-functions-mathcal-A-and-mathcal-B}-\eqref{eq-condition-1-02-for-measurable-functions-mathcal-A-and-mathcal-B} are all zeros. Then, by \eqref{eq-just-before-L-1-mass-conservation-putting-all-in-01} we have
	\begin{equation*}
		\begin{aligned}
	\sup_{0\leq t\leq T}\left|\int_{B_{R_1}}u^i(x,t)\,dx-\int_{B_{R_1}}u_0^i(x)\,dx\right|&\leq c_6R_0^{n-1-\frac{1+m_i}{1-m_i}+\frac{2m_i(n+2)}{2-n\left(1-m_i\right)}+\frac{2m_i}{2\left(1-m_i\right)-n\left(1-m_i\right)^2}}\\
	&=c_6R_0^{-\frac{\left(n^2+2n+4\right)m_i^2-2\left(n^2-n+3\right)m_i+\left(n-2\right)^2}{2\left(1-m_i\right)-n\left(1-m_i\right)^2}}.
	\end{aligned}
	\end{equation*}
	Therefore, $L^1$-mass conservation \eqref{eq-in-thm-for-L-1-mass-conservation-in-super-critical-case} holds for $\frac{n^2-n+3+\sqrt{7n^2+2n-7}}{n^2+2n+4}<m_i=\beta(m-1)+1<1$.
	\item Suppose that $U$ is equivalent to $\lambda_i\left(u^i\right)^{\beta_i}$, i.e., there exists some constants $0<c\leq C<\infty$ such that
	\begin{equation*}
	c\lambda_i\left(u^i\right)^{\beta_i}\leq U\leq C\lambda_i\left(u^i\right)^{\beta_i} \qquad \mbox{in $Q\left(R,\theta^{-\alpha_0}R^{2}\right)$}.
	\end{equation*}
	\end{enumerate}
Then the term $\left(1+|x|^2\right)^{\frac{n+2}{\theta_i}}$ in \eqref{eq-upper-bound-of-u-at-x-T-by-ration-of-T-and-absolute-of-x-to-sudden-power} can be replaced by an uniform constant. Therefore, if $U$ is equivalent to $\lambda_i\left(u^i\right)^{\beta_i}$, and constants $\textbf{C}_2$ and $\textbf{C}_4$ are all zeros, then $L^1$-mass conservation \eqref{eq-in-thm-for-L-1-mass-conservation-in-super-critical-case} holds for $\frac{n-2}{n}<m_i=\beta(m-1)+1<1$.
\end{remark}

\section{Local Continuity (Proof of Theorem \ref{eq-local-continuity-of-solution})}
\setcounter{equation}{0}
\setcounter{thm}{0}

In Section \ref{Uniform Boundedness of the function U}, we discussed the $L^{\infty}$ boundedness of the function $U$ which makes the diffusion coefficient under control. It is a very useful tool for investigating the regularity theories of solutions to the parabolic systems. With this observation, we are going to prove the local continuity  of the parabolic system \eqref{eq-standard-form-for-local-continuity-estimates-intro-1} under the structural assumptions \eqref{eq-standard-form-for-local-continuity-estimates-intro-2}, \eqref{eq-condition-for-U-primary--1}-\eqref{eq-condition-3-for-measurable-functions-mathcal-A-and-mathcal-B}, \eqref{eq-range-of-constant-q-in-forcing-term}.  We start by stating a  well-known result, Sobolev-type inequality, which plays an important role for the local continuity.

\begin{lemma}[cf. Lemma 3.1 of \cite{KL1}]\label{eq-Sobolev-Type-Inequality}
Let $\eta(x,t)$ be a cut-off function compactly supported in $B_r$ and let $u$ be a function defined in $\R^n\times(t_1,t_2)$ for any $t_2>t_1>0$. Then $u$ satisfies the following Sobolev inequalities:
\begin{equation}\label{eq-first-sobolev-inequality-1}
\left\|\eta u\right\|_{L^{\frac{2n}{n-2}}(\R^n)}\leq C\left\|\nabla (\eta u)\right\|_{L^2(\R^n)}
\end{equation}
and
\begin{equation}\label{eq-weighted-sobolev-inequality-for-holder}
\begin{aligned}
\|\eta\, u\|^2_{L^2(t_1,t_2;L^{2}(\R^n))}\leq C\Big(\sup_{t_1\leq t\leq t_2}\|\eta\, u\|^2_{L^{2}(\R^n)}+\|\D
(\eta\, u)\|^2_{L^2(t_1,t_2;L^2(\R^n))}\Big)\,
|\{\eta\, u>0\}|^{\frac{2}{n+2}}
\end{aligned}
\end{equation}
for some $C>0$.
\end{lemma}

For the local continuity of the parabolic system \eqref{eq-standard-form-for-local-continuity-estimates-intro-1}, we will use a modification of the technique introduced in \cite{Di}, \cite{KL1}, \cite{KL2} and \cite{HU}.  Choose a point $(x_0,t_0)\in\R^n\times(0,\infty)$ and a constant $R_0>0$ such that
\begin{equation*}
\left(x_0,t_0\right)+Q\left(R_0,R_0^{2-\epsilon}\right)=\left(x_0,t_0\right)+B_R\times\left(-R_0^{2-\epsilon},0\right)\subset \R^n\times(0,\infty)
\end{equation*} 
where $0<\epsilon<1$ is a small number which is determined by \eqref{eq-condition-of-epsilon-for-holder-continuity-at-start}. After translation, we may assume without loss of generality that 
\begin{equation*}
(x_0,t_0)=(0,0). 
\end{equation*}
By Theorem \ref{thm-Main-boundedness-of-diffusion-coefficients-of-generaized-equation}, there exists a constant $1<\Lambda<\infty$ such that
\begin{equation}\label{eq-solution-u-be-bounded-above-by-Lambda-1}
U(x,t)\leq \Lambda \qquad \forall (x,t)\in Q\left(R_0,R_0^{2-\epsilon}\right).
\end{equation}
Thus we can set, for each $1\leq i\leq k$,
\begin{equation*}
\left(\mu^i\right)^+=\esssup_{Q(R_0,R_0^{2-\epsilon})}u^i, \qquad \left(\mu^i\right)^-=\essinf_{Q(R_0,R_0^{2-\epsilon})}u^i, \qquad \omega^i=\osc_{Q(R_0,R_0^{2-\epsilon})} u^i=\left(\mu^i\right)^+-\left(\mu^i\right)^-.
\end{equation*}
By \eqref{eq-standard-form-for-local-continuity-estimates-intro-2}, the equation \eqref{eq-standard-form-for-local-continuity-estimates-intro-1} is non-degenerate on the region where $u^i>0$. Thus if $\left(\mu^i\right)^->0$ for some $1\leq i\leq k$, then the equation is uniformly parabolic in $Q\left(R_0,R_0^{2-\epsilon}\right)$. By standard regularity theory for the parabolic equation \cite{LSU}, the local H\"older continuity follows. Hence from now on, we assume that 
\begin{equation*}
\left(\mu^i\right)^-=0\qquad \forall 1\leq i\leq k.
\end{equation*}
If $\left(\mu^i\right)^+=0$ for some $1\leq i\leq k$, then
\begin{equation*}
u^i\equiv 0 \qquad \mbox{on $Q\left(R_0,R_0^{2-\epsilon}\right)$}.
\end{equation*}
This immediately implies the local continuity of solution $u^i$. Hence we also assume that
\begin{equation*}
\omega^i=\left(\mu^i\right)^+>0 \qquad \forall 1\leq i\leq k.
\end{equation*}
For the intrinsic scaling technique, we let 
\begin{equation*}
\omega_{_{M}}=\max_{1\leq i\leq k}\omega^i \qquad \mbox{and} \qquad \theta=\frac{\omega_{_{M}}}{4}
\end{equation*}
and construct the cylinder 
\begin{equation}\label{eq-construction-of-cylinder-with-some-constant-A-which-is-bigger}
Q\left(R,\theta^{-\alpha_0}R^2\right)=B_{R}\times\left(-\theta^{-\alpha_0}R^2,0\right) \qquad \left(\alpha_0=\beta_i(m-1)\right)
\end{equation}
where $\beta_i$ is given by \eqref{eq-standard-form-for-local-continuity-estimates-intro-2}. We will assume that the radius $0<R<R_0$ is sufficiently small that
\begin{equation}\label{eq-condition-between-R-and-theta-alpha--1}
\theta^{\alpha_{0}}>R^{\epsilon}.
\end{equation}
By \eqref{eq-construction-of-cylinder-with-some-constant-A-which-is-bigger} and \eqref{eq-condition-between-R-and-theta-alpha--1}, 
\begin{equation*}
Q\left(R,\theta^{-\alpha_{0}}R^2\right)\subset Q\left(R,R^{2-\epsilon}\right)\subset Q\left(R_0,R_0^{2-\epsilon}\right)
\end{equation*}
and
\begin{equation*}
\osc_{Q\left(R,\theta^{-\alpha_{0}}R^2\right)}u^i\leq \omega_{_M}=4\theta \qquad \forall 1\leq i\leq k.
\end{equation*}

\indent	For the proof of local continuity of the component $u^i$, two separated cases are considered. The first one is to find a parabolic cylinder of the form of \eqref{eq-construction-of-cylinder-with-some-constant-A-which-is-bigger} where $u^i$ is mostly large, and the other one is when such a cylinder cannot be found. In both cases, we are going to show that the (essential) oscillation of $u^i$ in a smaller cylinder decreases in a way that can be measured quantitatively.

\subsection{The First Alternative}

	Suppose that there exists a cylinder such that $u^i$ is mostly large. Then, through the first alternative, we will prove that the component $u^i$ is above a small level in a smaller cylinder. The statement of the first alternative is as follow.
\begin{lemma}\label{lem-the-first-alternative-for-holder-estimates}
There exists a positive number $\rho_0$ depending on $m$, $q$, $\lambda$, $\Lambda$ and $\frac{\Lambda}{\theta^{\,\beta_i}}$ such that if
\begin{equation}\label{eq-first-condition-of-small-region-of-lower-for-holder-estimates}
\left|\left\{\left(x,t\right)\in Q\left(R,\theta^{-\alpha_{0}}R^2\right):u^{\,i}(x,t)<\frac{\omega_{_{M}}}{2}\right\}\right|\leq\rho_0\left|Q\left(R,\theta^{-\alpha_{0}}R^2\right)\right| 
\end{equation}
then,
\begin{equation}\label{eq-conclusion-of-first-alternatives-half-of-difference}
u^{\,i}(x,t)>\frac{\omega_{_{M}}}{4}\, \qquad \forall (x,t)\in Q\left(\frac{R}{2},\theta^{-\alpha_{0}}\left(\frac{R}{2}\right)^2\right).
\end{equation}
\end{lemma}
\begin{proof}
For $j\in\N$, we set
\begin{equation*}
R_j=\frac{R}{2}+\frac{R}{2^{j}} \qquad \mbox{and} \qquad l_j=\left(\mu^i\right)_-+\left(\frac{\omega_{_{M}}}{4}+\frac{\omega_{_{M}}}{2^{j+1}}\right)=\frac{\omega_{_{M}}}{4}+\frac{\omega_{_{M}}}{2^{j+1}}.
\end{equation*}
Consider a cut-off function $\eta_j(x,t)\in C^{\infty}\left(\R^n\times\R\right)$ such that
\begin{equation*}
\begin{cases}
\begin{array}{cccl}
0\leq \eta_j\leq 1 &&& \mbox{in $Q\left(R_j,\theta^{\,-\alpha_{0}}R_j^2\right)$}\\
\eta_j=1 &&& \mbox{in$Q\left(R_{j+1},\theta^{\,-\alpha_{_0}}R_{j+1}^2\right)$ }\\
\eta_j=0 &&& \mbox{on $\left\{B_{R_j}\times\left\{t=-\theta^{\,-\alpha_{0}}R_j^2\right\}\right\}\cup \left\{\partial B_{R_j}\times\left[-\theta^{\,-\alpha_{0}}R_j^2,0\right]\right\}$}\\
\left|\nabla\eta_j\right|\leq \frac{2^{j+1}}{R},\quad \left|\left(\eta_j\right)_t\right|\leq \frac{2^{2(j+1)}\theta^{\,\alpha_{0}}}{R^2}&&& \mbox{in $Q\left(R_j,\theta^{\,-\alpha_{0}}R_j^2\right)$}
\end{array}
\end{cases}
\end{equation*}
Let $u=u^{\,i}$, $\beta_i=\beta$ and $\lambda_i=\lambda$ for the convenience. We first take $\vp=\left(u-l_j\right)_-\eta_j^2$ as a test function in the weak formulation \eqref{eq-identity--of-formula-for-weak-solution} and integrating it over $\left(-\theta^{-\alpha_{0}}R_j^2,t\right)$ for $t\in\left(-\theta^{-\alpha_{0}}R_j^2,0\right)$. Then, we have 
\begin{equation}\label{eq-energy-type-equation-with-Lebesgue-Steklov-everage}
\begin{aligned}
0&=\int_{-\theta^{-\alpha_{0}}R_j^2}^{t}\int_{B_{R_i}}u_t\left[\left(u-l_j\right)_-\eta_j^2\right]\,dxd\tau\\
&\qquad\qquad  +m\int_{-\theta^{-\alpha_{0}}R_j^2}^{t}\int_{B_{R_j}}U^{m-1}\mathcal{A}\left(\nabla u,u,x,t\right)\nabla\left[\left(u-l_j\right)_-\eta_j^2\right]\,dxd\tau\\
&\qquad \qquad \qquad\qquad+\int_{-\theta^{-\alpha_{0}}R_j^2}^{t}\int_{B_{R_j}}\mathcal{B}\left(u,x,t\right)\nabla\left[\left(u-l_j\right)_-\eta_j^2\right]\,dxd\tau\\
&:=I+II+III.
\end{aligned}
\end{equation}
Observe that
\begin{equation}\label{eq-range-of-u-and-u-sub-omega-l-j-minus}
u\leq \frac{\omega_{_M}}{2}=2\theta \qquad \mbox{on $\left\{u\leq l_j\right\}$},\qquad \qquad \left(u-l_j\right)_-\leq \frac{\omega_{_M}}{2}=2\theta.
\end{equation}

By \eqref{eq-range-of-u-and-u-sub-omega-l-j-minus}, we have
\begin{equation}\label{eq-controlling-the-first-opart-time-derivatives-to-sup}
\begin{aligned}
-I&=\frac{1}{2}\int_{-\theta^{-\alpha_{0}}R_j^2}^{t}\int_{B_{R_i}}\left[\left(u-l_j\right)^2_-\right]_t\eta_j^2\,dxd\tau\\
&=\frac{1}{2}\int_{B_{R_j}\times\left\{t\right\}}\left(u-l_j\right)_-^2\eta_j^2\,dx-\int_{-\theta^{-\alpha_{0}}R_j^2}^{t}\int_{B_{R_i}}\left(u-l_j\right)^2_-\eta_j\left|\left(\eta_j\right)_t\right|\,dxd\tau\\
&\geq \frac{1}{2}\int_{B_{R_j}\times\left\{t\right\}}\left(u-l_j\right)_-^2\eta_j^2\,dx-\frac{2^{2(j+1)}\theta^{\alpha_{0}}\left(2\theta\right)^2}{R^2}\int_{-\theta^{-\alpha_{0}}R_j^2}^{t}\int_{B_{R_j}}\chi_{\left[u\leq l_j\right]}\,dxd\tau.
\end{aligned}
\end{equation}
By \eqref{eq-condition-1-01-for-measurable-functions-mathcal-A-and-mathcal-B}, \eqref{eq-condition-1-02-for-measurable-functions-mathcal-A-and-mathcal-B}, \eqref{eq-range-of-u-and-u-sub-omega-l-j-minus}, \eqref{eq-range-of-lower-bvound-of-U-when-grad-u-omega-l-j-eqwual-ero} and Young's inequality
\begin{align}
-II&\geq m\int_{-\theta^{-\alpha_{0}}R_j^2}^{t}\int_{B_{R_j}}\eta_j^2U^{m-1}\left(\mathcal{A}\left(\nabla u,u,x,t\right)\cdot\nabla u\right)\chi_{_{\left[u\leq l_j\right]}}\,dxd\tau\notag\\
&\qquad \qquad -2m\int_{-\theta^{-\alpha_{0}}R_j^2}^{t}\int_{B_{R_j}}\eta_jU^{m-1}\left(u-l_j\right)_-\left|\mathcal{A}\left(\nabla u,u,x,t\right)\right|\left|\nabla\eta_j\right|\,dxd\tau\notag\\
&\geq m\textbf{c}\int_{-\theta^{-\alpha_{0}}R_j^2}^{t}\int_{B_{R_j}}U^{m-1}\left|\nabla\left(u-l_j\right)_-\right|^2\eta_j^2\,dxd\tau-m\textbf{C}_2l_j^2\int_{-\theta^{-\alpha_{0}}R_j^2}^{t}\int_{B_{R_j}}U^{m-1}\eta_j^2\,dxd\tau\notag\\
&\qquad  -2m\textbf{C}_3\int_{-\theta^{-\alpha_{0}}R_j^2}^{t}\int_{B_{R_j}}\eta_jU^{m-1}\left(u-l_j\right)_-\left|\nabla u\right|\left|\nabla\eta_j\right|\,dxd\tau\notag\\
&\qquad \qquad -2m\textbf{C}_4\int_{-\theta^{-\alpha_{0}}R_j^2}^{t}\int_{B_{R_j}}\eta_jU^{m-1}\left(u-l_j\right)_-u\left|\nabla\eta_j\right|\,dxd\tau\notag\\
&\geq\frac{m\textbf{c}}{2}\int_{-\theta^{-\alpha_{0}}R_j^2}^{t}\int_{B_{R_j}}U^{m-1}\left|\nabla\left(u-l_j\right)_-\right|^2\eta_j^2\,dxd\tau\notag\\
&\qquad \qquad-2m\left(\frac{\textbf{C}_3^2}{\textbf{c}}+\textbf{C}_4\right)\frac{2^{2(j+1)}\Lambda^{m-1}\left(2\theta\right)^2}{R^2}\int_{-\theta^{-\alpha_{0}}R_j^2}^{t}\int_{B_{R_j}}\chi_{\left[u\leq l_j\right]}\,dxd\tau\label{eq-estimates-for-III-forcing-term-1and-half}\\
&\qquad \qquad \qquad \qquad -m\textbf{C}_2\Lambda^{m-1}\left(2\theta\right)^2\int_{-\theta^{-\alpha_{0}}R_j^2}^{t}\int_{B_{R_j}}\chi_{_{\left[u\leq l_j\right]}}\,dxd\tau\notag
\end{align}
By \eqref{eq-condition-3-for-measurable-functions-mathcal-A-and-mathcal-B}, \eqref{eq-condition-fo-q-for-local-continuity-in-Thm} and \eqref{eq-range-of-u-and-u-sub-omega-l-j-minus}, we can get
\begin{equation}\label{eq-estimates-for-III-forcing-term-2}
\begin{aligned}
III&\leq \int_{-\theta^{-\alpha_{0}}R_j^2}^{t}\int_{B_{R_j}}\left|\mathcal{B}\left(u,x,t\right)\right|\left|\nabla\left(u-l_j\right)_-\right|\eta_j^2\,dxd\tau+\int_{-\theta^{-\alpha_{0}}R_j^2}^{t}\int_{B_{R_j}}\left|\mathcal{B}\left(u,x,t\right)\right|\left|\nabla\eta_j\right|\left(u-l_j\right)_-\eta_j\,dxd\tau\\
&\leq\frac{m\textbf{c}}{4}\int_{-\theta^{-\alpha_{0}}R_j^2}^{t}\int_{B_{R_j}}U^{m-1}\left|\nabla\left(u-l_j\right)_-\right|^2\eta_j^2\,dxd\tau\\
&\qquad \qquad+\left(\frac{1}{m\textbf{c}}+1\right)\frac{\textbf{C}_5^2\Lambda^{\frac{2(q-1)}{\underline{\beta}}}\left(2\theta\right)^2}{\lambda^{\frac{2(q-1)}{\underline{\beta}}}\Lambda^{m-1}}\int_{-\theta^{-\alpha_{0}}R_j^2}^{t}\int_{B_{R_j}}\chi_{_{\left[u\leq l_j\right]}}\,dxd\tau\\
&\qquad \qquad \qquad \qquad  +\frac{2^{2(j+1)}\Lambda^{m-1}\left(2\theta\right)^2}{R^2}\int_{-\theta^{-\alpha_{0}}R_j^2}^{t}\int_{B_{R_j}}\chi_{_{\left[u\leq l_j\right]}}\,dxd\tau
\end{aligned}
\end{equation}
where $\underline{\beta}=\min_{1\leq i\leq  k}\beta_i$. By \eqref{eq-energy-type-equation-with-Lebesgue-Steklov-everage}, \eqref{eq-controlling-the-first-opart-time-derivatives-to-sup}, \eqref{eq-estimates-for-III-forcing-term-1and-half} and \eqref{eq-estimates-for-III-forcing-term-2}, there exists a constant $C_0$ depending on $m$, $\lambda$, $\underline{\beta}$, $\textbf{C}_2$, $\textbf{C}_3$, $\textbf{C}_4$ and $\textbf{C}_5$ such that
\begin{equation}\label{eq-start-of-iteration-energy-type-ineqwuaitonly-of-u-l-j-minus-01}
\begin{aligned}
	&\sup_{-\theta^{-\alpha_{0}}R_j^2<t<0}\int_{B_{R_j}\times\left\{t\right\}}\left(u-l_j\right)_-^2\eta_j^2\,dx+\int_{-\theta^{-\alpha_{0}}R_j^2}^{0}\int_{B_{R_j}}U^{m-1}\left|\nabla\left(u-l_j\right)_-\right|^2\eta_j^2\,dxdt\\
	&\qquad \leq C_0\theta^2\left[\frac{2^{2\left(j+1\right)}\Lambda^{m-1}}{R^2}\int_{-\theta^{-\alpha_0}R_j^2}^{0}\int_{B_{R_j}}\chi_{\left[u\leq l_j\right]}\,dxdt+\left(\frac{\Lambda^{2(m-1)}+\Lambda^{\frac{2(q-1)}{\underline{\beta}}}}{\theta^{\alpha_0}}\right)\int_{-\theta^{-\alpha_0}R_j^2}^{0}\int_{B_{R_j}}\chi_{\left[u\leq l_j\right]}\,dxdt\right]
\end{aligned}
\end{equation}
since
\begin{equation*}
	4\theta=\omega_{_M}\leq \left(\frac{\Lambda}{\lambda}\right)^{\frac{1}{\beta}} \qquad \Rightarrow \qquad \theta^{\alpha_0}\leq \frac{1}{4^{\alpha_0}\lambda^{m-1}}\Lambda^{m-1}.
\end{equation*}
To control the diffusion coefficient $U^{m-1}$, we consider the function $u_{\omega}=\max\left\{u,\frac{\omega_{_M}}{4}\right\}$ which is introduced in \cite{HU}. Then, 
\begin{equation}\label{eq-range-of-lower-bvound-of-U-when-grad-u-omega-l-j-eqwual-ero}
\chi_{\left[u\leq l_j\right]}=\chi_{\left[u_{\omega}\leq l_j\right]}\qquad \mbox{and} \qquad 	\theta=\frac{\omega_{_M}}{4}\leq u\leq \left(\frac{U}{\lambda}\right)^{\frac{1}{\beta}} \qquad \mbox{on $\left\{\left|\nabla \left(u_{\omega}-l_j\right)_-\right|=0\right\}$}.
\end{equation}
Thus, by \eqref{eq-start-of-iteration-energy-type-ineqwuaitonly-of-u-l-j-minus-01} and \eqref{eq-range-of-lower-bvound-of-U-when-grad-u-omega-l-j-eqwual-ero} we can have 
\begin{equation}\label{eq-energy-type-indequality-almost-simplifed-03455}
\begin{aligned}
&\sup_{-\theta^{-\alpha_{0}}R_j^2<t<0}\int_{B_{R_j}\times\left\{t\right\}}\left(u_{\omega}-l_j\right)_-^2\eta_j^2\,dx+\theta^{\alpha_{0}}\int_{-\theta^{-\alpha_{0}}R_j^2}^{0}\int_{B_{R_j}}\left|\nabla\left(u_{\omega}-l_j\right)_-\eta_j\right|^2\,dxdt\\
&\qquad \leq C_1\theta^2\left[\frac{2^{2\left(j+1\right)}\Lambda^{m-1}}{R^2}\int_{-\theta^{-\alpha_0}R_j^2}^{0}\int_{B_{R_j}}\chi_{\left[u_{\omega}\leq l_j\right]}\,dxdt+\left(\frac{\Lambda^{2(m-1)}+\Lambda^{\frac{2(q-1)}{\underline{\beta}}}}{\theta^{\alpha_0}}\right)\int_{-\theta^{-\alpha_0}R_j^2}^{0}\int_{B_{R_j}}\chi_{\left[u_{\omega}\leq l_j\right]}\,dxdt\right]
\end{aligned}
\end{equation}
for some constant $C_1>0$. To control the last term in \eqref{eq-energy-type-indequality-almost-simplifed-03455}, we let $q_1$, $q_2\geq 1$ and $0<\kappa_1<1$ be constants satisfying
\begin{equation}\label{relation-between-contants-q-1-and-q-2-anbd-kappa-1}
\frac{n}{2q_1}+\frac{1}{q_2}=1-\kappa_1.
\end{equation}
Let
\begin{equation}\label{relation-between-contants-q-1-and-q-2-anbd-kappa-from-kappa-1}
\widehat{q}=\frac{2q_1\left(1+\kappa\right)}{q_1-1}, \qquad \widehat{r}=\frac{2q_2\left(1+\kappa\right)}{q_2-1} \qquad \mbox{and} \qquad \kappa=\frac{2}{n}\kappa_1.
\end{equation}
Then, by H\"older inequality and conditions on $R$ and $\theta$, there exists a constant $C_2>0$ such that
\begin{equation}\label{eq-controlling-the-third-opart-gradient-derivatives-to-v-0}
\begin{aligned}
\int_{-\theta^{-\alpha_{0}}R_j^2}^{t}\int_{B_{R_j}}\chi_{_{\left[u\leq l_j\right]}}\,dxd\tau&\leq R^{\frac{n}{a_1}}\int_{-\theta^{-\alpha_{0}}R_j^2}^{t}\left|A_{l_j,R_j}^-(\tau)\right|^{1-\frac{1}{a_1}}\,d\tau \\
&\leq R^{\frac{n}{a_1}}\left(\frac{R^2}{\theta^{\alpha_0}}\right)^{\frac{1}{a_2}}\left(\int_{-\theta^{-\alpha_{0}}R_j^2}^{t}\left|A_{l_j,R_j}^-(\tau)\right|^{\left(1-\frac{1}{a_1}\right)\frac{a_2}{a_2-1}}\,d\tau\right)^{1-\frac{1}{a_2}}\\
&\leq C_2\left(\int_{-\theta^{-\alpha_{0}}R_j^2}^{t}\left|A_{l_j,R_j}^-(\tau)\right|^{\frac{\widehat{r}}{\widehat{q}}}\,d\tau\right)^{\frac{2}{\widehat{r}}\left(1+\kappa\right)}
\end{aligned}
\end{equation}
where $A_{l,r}^-(t)=\left\{x\in B_r:\left(u-l\right)_-\geq 0\right\}$. By \eqref{eq-energy-type-indequality-almost-simplifed-03455} and \eqref{eq-controlling-the-third-opart-gradient-derivatives-to-v-0},  we have
\begin{equation}\label{eq-simplifying-of-eq-after-h-to-zero}
\begin{aligned}
&\sup_{-\theta^{-\alpha_{0}}R_j^2<t<0}\int_{B_{R_j}\times\left\{t\right\}}\left(u_{\omega}-l_j\right)_-^2\eta_j^2\,dx+\theta^{\alpha_{0}}\int_{-\theta^{-\alpha_{0}}R_j^2}^{0}\int_{B_{R_j}}\left|\nabla\left(u_{\omega}-l_j\right)_-\eta_j\right|^2\,dxdt\\
&\qquad\leq C_3\theta^2\left[\frac{2^{2\left(j+1\right)}\Lambda^{m-1}}{R^2}\int_{-\theta^{-\alpha_0}R_j^2}^{0}\int_{B_{R_j}}\chi_{\left[u_{\omega}\leq l_j\right]}\,dxdt+\left(\frac{\Lambda^{2(m-1)}+\Lambda^{\frac{2(q-1)}{\underline{\beta}}}}{\theta^{\alpha_0}}\right)\left(\int_{-\theta^{-\alpha_{0}}R_j^2}^{t}\left|A_{l_j,R_j}^-(\tau)\right|^{\frac{\widehat{r}}{\widehat{q}}}\,d\tau\right)^{\frac{2}{\widehat{r}}\left(1+\kappa\right)}\right]
\end{aligned}
\end{equation}
for some constant $C_3$ depending on $C_1$ and $C_2$. We now  take the change of variables
\begin{equation}\label{eq-change-of-variables-from-t-to-z-for-uniform-cylinder}
z=\theta^{\alpha_{0}}\,t
\end{equation}
and set the new functions
\begin{equation*}
\overline{u}_{\omega}\left(\cdot,z\right)=u_{\omega}\left(\cdot,\theta^{-\alpha_{0}}z\right) \qquad \mbox{and} \qquad \overline{\eta}_{i}\left(\cdot,z\right)=\eta_{i}\left(\cdot,\theta^{-\alpha_{0}}z\right).
\end{equation*}
Then, by \eqref{eq-simplifying-of-eq-after-h-to-zero} and \eqref{eq-change-of-variables-from-t-to-z-for-uniform-cylinder} we have
\begin{equation}\label{eq-simplifying-with-change-of-variables}
\begin{aligned}
&\sup_{-R_j^2<z<0}\int_{B_{R_j}\times\left\{z\right\}}\left(\overline{u}_{\omega}-l_j\right)_-^2\overline{\eta}_j^2\,dx+\int_{-R_j^2}^{0}\int_{B_{R_j}}\left|\nabla\left(\overline{u}_{\omega}-l_j\right)_-\overline{\eta}_j\right|^2\,dxdz \\
&\qquad \qquad \qquad \leq  C_3\theta^2\left[\frac{2^{2\left(j+1\right)}}{R^2}\left(\frac{\Lambda}{\theta^{\,\beta}}\right)^{m-1}A_j+\left(\Lambda^{2(m-1)}+\Lambda^{\frac{2(q-1)}{\underline{\beta}}}\right)\theta^{-\alpha_0\left(2-\frac{1}{q_2}\right)}\left(\int_{-R_j^2}^{0}\left|A_{i}(z)\right|^{\frac{\widehat{r}}{\widehat{q}}}\,dz\right)^{\frac{2}{\widehat{r}}\left(1+\kappa\right)}\right]
\end{aligned}
\end{equation}
where 
\begin{equation*}
A_j=\int_{-R_j^2}^{0}\int_{B_{R_j}}\chi_{\left[\overline{u}_{\omega}\leq l_j\right]}\,dxdz \qquad \mbox{and}\qquad A_j(z)=\left\{x\in B_{R_j}:\overline{u}_{_{\omega}}(x,z)<l_j\right\}.
\end{equation*}
By Lemma \ref{eq-Sobolev-Type-Inequality} and \eqref{eq-simplifying-with-change-of-variables},
\begin{equation}\label{eq-simplifying-after-change-of-variables-with-A-i}
\begin{aligned}
&\left\|\left(\overline{u}_{\omega}-l_j\right)_-^2\overline{\eta}_j^2\right\|_{L^2\left(Q\left(R_j,R_j^2\right)\right)} \\
&\qquad \leq C_4\theta^2\,\left[\frac{2^{2\left(j+1\right)}}{R_j^2}\left(\frac{\Lambda}{\theta^{\,\beta}}\right)^{m-1}A_j+\left(\Lambda^{2(m-1)}+\Lambda^{\frac{2(q-1)}{\underline{\beta}}}\right)\theta^{-\alpha_0\left(2-\frac{1}{q_2}\right)}\left(\int_{-R_j^2}^{0}\left|A_{j}(z)\right|^{\frac{\widehat{r}}{\widehat{q}}}\,dz\right)^{\frac{2}{\widehat{r}}\left(1+\kappa\right)}\right]\,A_j^{\frac{2}{n+2}}
\end{aligned}
\end{equation}
for some constant $C_4>0$. This immediately implies
\begin{equation}\label{eq-iteration-form-of-A-j-s-just-before-inverting-to-X-j-and-Y-J-01}
A_{j+1}\leq C_42^{4j+1}\,\left[\left(\frac{\Lambda}{\theta^{\,\beta}}\right)^{m-1}\frac{A_j}{R_j^2}+\left(\Lambda^{2(m-1)}+\Lambda^{\frac{2(q-1)}{\underline{\beta}}}\right)\theta^{-\alpha_0\left(2-\frac{1}{q_2}\right)}\left(\int_{-R_j^2}^{0}\left|A_{j}(z)\right|^{\frac{\widehat{r}}{\widehat{q}}}\,dz\right)^{\frac{2}{\widehat{r}}\left(1+\kappa\right)}\right]\,A_j^{\frac{2}{n+2}}
\end{equation}
since
\begin{equation*}
\left\|\left(\overline{u}_{\omega}-l_j\right)_-^2\overline{\eta}_j^2\right\|_{L^2\left(Q\left(R_j,R_j^2\right)\right)}\geq \left(l_{j+1}-l_j\right)^2\int_{-R_j^2}^{0}\left|\left\{(x,t)\in B_{R_{j+1}}:\overline{u}_{\omega}\leq l_{j+1}\right\}\right|\,dt=\left(\frac{\theta}{2^{j}}\right)^2A_{j+1}.
\end{equation*}
Choose the number $\epsilon>0$ sufficiently small that
\begin{equation}\label{eq-condition-of-epsilon-for-holder-continuity-at-start}
\epsilon<\frac{nq_2\kappa}{2q_2-1}. 
\end{equation}
Then, by \eqref{eq-iteration-form-of-A-j-s-just-before-inverting-to-X-j-and-Y-J-01} there exists a constant $C_5>0$ depending on $\frac{\Lambda}{\theta^{\,\beta}}$ such that
\begin{equation}\label{eq-iteration-for-X-jk-and-Y-j-first}
X_{j+1}\leq C_516^j\left(X_j^{1+\frac{2}{n+2}}+\theta^{-\alpha_0\left(2-\frac{1}{q_2}\right)}R_j^{n\kappa}X_j^{\frac{2}{n+2}}Y_j\right)\leq C_516^j\left(X_j^{1+\frac{2}{n+2}}+X_j^{\frac{2}{n+2}}Y_j\right) \qquad \forall j\in\N.
\end{equation}
where
\begin{equation*}
X_j=\frac{A_j}{\left|Q\left(R_j,R_j^2\right)\right|} \qquad \mbox{and}\qquad Y_j=\frac{1}{\left|B_{R_j}\right|}\left(\int_{-R_j^2}^{0}\left|A_{j}(z)\right|^{\frac{\widehat{r}}{\widehat{q}}}\,dz\right)^{\frac{2}{\widehat{r}}}.
\end{equation*}
By an argument similar to the one presented in the proof of Lemma 3.5 of \cite{KL1}, we also have
\begin{equation}\label{eq-iteration-for-X-jk-and-Y-j-second}
Y_{j+1}\leq C_616^j\left(X_j+Y^{1+\kappa}_j\right)\qquad \forall j\in\N
\end{equation}
for some constant $C_6>0$. By \eqref{eq-iteration-for-X-jk-and-Y-j-first} and \eqref{eq-iteration-for-X-jk-and-Y-j-second}, there exists a constant $C_7>0$ such that
\begin{equation*}
L_{j+1}\leq C_716^{i(1+\kappa)}L_j^{1+\widehat{\kappa}}\qquad \forall j\in\N
\end{equation*}
where $L_j=X_j+Y_j^{1+\kappa}$ and $\widehat{\kappa}=\min\left\{\kappa,\frac{2}{n+2}\right\}$. If we take the constant $\rho_0>0$ in \eqref{eq-first-condition-of-small-region-of-lower-for-holder-estimates} sufficiently small that
\begin{equation*}
L_0\leq C_7^{-\frac{1+\kappa}{\widehat{\kappa}}}16^{-\frac{1+\kappa}{\widehat{\kappa}^2}}
\end{equation*}
holds, then
\begin{equation*}
L_j\leq C_7^{-\frac{(1+\kappa)(1+\widehat{\kappa})}{\widehat{\kappa}}}16^{-\frac{(1+\kappa)(1+j\,\widehat{\kappa})}{\widehat{\kappa}^2}} \to 0 \qquad \mbox{as $i\to\infty$}
\end{equation*}
and the lemma follows. 
\end{proof}

\begin{remark}
Let the constants $\lambda_i$, $\beta_i$ be given by \eqref{eq-standard-form-for-local-continuity-estimates-intro-2}. If $U$ is equivalent to $\lambda_i\left(u^i\right)^{\beta_i}$, i.e., there exists some constants $0<c\leq C<\infty$ such that
\begin{equation*}
c\lambda_i\left(u^i\right)^{\beta_i}\leq U\leq C\lambda_i\left(u^i\right)^{\beta_i} \qquad \mbox{in $Q\left(R,\theta^{-\alpha_0}R^{2}\right)$},
\end{equation*}
then the constant $\rho_0$ in \eqref{eq-first-condition-of-small-region-of-lower-for-holder-estimates} is independent of $U$ and $\omega^1$, $\cdots$, $\omega^k$.
\end{remark}
\begin{remark}\label{remark-first-alternative-for-fast-diffusion-type-system}
The first alternative can be extended to the fast diffusion type system, i.e., the lemma is still true for $\max\left(0,1-\frac{1}{\beta_i}\right)<m<1$. To explain it, we let $\epsilon>0$ and $0<a<1$ be constants such that
\begin{equation*}
m-1+\frac{a}{\beta_i}>0.
\end{equation*}
Consider the quantity $a\left(u^i+\epsilon\right)^{a-1}\left(\left(u^i\right)^a-l_i^a\right)_-\eta_j^2$ as a test function in the proof of Lemma  \ref{lem-the-first-alternative-for-holder-estimates}. Then, letting $\epsilon\to 0$ we can get the following energy type inequality
\begin{equation}\label{eq-energy-type-inequality-for-fast-diffysion-type-system}
\begin{aligned}
&\sup_{-\theta^{\,-\alpha_{0}}R_i^2<t<0}\int_{B_{R_j}}\left(u^{\,a}-l^{\,a}_j\right)_-^2\eta_j^2\,dx+\Lambda^{m-1}\int_{-\theta^{-\alpha_{0}}R_j^2}^{0}\int_{B_{R_j}}\left|\nabla\left(u^{\,a}-l_j\right)_-\eta_j\right|^2\,dxdt\\
&\qquad \qquad \leq C\Lambda^{m-1}\theta^{2a}\Bigg[\frac{2^{2\left(i+1\right)}}{R^2}\left(1+\left(\frac{\Lambda}{\theta^{\beta_i}}\right)^{a}+\left(\frac{\theta^{\beta_i}}{\Lambda}\right)^{a}\right)\int_{-\theta^{-\alpha_0}R_i^2}^{0}\int_{B_{R_i}}\chi_{\left[u\leq l_i\right]}\,dxdt\\
&\qquad \qquad \qquad \qquad\qquad \qquad  +\left(\left(\frac{\Lambda}{\theta^{\beta_i}}\right)^{a}+\frac{\Lambda^{2(q-1)}}{\Lambda^{m-1}}\right)\left(\int_{-\theta^{-\alpha_{0}}R_j^2}^{t}\left|A_{l_j,R_j}^-(\tau)\right|^{\frac{\widehat{r}}{\widehat{q}}}\,d\tau\right)^{\frac{2}{\widehat{r}}\left(1+\kappa\right)}\Bigg]
\end{aligned}
\end{equation}
for some constant $C>0$. Therefore, by an argument similar to the one presented in the proof of Lemma \ref{lem-the-first-alternative-for-holder-estimates} we can get a desired conclusion.
\end{remark}

\subsection{The Second Alternative}

	At this moment, we have shown that if the measure of the set
	\begin{equation*}
	\left\{\left(x,t\right)\in Q\left(R,\theta^{-\alpha_{0}}R^2\right):u^i(x,t)<\frac{\omega_{_M}}{2}\right\}
	\end{equation*}
	is very small then the component $u^i$ is strictly bounded below away from zero (essential infimum) in a smaller cylinder of $Q\left(R,\theta^{-\alpha_{0}}R^2\right)$. This is, their essential oscillation will decrease. We now need to get rid of assumption "very small".\\
\indent Suppose that the assumption of Lemma \ref{lem-the-first-alternative-for-holder-estimates} does not hold, i.e., for every sub-cylinder $Q\left(R,\theta^{-\alpha_{0}}R^2\right)$, $\left(0<R_0<\theta^{\frac{\alpha_0}{\epsilon}}\right)$,
\begin{equation*}\label{eq-start-point-of-the-second-alternative-assumption-violated}
\left|\left\{\left(x,t\right)\in Q\left(R,\theta^{-\alpha_{0}}R^2\right):u^i(x,t)<\frac{\omega_{_M}}{2}\right\}\right|>\rho_0\left|Q\left(R,\theta^{-\alpha_{0}}R^2\right)\right|.
\end{equation*}
Then
\begin{equation}\label{basic-assumption-for-second-alternative-01}
\left|\left\{\left(x,t\right)\in Q\left(R,\theta^{-\alpha_{0}}R^2\right):u^i(x,t)>\frac{\omega_{_M}}{2}\right\}\right|\leq\left(1-\rho_0\right)\left|Q\left(R,\theta^{-\alpha_{0}}R^2\right)\right|
\end{equation}
is valid for all cylinders 
\begin{equation*}
Q\left(R,\theta^{-\alpha_{0}}R^2\right)\subset Q\left(R, R^{2-\epsilon}\right).
\end{equation*}
In the second alternative, we are going to show that the essential oscillation of $u^i$ decreases in a smaller cylinder by showing that the essential supremum of $u^i$ decreases. We start this alternative by stating the following lemma.
\begin{lemma}\label{lem-propostion-of-area-near-supremum-at-some-time-t-ast}
If \eqref{eq-first-condition-of-small-region-of-lower-for-holder-estimates} is violated, then there exists a time level
\begin{equation*}
t^{\ast}\in\left[-\theta^{-\alpha_{0}}R^2,-\frac{\rho_0}{2}\theta^{-\alpha_{0}}R^2\right]
\end{equation*}
such that
\begin{equation*}
\left|\mathcal{A}_0\right|=\left|\left\{x\in B_{R}:u^i\left(x,t^{\ast}\right)>\frac{\omega_{_M}}{2}\right\}\right|\leq \left(\frac{1-\rho_0}{1-\frac{\rho_0}{2}}\right)\left|B_R\right|.
\end{equation*}
\end{lemma}
\begin{proof}
	Suppose not. Then
	\begin{equation*}
	\begin{aligned}
	\left|\left\{\left(x,t\right)\in Q\left(R,\theta^{-\alpha_{0}}R^2\right):u^i(x,t)>\frac{\omega_{_M}}{2}\right\}\right|&\geq  \int_{-\theta^{-\alpha_{0}}R^2}^{-\frac{\rho_0}{2}\theta^{-\alpha_{0}}R^2}\left|\left\{x\in B_{R}:u^i\left(x,t\right)>\frac{\omega_{_M}}{2}\right\}\right|\,dt\\
	&> \left(\frac{1-\rho_0}{1-\frac{\rho_0}{2}}\right)\left|B_R\right|\left(1-\frac{\rho_0}{2}\right)\theta^{-\alpha_{0}}R^2\\
	&=\left(1-\rho_0\right)\left|Q\left(R,\theta^{-\alpha_{0}}R^2\right)\right|
	\end{aligned}
	\end{equation*}
	which contradicts \eqref{basic-assumption-for-second-alternative-01}. 
\end{proof}
By Lemma \ref{lem-propostion-of-area-near-supremum-at-some-time-t-ast}, there exists a time $t^{\ast}<0$ such that the region $\mathcal{A}_0$ takes a portion of the ball $B_R$. The next lemma shows that this occurs for all $t\geq t^{\ast}$. 
\begin{lemma}
There exists a positive integer $s_1>1$ depending on $\frac{\Lambda}{\theta^{\,\beta_i}}$ such that
\begin{equation}\label{eq-supremum-small-occurs-for-all-time-near-top}
\left|\left\{x\in B_{R}:u^i(x,t)>\left(1-\frac{1}{2^{s_1}}\right)\omega_{_M}\right\}\right|<\left(1-\left(\frac{\rho_0}{2}\right)^2\right)\left|B_R\right|,\qquad \forall t\in\left[t^{\ast},0\right].
\end{equation}
\end{lemma}

\begin{proof}
We will use a modification of the proof of Lemma 3.7 of \cite{KL1} to prove the lemma. Let $u=u^i$, $\beta_i=\beta$, $\gamma_i=\gamma$ for the convenience and let 
\begin{equation*}
H=\sup_{B_{R}\times\left[t^{\ast},0\right]}\left(u-\frac{\omega_{_M}}{2}\right)_+\leq \frac{\omega_{_M}}{2}
\end{equation*}
and assume that there exists a constant $1<s_2\in\N$ such that
\begin{equation*}
0<\frac{\omega_{_M}}{2^{s_2+1}}<H 
\end{equation*}
If there's no such integer $s_2$, \eqref{eq-supremum-small-occurs-for-all-time-near-top} holds for any $s_1>1$ and the lemma follows.\\
\indent We now introduce the logarithmic function which appears in Section 2 of \cite{Di} by
\begin{equation*}
\Psi\left(H,\left(u-k\right)_+,c\right)=\max\left\{0,\,\, \log\left(\frac{H}{H-\left(u-k\right)_++c}\right)\right\}
\end{equation*}
for $k=\frac{\omega_{_M}}{2}$ and $c=\frac{\omega_{_M}}{2^{s_2+1}}$.  Note that
\begin{equation}\label{eq-region-where-Psi-zero-on-u-leq-k}
\Psi\left(H,\left(u-k\right)_+,c\right)=0 \qquad \mbox{if $u\leq k=\frac{\omega_{_M}}{2}$}.
\end{equation}
For simplicity, we let $\psi\left(u\right)=\Psi\left(H,\left(u-k\right)_+,c\right)$. Then $\psi$ satisfies
\begin{equation}\label{eq-condition-of-psi-from-Psi-H-u-k-c}
\psi\leq s_2\log 2, \qquad 0\leq \left(\psi\right)'\leq \frac{2^{s_2+1}}{\omega_{_M}} \qquad \mbox{and} \qquad \psi''=\left(\psi'\right)^2\geq 0.
\end{equation}
Set 
\begin{equation*}
\vp=\left(\psi^2\left(u\right)\right)'\xi^2
\end{equation*} 
and take it as a test function in \eqref{eq-identity--of-formula-for-weak-solution} where  $\xi(x)\geq 0$ is a smooth cut-off function such that
\begin{equation}\label{eq-condition-of-cut-off-function-independent-of-t}
\xi=1 \quad \mbox{in $B_{\left(1-\nu\right)R}$}, \qquad \xi=0 \quad \mbox{on $\partial B_{R}$} \qquad \mbox{and} \qquad \left|\nabla\xi\right|\leq \frac{C}{\nu R} 
\end{equation}
for some constants $0<\nu<1$ and $C>0$. Then integrating \eqref{eq-identity--of-formula-for-weak-solution} over $\left(t^{\ast},t\right)$ for all $t\in\left(t^{\ast},0\right)$, we have
\begin{equation}\label{eq-first-step-of-multiplying-test-=function-for-second-alternativwes}
\begin{aligned}
0&=\int_{t^{\ast}}^t\int_{B_R}\left(\psi^2\left(u\right)\xi^2\right)_{\tau}\,dxd\tau+m\int_{t^{\ast}}^t\int_{B_R}U^{m-1}\mathcal{A}\left(\nabla u,u,x,t\right)\cdot\nabla\left(\left(\psi^2\left(u\right)\right)'\xi^2\right)\,dxd\tau\\
&\qquad \qquad +\int_{t^{\ast}}^t\int_{B_R}\mathcal{B}\left(u,x,t\right)\cdot\nabla\left(\left(\psi^2\left(u\right)\right)'\xi^2\right)\,dxd\tau\\ 
&=I+II+III.
\end{aligned}
\end{equation}
Then we have
\begin{align}
I=\int_{B_R\times\left\{t\right\}}\psi^2\left(u\right)\xi^2\,dx-\int_{B_R\times\left\{t^{\ast}\right\}}\psi^2\left(u\right)\xi^2\,dx \label{eq-first-integral-removing-time-derivatives-1}
\end{align}
and
\begin{equation}\label{eq-second-integral-using-youngs-indequality-after-h-to-010}
\begin{aligned}
II&\geq 2\textbf{c}m\int_{t^{\ast}}^t\int_{B_R}U^{m-1}\left(1+\psi\right)\left(\psi'\right)^2\xi^2\left|\nabla u\right|^2\,dxd\tau\\
&\qquad \qquad -2\textbf{C}_2m\int_{t^{\ast}}^t\int_{B_R}U^{m-1}\left(1+\psi\right)\left(\psi'\right)^2\xi^2u^2\,dxd\tau\\
&\qquad \qquad  \qquad \qquad -\left(\frac{\textbf{C}_3^2}{\textbf{c}}+\textbf{C}_4m\right)\int_{t^{\ast}}^t\int_{B_R}U^{m-1}\psi\left|\nabla \xi\right|^2\,dxd\tau \\
&\qquad \qquad  \qquad \qquad \qquad \qquad -\textbf{C}_4m\int_{t^{\ast}}^t\int_{B_R}U^{m-1}\psi\left(\psi'\right)^2\xi^2u^2\,dxd\tau
\end{aligned}
\end{equation}
and
\begin{equation}\label{eq-third-integral-using-youngs-indequality-after-h-to-010}
\begin{aligned}
-III&=2\textbf{C}_5\int_{t^{\ast}}^t\int_{B_R}u^q\left(1+\psi\right)\left(\psi'\right)^2\xi^2\left|\nabla u\right|\,dxd\tau+4\textbf{C}_5\int_{t^{\ast}}^t\int_{B_R}u^q\psi\psi'\xi \left|\nabla \xi\right|\,dxd\tau\\
& \leq \textbf{c}m\Lambda^{m-1}\int_{t^{\ast}}^t\int_{B_R}\left(1+\psi\right)\left(\psi'\right)^2\xi^2\left|\nabla u\right|^2\,dxd\tau\\
& \qquad+\frac{2\textbf{C}_5^2}{\textbf{c}m\Lambda^{m-1}}\int_{t^{\ast}}^t\int_{B_R}u^{2q}\left(1+\psi\right)\left(\psi'\right)^2\xi^2\,dxd\tau\\
& \qquad \qquad +4\textbf{c}m\Lambda^{m-1}\int_{t^{\ast}}^t\int_{B_R}\psi\left|\nabla \xi\right|^2\,dxd\tau.
\end{aligned}
\end{equation}
By \eqref{eq-condition-of-psi-from-Psi-H-u-k-c}, \eqref{eq-condition-of-cut-off-function-independent-of-t}, \eqref{eq-first-step-of-multiplying-test-=function-for-second-alternativwes}, \eqref{eq-first-integral-removing-time-derivatives-1}, \eqref{eq-second-integral-using-youngs-indequality-after-h-to-010}, \eqref{eq-third-integral-using-youngs-indequality-after-h-to-010} and Lemma \ref{lem-propostion-of-area-near-supremum-at-some-time-t-ast},
\begin{equation}\label{eq-after-togethering-the-frist-and-second-integral0-and-simplify}
\begin{aligned}
&\int_{B_R\times\left\{t\right\}}\psi^2\left(u\right)\xi^2\,dx\\
&\qquad \leq \left[s_2^2\left(\log2\right)^2\left(\frac{1-\rho_0}{1-\frac{\rho_0}{2}}\right)+C\left(\frac{s_2\log 2}{\nu^2}\left(\frac{\Lambda}{\theta^{\,\beta}}\right)^{m-1}+\left(\Lambda^{m-1}+\Lambda^{\frac{2(q-1)}{\beta}}\right)4^{s_2+1}R^{2-\epsilon}s_2\log 2\right)\right]\left|B_R\right|
\end{aligned}
\end{equation}
holds for all $t\in\left(t^{\ast},0\right)$ with some constant $C>0$ depending $m$, $q$, $\lambda$, $\beta$, $\textbf{c}$, $\textbf{C}_3$, $\textbf{C}_4$ and $\textbf{C}_5$. Let 
\begin{equation*}
\mathcal{S}=\left\{x\in B_{(1-\nu)R}:u(x,t)>\left(1-\frac{1}{2^{s_2+1}}\right)\omega_{_M}\right\}.
\end{equation*}
Then, the left hand side of \eqref{eq-after-togethering-the-frist-and-second-integral0-and-simplify} is bounded below by
\begin{equation}\label{lower-bound-of-the-left-hand0sidelintegarl-of-simplify}
\int_{B_R\times\left\{t\right\}}\psi^2\left(u\right)\xi^2\,dx\geq \int_{\mathcal{S}}\psi^2\left(u\right)\xi^2\,dx\geq\left(s_2-1\right)^2\left(\log 2\right)^2\left|\mathcal{S}\right| \qquad \forall t\in\left(t^{\ast},0\right). 
\end{equation}
On the other hand,
\begin{equation}\label{eq-volume-of-set-greater-than-mu--2-to-=s-0-and-s-2-with-mathcal-S}
\left|\left\{x\in B_{R}:u(x,t)>\left(1-\frac{1}{2^{s_2+1}}\right)\omega_{_M}\right\}\right|\leq\left|\mathcal{S}\right|+N\nu\left|B_R\right|.
\end{equation}
By \eqref{eq-after-togethering-the-frist-and-second-integral0-and-simplify}, \eqref{lower-bound-of-the-left-hand0sidelintegarl-of-simplify} and \eqref{eq-volume-of-set-greater-than-mu--2-to-=s-0-and-s-2-with-mathcal-S},
\begin{equation*}
\begin{aligned}
&\left|\left\{x\in B_{R}:u(x,t)>\left(1-\frac{1}{2^{s_2+1}}\right)\omega_{_M}\right\}\right|\\
&\qquad \qquad \leq \left[\left(\frac{s_2}{s_2-1}\right)^2\left(\frac{1-\rho_0}{1-\frac{\rho_0}{2}}\right)+n\nu+C\left(\frac{s_2}{\nu^2(s_2-1)^2\log 2}\left(\frac{\Lambda}{\theta^{\,\beta}}\right)^{m-1}+\frac{\left(\Lambda^{m-1}+\Lambda^{\frac{2(q-1)}{\beta}}\right)4^{s_2+1}R^{2-\epsilon}s_2}{\left(s_2-1\right)^2\log 2}\right)\right]\left|B_R\right|.
\end{aligned}
\end{equation*}
To complete the proof, we choose $\nu$ so small that $n\nu\leq \frac{3}{8}\rho_0^2$ and then $s_2$ so large that
\begin{equation*}
\left(\frac{s_2}{s_2-1}\right)^2\leq \left(1-\frac{1}{2}\rho_0\right)\left(1+\rho_0\right) \qquad \mbox{and} \qquad C\frac{s_2}{\nu^2(s_2-1)^2\log 2}\left(\frac{\Lambda}{\theta^{\,\beta}}\right)^{m-1}\leq \frac{1}{4}\rho_0^2.
\end{equation*}
With such $\nu$ and $s_2$, we choose the radius $R$ sufficiently small that
\begin{equation}\label{eq-range-of-R-sufficiently-small-second-2}
\frac{C\left(\Lambda^{m-1}+\Lambda^{\frac{2(q-1)}{\beta}}\right)4^{s_2+1}R^{2-\epsilon}s_2}{\left(s_2-1\right)^2\log 2}\leq \frac{3}{8}\rho_0^2.
\end{equation}
Then \eqref{eq-supremum-small-occurs-for-all-time-near-top} holds for $s_1=s_2+1$ and the lemma follows. 
\end{proof}

Since $t^{\ast}\in\left[-\theta^{-\alpha_{0}}R^2,-\frac{\rho_0}{2}\theta^{-\alpha_{0}}R^2\right]$, the previous lemma implies the following result.
\begin{cor}\label{cor--propostion-of-area-near-supremum-at-some-time-t-ast}
There exists a positive integer $s_1>s_0$ such that for all $t\in\left(-\frac{\rho_0}{2}\theta^{-\alpha_{0}}R^2,0\right)$
\begin{equation}\label{eq-supremum-small-occurs-for-all-time-near-top-2}
\left|\left\{x\in B_{R}:u^i(x,t)>\left(1-\frac{1}{2^{s_1}}\right)\omega_{_M}\right\}\right|<\left(1-\left(\frac{\rho_0}{2}\right)^2\right)\left|B_R\right|.
\end{equation}
\end{cor}

To control the measure of the region where $u^i$ is close to the value $\omega_M$, we are going to use the following De Giorgi's isoperimetric inequality.
\begin{lemma}[De Giorgi\cite{De}]\label{De-giorgi}
If $f\in W^{1,1}(B_r)$ $(B_r\subset\R^n)$ and $l_1,\,l_2\in \R$, $l_1<l_2$,
then
\begin{equation*}
(l_2-l_1)\left|\left\{x\in B_r: f(x)>l_2\right\}\right|\leq
\frac{Cr^{n+1}}{\left|\left\{x\in B_r:
f(x)<l_1\right\}\right|}\int_{l_1<f<l_2}|\nabla f|\,\,dx,
\end{equation*}
where $C$ depends only on $n$.
\end{lemma}
By Corollary \ref{cor--propostion-of-area-near-supremum-at-some-time-t-ast} and Lemma \ref{De-giorgi}, we have the following lemma which control the measure of upper level sets.
\begin{lemma}\label{lem-last-condition-for-second-alternativ-from-violate-of-first-alternative}
If \eqref{eq-first-condition-of-small-region-of-lower-for-holder-estimates} is violated, for every $\nu_{\ast}\in\left(0,1\right)$ there exists a natural number $s^{\ast}>s_1>1$ depending on $\frac{\Lambda}{\theta^{\,\beta}}$ such that
\begin{equation}\label{eq-condition-with-s-upper-ast-and-nu--sub-ast-for-second-alternative}
\left|\left\{(x,t)\in Q\left(R,\frac{\rho_0}{2}\theta^{-\alpha_{0}}R^2\right): u^i(x,t)>\left(1-\frac{1}{2^{s^{\ast}}}\right)\omega_{_M}\right\}\right|\leq \nu_{\ast}\left|Q\left(R,\frac{\rho_0}{2}\theta^{-\alpha_{0}}R^2\right)\right|.
\end{equation}
\end{lemma}

\begin{proof}
We will use a  modification of the proof of Lemma 8.1 of Section III of \cite{Di} to prove the lemma. Let $l_1=\left(1-\frac{1}{2^s}\right)\omega_{_M}$ and $l_2=\left(1-\frac{1}{2^{s+1}}\right)\omega_{_M}$ for $s\geq s_1$ and let $\eta(x,t)\in C^{\infty}\left(Q\left(2R,\rho_0\theta^{-\alpha_{0}}R^2\right)\right)$ be a cut-off function such that
\begin{equation*}
\begin{cases}
\begin{array}{cccl}
0\leq \eta\leq 1 &&& \mbox{in $Q\left(2R,\rho_0\theta^{-\alpha_{0}}R^2\right)$}\\
\eta=1 &&& \mbox{in $Q\left(R,\frac{\rho_0}{2}\theta^{-\alpha_{0}}R^2\right)$ }\\
\eta=0 &&& \mbox{on the parabolic boundary of $Q\left(2R,\rho_0\theta^{-\alpha_{0}}R^2\right)$}\\
\left|\nabla\eta\right|\leq \frac{1}{R},\qquad \left|\eta_t\right|\leq \frac{2\theta^{\alpha_{0}}}{\rho_0R^2}.&&&
\end{array}
\end{cases}
\end{equation*} 
Let $u=u^i$, $\lambda=\lambda_i$, $\beta=\beta_i$ for the convenience and put $\vp=\left(u_h-k\right)_+\xi^2$ in the weak formulation \eqref{eq-formulation-for-weak-solution-of-u-h}, Integrate it over $\left(-\rho_0\theta^{-\alpha_{0}}R^2,t\right)$ for $t\in\left(-\rho_0\theta^{-\alpha_{0}}R^2,0\right)$ and take the limit as $h\to 0$. Then, by an argument simlar to the proof of Energy type inequality \eqref{eq-simplifying-of-eq-after-h-to-zero} there exists a constant $C>0$ depending on $m$, $\lambda$ and $\Lambda$ such that
\begin{equation}\label{eq-upper-bound-of-gradient-of-u-i-minus-k-in-second-alternative}
\begin{aligned}
&\int_{-\frac{\rho_0}{2}\theta^{-\alpha_{0}}R^2}^t\int_{B_{R}}\left|\nabla \left(u^{\,i}-l_1\right)_+\right|^2\,dx\,d\tau\\
&\qquad \leq C\left(\frac{\omega}{2^s}\right)^2\frac{1}{\rho_0R^2}\left(1+\left(\frac{\Lambda}{\theta^{\,\beta}}\right)^{m-1}+2^sR^{n\kappa-\epsilon\left(2-\frac{1}{q_2}\right)}\right)\left|Q\left(R,\frac{\rho_0}{2}\theta^{-\alpha_{0}}R^2\right)\right|
\end{aligned}
\end{equation}
where constants $\kappa$, $q_2$ are given by \eqref{relation-between-contants-q-1-and-q-2-anbd-kappa-1} and \eqref{relation-between-contants-q-1-and-q-2-anbd-kappa-from-kappa-1}. Let
\begin{equation*}
A_s\left(t\right)=\left\{x\in B_{R}:u(x,t)>\left(1-\frac{1}{2^s}\right)\omega\right\}, \qquad \forall t\in\left(-\frac{\rho_0}{2}\theta^{-\alpha_{0}}R^2,0\right)
\end{equation*} 
and 
\begin{equation*}
A_s=\int_{-\frac{\rho_0}{2}\theta^{-\alpha_{0}}R^2}^0\left|A_s(t)\right|\,dt.
\end{equation*}
Then, by Corollary \ref{cor--propostion-of-area-near-supremum-at-some-time-t-ast}, Lemma \ref{De-giorgi} and \eqref{eq-upper-bound-of-gradient-of-u-i-minus-k-in-second-alternative} we have
\begin{equation*}
\begin{aligned}
&\left(\frac{\omega_{_M}}{2^{s+1}}\right)\left|A_{s+1}(t)\right|\leq \frac{CR}{\rho_0^2}\int_{\left\{\left(1-\frac{1}{2^s}\right)\omega_{_M}<u<\left(1-\frac{1}{2^{s+1}}\right)\omega_{_M}\right\}}\left|\nabla u\right|\,dx\qquad \qquad\forall s=s_1,\cdots,s^{\ast}-1\\
&\Rightarrow \qquad \left(\frac{\omega_{_M}}{2^{s+1}}\right)A_{s+1}\leq \frac{CR}{\rho_0^2}\left(\int_{-\frac{\rho_0}{2}\theta^{-\alpha_{0}}R^2}^0\int_{B_R}\left|\nabla(u-l_1)_+\right|^2\,dx\,dt\right)^{\frac{1}{2}}\left|A_s\bs A_{s+1}\right|^{\frac{1}{2}}\\
&\Rightarrow \qquad A_{s+1}^2\leq \frac{C}{\rho_0^5}\left(1+\left(\frac{\Lambda}{\theta^{\,\beta}}\right)^{m-1}+2^{s^{\ast}}R^{n\kappa-\epsilon\left(2-\frac{1}{q_2}\right)}\right)\left|Q\left(R,\frac{\rho_0}{2}\theta^{-\alpha_{0}}R^2\right)\right|\left|A_s\bs A_{s+1}\right|  \\
&\Rightarrow \qquad \left(s^{\ast}-s_1\right)A_{s^{\ast}}^2\leq \sum_{s=s_1}^{s^{\ast}-1}A_{s+1}^2\leq \frac{C}{\rho_0^5}\left(1+\left(\frac{\Lambda}{\theta^{\,\beta}}\right)^{m-1}+2^{s^{\ast}}R^{n\kappa-\epsilon\left(2-\frac{1}{q_2}\right)}\right)\left|Q\left(R,\frac{\rho_0}{2}\theta^{-\alpha_{0}}R^2\right)\right|\sum_{s=s_1}^{s^{\ast}-1}\left|A_{s_1}\bs A_{s^{\ast}}\right|\\
&\Rightarrow \qquad A_{s^{\ast}}^2\leq \frac{C}{\rho_0^5\left(s^{\ast}-s_1\right)}\left(1+\left(\frac{\Lambda}{\theta^{\,\beta}}\right)^{m-1}+2^{s^{\ast}}R^{n\kappa-\epsilon\left(2-\frac{1}{q_2}\right)}\right)\left|Q\left(R,\frac{\rho_0}{2}\theta^{-\alpha_{0}}R^2\right)\right|^2.
\end{aligned}
\end{equation*}
Thus if we choose $s^{\ast}\in\N$ sufficiently large that
\begin{equation*}
\frac{C}{\rho_0^5\left(s^{\ast}-s_1\right)}\left(2+\left(\frac{\Lambda}{\theta^{\,\beta}}\right)^{m-1}\right)\leq \nu_{\ast}^2
\end{equation*}
and then $R$ sufficiently small that
\begin{equation}\label{eq-second-condition-for-range-of-R-by-s-ast}
2^{2s^{\ast}}R^{n\kappa-\epsilon\left(2-\frac{1}{q_2}\right)}\leq 1,
\end{equation}
then \eqref{eq-condition-with-s-upper-ast-and-nu--sub-ast-for-second-alternative} holds and the lemma follows. 
\end{proof}

\begin{remark}
Let the constants $\lambda_i$, $\beta_i$ be given by \eqref{eq-standard-form-for-local-continuity-estimates-intro-2}. If $U$ is equivalent to $\lambda_i\left(u^i\right)^{\beta_i}$, i.e., there exists some constants $0<c\leq C<\infty$ such that
\begin{equation*}
c\lambda_i\left(u^i\right)^{\beta_i}\leq U\leq C\lambda_i\left(u^i\right)^{\beta_i} \qquad \mbox{in $Q\left(R,\theta^{-\alpha_0}R^{2}\right)$},
\end{equation*}
then the constant $s^{\ast}$ is independent of $U$ and $\omega^1$, $\cdots$, $\omega^k$.
\end{remark}
By Lemma \ref{lem-last-condition-for-second-alternativ-from-violate-of-first-alternative}, we have a similar assumption to the one in Lemma \ref{lem-the-first-alternative-for-holder-estimates} for sufficiently small number $\nu_{\ast}>0$. Therefore, by an argument similar to the proof of Lemma \ref{lem-the-first-alternative-for-holder-estimates}, we can have the following result.

\begin{lemma}\label{lem-the-second-alternative-for-holder-estimates}
The number $\nu_{\ast}\in\left(0,1\right)$ can be chosen (and hence $s^{\ast}$) such that
\begin{equation*}
u^i(x,t)\leq \left(1-\frac{1}{2^{s^{\ast}+1}}\right)\omega_{_M} \qquad \mbox{a.e. on $Q\left(\frac{R}{2},\frac{\rho_0}{2}\theta_{0}^{-\alpha_{0}}\left(\frac{R}{2}\right)^2\right)$}. 
\end{equation*}
\end{lemma}
\begin{remark}\label{remark-explain-extension-to-fde-system-on-second-alternative}
Throughout the second alternative, the diffusion coefficient $U^{m-1}$ is still nondegenerate for $0<m<1$ since 
\begin{equation*}
0<\frac{\omega_{_M}}{2}\leq U\leq \Lambda<\infty.
\end{equation*} 
Therefore the second alternative can be extended to the fast diffusion type system, i.e., the Lemma \ref{lem-the-second-alternative-for-holder-estimates} holds for $1-\frac{1}{\beta_i}<m<1$.
\end{remark}

\subsection{Local Continuity}

By Lemma \ref{lem-the-first-alternative-for-holder-estimates} and Lemma \ref{lem-the-second-alternative-for-holder-estimates}, we have the following Oscillation Lemma.
\begin{lemma}[Oscillation Lemma]\label{lem-Oscillation-Lemma} 
Let $1\leq i\leq k$. There exist numbers $\rho_0$, $\sigma_0\in\left(0,1\right)$ depending on the $\frac{\Lambda}{\theta^{\,\beta}}$ such that if
\begin{equation*}
\osc_{Q\left(R,\theta^{-\alpha_{0}}R^{2}\right)}u^i\leq \omega_{_M}
\end{equation*}
then we have
\begin{equation}\label{eq-inequality-for-Oscillation-Lemma}
\osc_{Q\left(\frac{R}{2},\frac{\rho_0}{2}\theta^{-\alpha_{0}}\left(\frac{R}{2}\right)^{2}\right)}u^i\leq \sigma_0\omega_{_M}.
\end{equation}
\end{lemma}
\begin{proof}[\textbf{Proof of Theorem \ref{eq-local-continuity-of-solution}}]
By Lemma \ref{lem-Oscillation-Lemma}, a family of nest and shrinking cylinders $\left\{Q_n\right\}_{n=1}^{\infty}$, whose radius is $R_n$, and a decreasing sequence $\left\{\omega_n\right\}_{n=1}^{\infty}$ can be constructed recursively such that
\begin{equation*}
\frac{R_{n+1}}{R_n}<c, \qquad \forall n\in\N,
\end{equation*}
for some constant $0<c<1$ and
\begin{equation}\label{eq-iterally-decreasing-of-oscillation-for-continuity-of-solution}
\osc_{Q_n}u^i\leq\omega_n, \qquad \forall n\in\N,
\end{equation}
and 
\begin{equation*}
	\omega_{n+1}\leq \sigma\left(\omega_{n}\right)\omega_{n}, \qquad \left(0<\sigma\left(\omega_{n}\right)<1\right).
\end{equation*}
By an argument similar to the one presented in the proof of Theorem 2 in section 7 of \cite{Ur}, we have
\begin{equation*}
	\omega_{n}\to 0 \qquad \mbox{as $n\to\infty$}.
\end{equation*}
Therefore, the local continuity of $u^i$ holds and the theorem follows. 
\end{proof}
\begin{remark}\label{remark-last-in-continuity-some-case-just-conti-under-some-condition-up-to-holder-estimate}
Since the constant $\sigma_0$ in \eqref{eq-inequality-for-Oscillation-Lemma} depends on the ratio $\frac{\Lambda}{\theta^{\,\beta_i}}$ at each step of iteration, we can't find the modulus of continuity at this stage. We refer the reader to the paper  \cite{Ur} for the details on the local continuity of the partial differential equations.
\end{remark}

\begin{remark}\label{eq-remark-for-holder-estimates-for-function-u-i-s-in-solution-bold-u}
	Let the constants $\lambda_i$, $\beta_i$ be given by \eqref{eq-standard-form-for-local-continuity-estimates-intro-2}. If $U$ is equivalent to $\lambda_i\left(u^i\right)^{\beta_i}$ in $\R^n\times(0,\infty)$, i.e., there exist some constants $0<c\leq C<\infty$ such that
	\begin{equation*}
	c\lambda_i\left(u^i\right)^{\beta_i}\leq U\leq C\lambda_i\left(u^i\right)^{\beta_i} \qquad \mbox{in $\R^n\times(0,\infty)$},
	\end{equation*}
	then each components of the solution $\bold{u}$ is locally H\"older continuous in $\R^n\times\left(0,\infty\right)$. Moreover, all components have the same modulus of continuity. We refer the reader to the paper \cite{KL2}, \cite{KL3} for the details.
\end{remark}

 \section*{\bf{Acknowledgements}}
 Ki-Ahm Lee  was supported by Samsung Science and Technology Foundation under Project Number SSTF-BA1701-03. Ki-Ahm Lee also holds a joint appointment with Research Institute of Mathematics of Seoul National University. Sunghoon Kim was supported by the National Research Foundation of Korea(NRF) grant funded by the Korea government(MSIT) (No.  2020R1F1A1A01048334). Sunghoon Kim was also supported by the Research Fund, 2021 of The Catholic University of Korea.

\end{document}